\documentclass[11pt,twoside]{preprint}

\usepackage{hyperref}
\usepackage{breakurl}

\usepackage{times}
\usepackage{microtype}
\usepackage{amsmath}
\usepackage{amssymb}
\usepackage{mathtools}
\usepackage{tikz}
\usepackage{wasysym}
\usepackage{centernot}
\usepackage{color}
\usepackage{setspace}
\usepackage{appendix}
\usepackage{booktabs}
\usepackage{multirow}

\usepackage{mathrsfs}
\usepackage{mhequ}
\usepackage{mhenvs}
\usepackage{mhsymb}
\usepackage[top=3.5cm, bottom=3cm, left=2cm, right=2cm]{geometry}

\colorlet{symbols}{blue!90!black}
\def\symbol#1{\textcolor{symbols}{#1}}
\def\1{\mathbf{\symbol{1}}}

\def\emptyset{\mathop{\centernot\ocircle}}

\colorlet{testcolor}{green!60!black}

\definecolor{darkred}{rgb}{0.9,0.1,0.1}

\def\comment#1{\ifthenelse{\isodd{\value{page}}}{\marginpar{\raggedright\scriptsize{\textcolor{darkred}{#1}}}}{\marginpar{\raggedleft\scriptsize{\textcolor{darkred}{#1}}}}}  

\newop{diag}

\def\DD{\mathscr{D}}
\let\D\CD

\def\sR{\mathsf{R}}
\def\sC{\mathsf{C}}
\def\sM{\mathsf{M}}

\def\bG{{\mathbb{G}}}

\def\bR{{\mathbb R}}

\def\sf{{\mathsf f}}

\def\EE{{\mathscr E}}
\def\FF{{\mathscr F}}
\def\sE{{\mathscr E}}
\def\sF{{\mathscr F}}
\def\sA{{\mathscr A}}

\def\s{{\mathrm{s}}}

\def\L{{\mathcal{L}}}
\def\cE{\mathcal{E}}
\def\cF{\mathcal{F}}

\def\shortline{{\raisebox{0mm}{-}}}

\def\rDelta{\mathring{{\Delta}}}

\def\limsup{\mathop{\overline{\mathrm{lim}}}}
\def\liminf{\mathop{\underline{\mathrm{lim}}}}

\def\${|\!|\!|}
\def\l|{\left|\!\left|\!\left|}
\def\r|{\right|\!\right|\!\right|}

\begin{document}

\title{On a stiff problem in two-dimensional space}
\author{Liping Li$^{1,2, 4}$\thanks{The first named author is partially supported by NSFC (No. 11688101, No. 11801546 and No. 11931004), Key Laboratory of Random Complex Structures and Data Science, Academy of Mathematics and Systems Science, Chinese Academy of Sciences (No. 2008DP173182), and Alexander von Humboldt Foundation in Germany.}, Wenjie Sun$^{3,5}$}
\institute{RCSDS, HCMS, Academy of Mathematics and Systems Science, Chinese Academy of Sciences,  China.  
\and Department of Mathematics, Bielefeld University,  Germany. 
\and School of Mathematical Sciences, Tongji University, China.
\and 
\email{liliping@amss.ac.cn} \and 
\email{wjsun@tongji.edu.cn}}

\maketitle

\begin{abstract}
In this paper we will study a stiff problem in two-dimensional space and especially characterize its probabilistic counterpart.  Roughly speaking,  the heat equation with a parameter $\varepsilon>0$ is under consideration:
\[
\partial_t u^\varepsilon(t,x)=\frac{1}{2}\nabla \cdot \left(\mathbf{A}_\varepsilon(x)\nabla u^\varepsilon(t,x) \right),\quad t\geq 0, x\in \bR^2, 
\]
where $\mathbf{A}_\varepsilon(x)=\text{Id}_2$, the identity matrix, for $x\notin \Omega_\varepsilon:=\{x=(x_1,x_2)\in \bR^2: |x_2|<\varepsilon\}$ and $$\mathbf{A}_\varepsilon(x):=\begin{pmatrix} a_\varepsilon^{\raisebox{0mm}{-}} & 0 \\ 0 & a^\shortmid_\varepsilon  \end{pmatrix},\quad x\in \Omega_\varepsilon$$ with two constants $a^\shortline_\varepsilon, a^\shortmid_\varepsilon>0$.  There exists a diffusion process $X^\varepsilon$ on $\bR^2$ associated to this heat equation in the sense that $u^\varepsilon(t,x):=\mathbf{E}^xu^\varepsilon(0,X_t^\varepsilon)$ is its unique weak solution. Note that $\Omega_\varepsilon$ collapses to the $x_1$-axis,  a barrier of zero volume, as $\varepsilon\downarrow 0$. The main purpose of this paper is to figure out all possible limiting process $X$ of $X^\varepsilon$ as $\varepsilon\downarrow 0$.  In addition,  the limiting flux $u$ of $u^\varepsilon$ as $\varepsilon\downarrow 0 $ and all possible boundary conditions satisfied by $u$ at the barrier will be also obtained. 
\end{abstract}

\tableofcontents

\section{Introduction}
The stiff problem,  first raised in \cite{SanchezPalencia:1980dg}, is concerned with a thermal conduction model with a singular barrier of zero volume.  Generally speaking,  for $d\geq 1$,  let $\Omega\subset \bR^d$ be a domain and,  for any $\varepsilon>0$, $\Omega_\varepsilon\subset \Omega$  be another domain collapsing to a barrier of dimension $d-1$ as $\varepsilon\downarrow 0$.  Consider the following heat equation
\begin{equation}\label{eq:heat1}
	\partial_t u^\varepsilon(t,x)=\frac{1}{2}\nabla \cdot \left(\mathbf{A}_\varepsilon(x)\nabla u^\varepsilon(t,x) \right),\quad t\geq 0, x\in \Omega, 
\end{equation}  
where the conductivity matrix function $\mathbf{A}_\varepsilon(x)$ (conductivity in abbreviation) is nice and independent of $\varepsilon$ in $\Omega\setminus \bar\Omega_\varepsilon$, but probably singular near the boundary $\partial \Omega_\varepsilon$ or even in $\Omega_\varepsilon$.  The solution to \eqref{eq:heat1} is usually called a \emph{flux}.  Then the stiff problem focuses on the existence and related properties of the {\emph{limiting flux}, i.e. the limit $u$ of $u^\varepsilon$ as $\varepsilon\downarrow 0$, in a certain meaning.  

\begin{figure}
\centering
\includegraphics[scale=0.5]{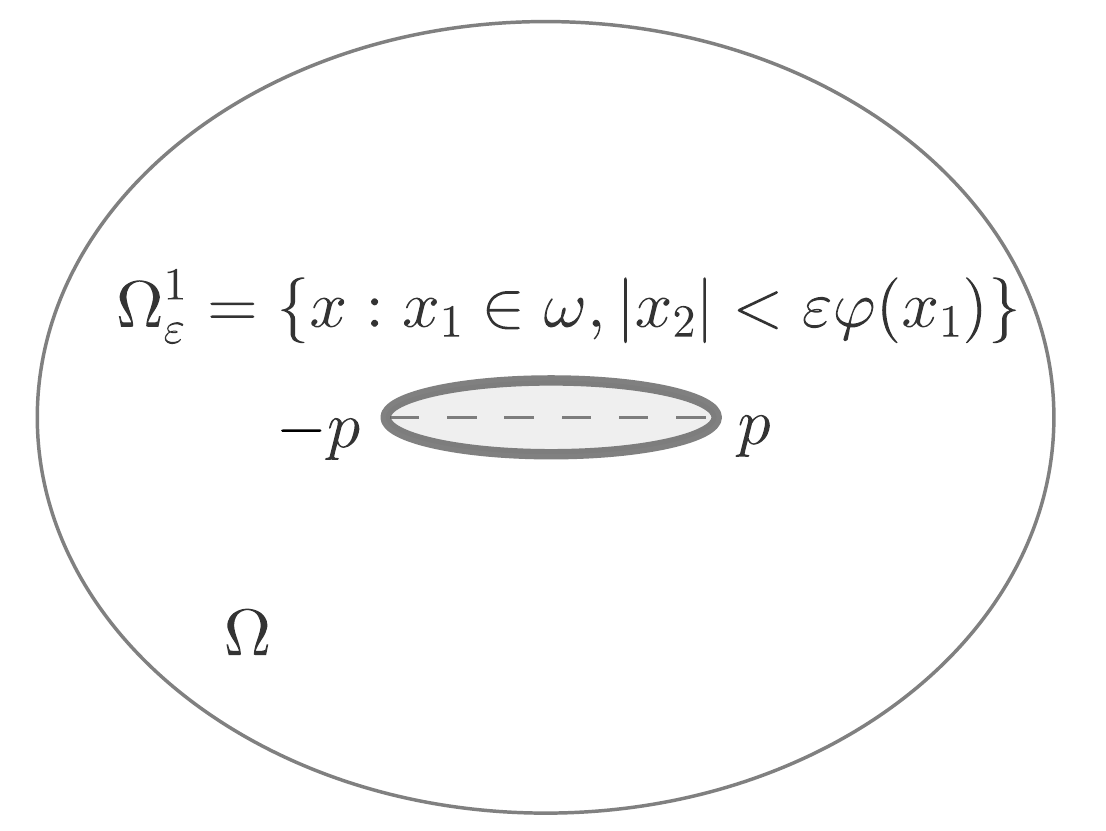}
\includegraphics[scale=0.52]{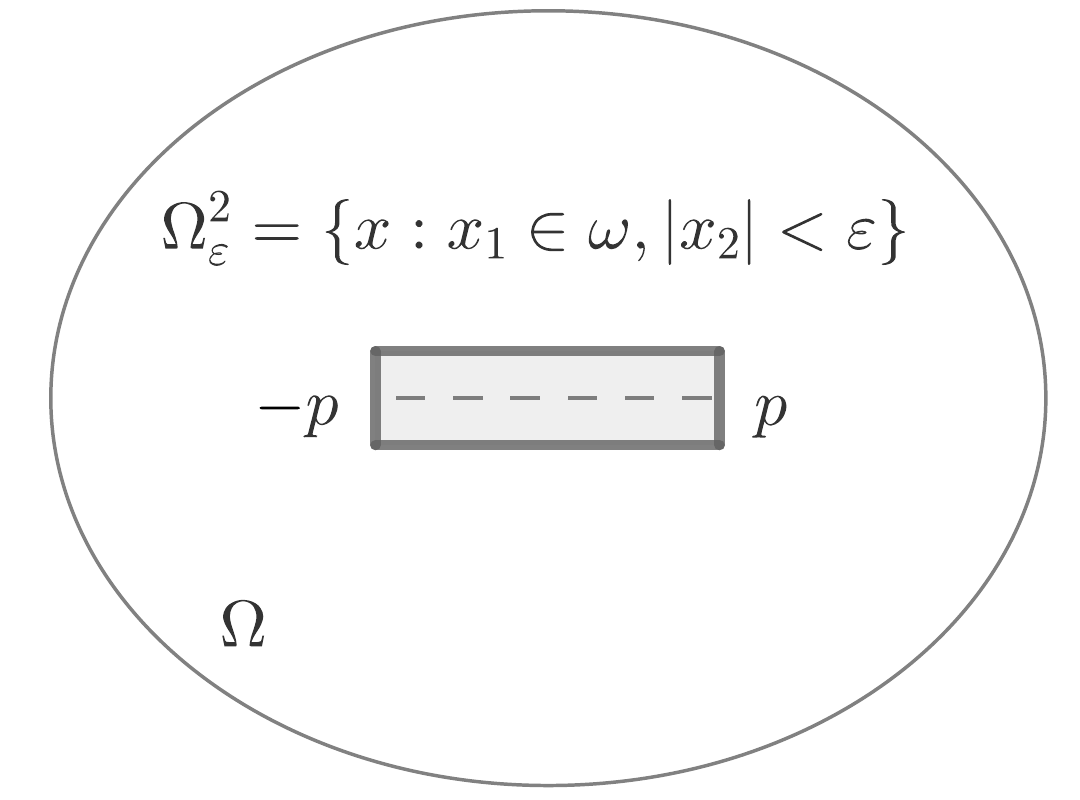}
\caption{$\Omega^1_\varepsilon$ and $\Omega^2_\varepsilon$}
\label{Figure1}
\end{figure}

In \cite{SanchezPalencia:1980dg},  two special cases in two-dimensional space are paid particular attention to.  As illustrated in Figure~\ref{Figure1},  the first one is the heat transfer through a narrow plate $\Omega^1_\varepsilon$ with a small conductivity,  where $\mathbf{A}_\varepsilon(x)$ is taken to be the identity matrix $\text{Id}_2$ outside $\Omega^1_\varepsilon$ and $\varepsilon\cdot \text{Id}_2$ inside $ \Omega_\varepsilon^1$,  and the second is the heat transfer through $\Omega^2_\varepsilon$ with a high conductivity,  where $\mathbf{A}_\varepsilon(x)$ is still taken to be $\text{Id}_2$ outside $\Omega^2_\varepsilon$  but $\frac{1}{\varepsilon}\cdot \text{Id}_2$ inside $\Omega^2_\varepsilon$ instead.  In both two cases, the limiting flux $u$ is shown to exist in an $L^2$-sense by an argument of functional analysis.   In a series work of Wang et al. (see \cite{CPW12} and the references therein),  a related problem  in two-dimensional space, where the inner barrier is replaced by an outer layer surrounding $\Omega$,  is studied by an argument of PDE.  It turns out that the limiting flux exists as a solution to the standard heat equation $\partial_t u=\frac{1}{2}\Delta u$ in $\Omega$, and meanwhile several \emph{effective boundary conditions} on $\partial \Omega$ satisfied by $u$ are obtained regarding various cases.

It is well known that the heat equation \eqref{eq:heat1} is usually associated to a diffusion process in probability,  i.e.  there exists a diffusion process $X^\varepsilon$ on $\Omega$ such that the unique weak solution to \eqref{eq:heat1} is 
\[
	u^\varepsilon(t,x)=\mathbf{E}^x u^\varepsilon(0,X^\varepsilon_t).  
\]
However, all literatures mentioned above give no insight into the limit of these diffusions.  
As far as we know, it is Lejay in \cite{L16} who first figures out the probabilistic counterpart of a stiff problem in one-dimensional space (also called a \emph{thin layer problem} in \cite{L16}),  where $\Omega=\bR, \Omega_\varepsilon=I_\varepsilon:=(-\varepsilon,\varepsilon)$ and $\mathbf{A}_\varepsilon(x):=a_\varepsilon(x)$ is taken to be constant $1$ outside $I_\varepsilon$ and $\kappa \varepsilon$ inside $I_\varepsilon$ for a given parameter $\kappa>0$.  Lejay shows that, as $\varepsilon\downarrow 0$, $X^\varepsilon$ converges to the so-called \emph{snapping out Brownian motion} (SNOB in abbreviation) in a sense.  Accordingly the limiting flux is not continuous at $0$ and in fact, the boundary condition 
\[
	\nabla u(t,0-)=\nabla u(t,0+)=\frac{\kappa}{2}\left(u(t,0+)-u(t,0-) \right)
\]
is satisfied.  Nevertheless, the derivation in \cite{L16} does not clarify the essential principle behind the convergence of $u^\varepsilon$ or $X^\varepsilon$.  In a recent paper \cite{LS19} general one-dimensional stiff problems with $\Omega=\bR, \Omega_\varepsilon=I_\varepsilon$ are studied by us.   We find that the \emph{thermal resistance},  rather than the conductivity,  of the barrier should play a central role in the convergence of $X^\varepsilon$,  and a phase transition in terms of the thermal resistance is manifested in this stiff problem.  

In this paper we will continue to study a stiff problem in two-dimensional space and especially characterize its probabilistic counterpart.  As illustrated in Figure~\ref{2-dimstiff}, we impose from now on
\[
	\Omega=\bR^2,\quad \Omega_\varepsilon=\{x=(x_1,x_2)\in \bR^2: x_1\in \bR, |x_2|<\varepsilon\},
\]
and $\mathbf{A}_\varepsilon(x)=\text{Id}_2$ for $x\in \Omega^c_\varepsilon$ while
\begin{equation}\label{eq:Avarepsilon}
	\mathbf{A}_\varepsilon(x)=\begin{pmatrix} a_\varepsilon^{\raisebox{0mm}{-}} & 0 \\ 0 & a^\shortmid_\varepsilon  \end{pmatrix},  \quad x\in \Omega_\varepsilon,
\end{equation}
where the \emph{tangent conductivity} $a_\varepsilon^{\raisebox{0mm}{-}}$ and the \emph{normal conductivity} $a^\shortmid_\varepsilon$ are two strictly positive constants depending on $\varepsilon$.  Here and hereafter the (outer) normal/tangent direction means that of $\partial \Omega^c_\varepsilon$,  i.e.  the direction along the $x_1$-/$x_2$-axis. 
 Note that $\Omega_\varepsilon$ collapses to the barrier,  i.e. the $x_1$-axis, as $\varepsilon\downarrow 0$.  Heuristically $\Omega_\varepsilon$ may be regarded as the limit of $\Omega^1_\varepsilon$ or $\Omega^2_\varepsilon$ as $p\rightarrow \infty$.  But unlike the two cases appearing in \cite{SanchezPalencia:1980dg},  $a^\shortmid_\varepsilon$ and $a^\shortline_\varepsilon$ are possibly different.  As will be explained in \S\ref{SEC41},  there exists a diffusion process $X^\varepsilon$ on $\bR^2$ associated to \eqref{eq:heat1}.  More precisely,  given an initial condition $u^\varepsilon(0,\cdot)=u_0\in H^1(\bR^2)$,  
\begin{equation}\label{eq:19}
	u^\varepsilon(t,x)=\mathbf{E}^x u_0(X^\varepsilon_t)
\end{equation}
is the unique weak solution to \eqref{eq:heat1}. Our main purpose is to explore the limit $X$ of $X^\varepsilon$ as well as the limit $u$ of $u^\varepsilon$ as $\varepsilon\downarrow 0$.  Since $\mathbf{A}_\varepsilon$ converges to the identity matrix outside the barrier,  we believe that $X$ should be equivalent to a two-dimensional Brownian motion before hitting the barrier and $u$ should be a weak solution to the standard heat equation outside the barrier.  
Hence the crucial problem is to characterize the behaviour of $X$ near the barrier and to obtain the boundary condition satisfied by $u$ at the barrier.  
The main result,  which will be stated in Theorem~\ref{THMAIN},  demonstrates that the desirable limits are due to the collaboration of two effects,  called \emph{normal resisting} and \emph{tangent accelerating}.  In what follows let us explain them respectively.  


\begin{figure}
\centering
\includegraphics[scale=0.55]{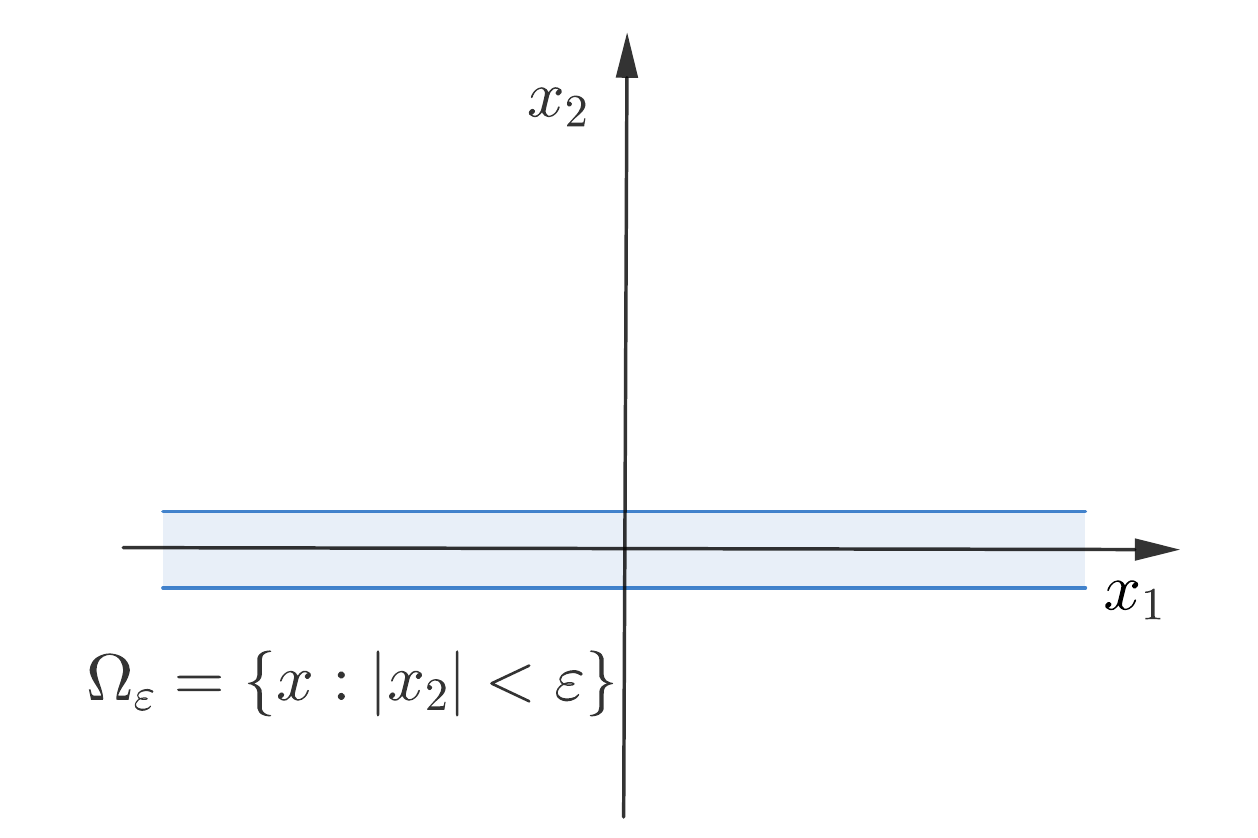}
\caption{A stiff problem in two-dimensional space}
\label{2-dimstiff}
\end{figure}

Roughly speaking,  the effect of normal resisting acts in resisting the heat flow to pass through the barrier along the normal direction.  To explain more details,  the idea for solving general one-dimensional stiff problems in \cite{LS19} will be first reviewed.  Regarding one-dimensional cases we have noted that the thermal resistance plays a central role.  To be precise,  when $\Omega_\varepsilon=I_\varepsilon$ and $\mathbf{A}_\varepsilon(x)=a_\varepsilon(x)$ (but $a_\varepsilon(x)$ is not necessarily assumed to be $\kappa \varepsilon$ inside $I_\varepsilon$),  the thermal resistance is quantified by the measure 
\[
	\lambda_\varepsilon(dx)=\frac{dx}{a_\varepsilon(x)},
\]  
which obviously measures the ability of the material to resist the heat transfer.  In addition,   $\bar{\gamma}:=\lim_{\varepsilon\downarrow \infty} \lambda_\varepsilon(I_\varepsilon)$, called the \emph{total thermal resistance} of the barrier, is assumed to exist in $[0,\infty]$.  Then a heuristic observation $\lambda_\varepsilon \rightarrow \lambda + \bar{\gamma}\cdot \delta_0$ as $\varepsilon\downarrow 0$,  where $\lambda$ is the Lebesgue measure and $\delta_0$ is the Dirac measure at $0$, illustrated in Figure~\ref{1-dimstiff}  sheds light on this stiff problem: (We should emphasize that the stage in \cite{LS19} is pretty wide,  and even $\lambda_\varepsilon$ is not necessarily assumed to be absolutely continuous.  But for convenience's sake only Brownian case is stated here.)
\begin{itemize}
\item[(1)] $\bar{\gamma}=0$: The barrier makes no sense and $X^\varepsilon$ converges to a one-dimensional Brownian motion. 
\item[(2)] $0<\bar{\gamma}<\infty$: The heat flow can penetrate the barrier only partially,  and $X^\varepsilon$ converges to the SNOB.  Lejay's thin layer problem, where $\bar{\gamma}=\frac{2}{\kappa}$, satisfies it.  The SNOB is a Feller process on $\bG:=(-\infty, 0-]\cup [0+,\infty)$, in which $0\in \bR$ corresponds to two distinct points.  It behaves like a reflecting Brownian motion on $\bG_-:=(-\infty, 0-]$ or $\bG_+:=[0+,\infty)$ while may change its sign and start as a new reflecting Brownian motion on the other component of $\bG$ by chance when it hits $0-$ or $0+$.  To be more exact,  sign changing is realized by jumps between $0+$ and $0-$.  Loosely speaking,  the SNOB consists of two components: Brownian motion outside $\{0\pm\}$ and \emph{snapping out jumps} inside $\{0\pm\}$.  
\item[(3)] $\bar{\gamma}=\infty$: The resisting is so strong that the heat flow cannot penetrate the barrier.  Meanwhile $X^\varepsilon$ converges to a reflecting Brownian motion on $\bG$,  namely a distinct union of two reflecting Brownian motions on $\bG_\pm$.    
\end{itemize}
Note that when $0<\bar{\gamma}\leq \infty$,  $0$ is the only discontinuous point of the limiting flux $u$;  see \cite[Theorem~5.1]{LS19}.  Hence splitting $\bR$ into two components $\bG_\pm$ makes it possible to build a ``diffusion process" associated to $u$.  Returning back to the two-dimensional case in Figure~\ref{2-dimstiff},  we may think of the collapsing of $\Omega_\varepsilon$ as the analogical collapsing of $I_\varepsilon$, where the endpoints $\pm \varepsilon$ are replaced by the horizontal lines at heights $\pm \varepsilon$.  Inspired by the argument in \cite{LS19},  it is expected to find out certain parameters measuring the normal resisting of $\Omega_\varepsilon$ as well as the barrier.  In practise, since $1/a^\shortmid_\varepsilon$ is the normal resistance ratio in $\Omega_\varepsilon$ and $\varepsilon$ measures the scale of $\Omega_\varepsilon$,  these parameters will be taken to be
\begin{equation}\label{eq:Rmid}
	\sR^\shortmid_\varepsilon:=\frac{\varepsilon}{a^\shortmid_\varepsilon}, 
\end{equation}
called the \emph{normal total resistance} of $\Omega_\varepsilon$, and the limit $\sR^\shortmid:=\lim_{\varepsilon\downarrow 0} \sR^\shortmid_\varepsilon$,  called the \emph{normal total resistance} (of the barrier), is assumed to exist in $[0,\infty]$ in Theorem~\ref{THMAIN}. 
When $0<\sR^\shortmid\leq \infty$,  we will see that the heat transfer can be effectively resisted by the barrier,  and like the stiff problems in \cite{LS19},  the state space $\bG^2=\bR\times \bG$ of the limiting process $X$ is obtained by splitting $\bR^2$ into two distinct components along the $x_1$-axis.  Particularly when $0<\sR^\shortmid<\infty$,  $X$ may enjoy the so-called \emph{snapping out jumps} taking place between the dual points $(x_1, 0\pm)$ in  $\partial \bG^2:=\bR\times \{0\pm\}$.  

\begin{figure}
\centering
\includegraphics[scale=0.55]{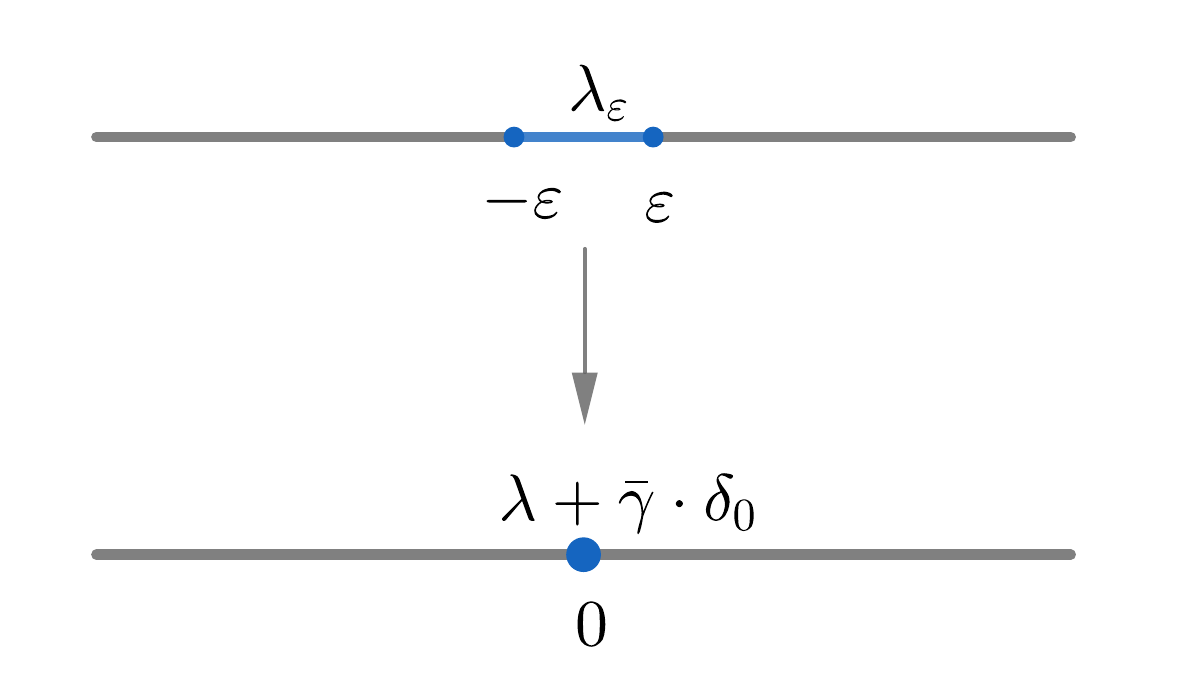}
\caption{Stiff problems in one-dimensional space}
\label{1-dimstiff}
\end{figure}


Another effect of tangent accelerating is based on the heuristic observation as follows.  Since $\Omega_\varepsilon$ is very narrow,  the normal movements of $X^\varepsilon$ in $\Omega_\varepsilon$ can be neglected,  while $X^\varepsilon$ may speed up along the tangent direction in $\Omega_\varepsilon$ when the tangent conductivity is very high.  In other words,  it moves more rapidly along the tangent direction than a Brownian motion, when passes through $\Omega_\varepsilon$.  To quantify this accelerating effect,  we set
\[
	\sC^\shortline_\varepsilon:=\varepsilon a^\shortline_\varepsilon,
\]
called the \emph{tangent total conductivity} of $\Omega_\varepsilon$, and the limit $\sC^\shortline:=\lim_{\varepsilon\downarrow 0} \sC^\shortline_\varepsilon$, called the \emph{tangent total conductivity} (of the barrier), is assumed to exist in $[0,\infty]$.  It will be shown that when $\sC^\shortline >0$,  the limiting process $X$ also enjoys tangent acceleration upon hitting the barrier.  Particularly when $\sC^\shortline=\infty$,  $X$ speeds up along the tangent direction upon hitting the barrier so rapidly,  that it may impossibly leave the barrier.  In other words,  $X$ becomes an absorbing Brownian motion on $\bR^2_0:=\{(x_1,x_2)\in \bR^2: x_2\neq 0\}$.



With these two effects at hand,  it is necessary to ask whether and how they work in collaboration to yield the limiting process.  
To obtain the final conclusion,  we further introduce the following parameters:
\[
	\sM_\varepsilon:= \sqrt{\frac{\sC^\shortline_\varepsilon}{\sR^\shortmid_\varepsilon}},
\]
called the \emph{mixing scale} at $\varepsilon$, and the limit $\sM:=\lim_{\varepsilon\downarrow 0}\sM_\varepsilon$, called the \emph{mixing scale}, is assumed to exist in $[0,\infty]$ as well.  Then in Theorem~\ref{THMAIN} we eventually conclude that
\begin{itemize}
\item[(1)] When $\sM=0$,  normal resisting plays a central role,  while tangent accelerating makes no sense.  Furthermore, the limiting process $X$ exists in a sense and manifests a \emph{normal phase transition}:
\begin{description}
\item[(N1)] $\sR^\shortmid=0$: The barrier makes no sense,  and $X$ is  a two-dimensional Brownian motion.
\item[(N2)] $0<\sR^\shortmid<\infty$: This is analogical to the case $0<\bar{\gamma}<\infty$ in \cite{LS19}.  The state space of $X$ is $\bG^2$,  and $X$ enjoys so-called \emph{snapping out jumps} on $\partial \bG^2$.  In other words,  whenever hitting a boundary point $(x_1, 0\pm)$, $X$ has a chance to jump to its dual point $(x_1,0\mp)$. 
\item[(N3)] $\sR^\shortmid=\infty$: This is analogical to the case $\bar{\gamma}=\infty$ in \cite{LS19},  and $X$ consists of two distinct components on $\bG^2_\pm$ respectively, each of which is a reflecting Brownian motion. 
\end{description}
\item[(2)] When $\sM=\infty$,  tangent accelerating plays a central role,  while  normal resisting makes no sense.  The limiting process $X$ exists and manifests a \emph{tangent phase transition}:
\begin{description}
\item[(T1)] $\sC^\shortline=0$:  Clearly this case is the same as \textbf{(N1)}.
\item[(T2)] $0<\sC^\shortline<\infty$: Due to the absence of normal resisting, $X$ is still a diffusion process on $\bR^2$.  Meanwhile $X$ enjoys tangent accelerating upon hitting the barrier,  which is quantified by $2\sC^\shortline L_t$,  as illustrated in \eqref{representationoftype4}.    Here $L_t$ is the local time of $X$ at the $x_1$-axis.  
\item[(T3)] $\sC^\shortline=\infty$: $X$ is the absorbing Brownian motion on $\bR^2_0$. 
\end{description}
\item[(3)] The most interesting case is $0<\sM<\infty$.  Unless $\sR^\shortmid=\sC^\shortline=0$,  both normal resisting and the tangent accelerating make sense.  In fact,  $X$ exists and manifests a \emph{mixing phase transition}:
\begin{description}
\item[(M1)] $\sC^\shortline, \sR^\shortmid=0$: This trivial case is the same as \textbf{(N1)}.
\item[(M2)] $0<\sC^\shortline, \sR^\shortmid<\infty$: Under the effects of normal resisting and tangent accelerating,  $X$ stays on $\bG^2$ and enjoys two kinds of jump on $\partial \bG^2$: The first kind, called \emph{self-interacting jump},  takes place on $\partial \bG^2_+:=\bR\times \{0+\}$ or $\partial \bG^2_-:=\bR\times \{0-\}$ separately; while the second kind, called \emph{interacting jump}, takes place between $\partial \bG^2_+$ and $\partial \bG^2_-$.  At a heuristic level,  $X$ can be thought of as a mixture of those limiting processes in \textbf{(N2)} and \textbf{(T2)}; see Remark~\ref{RM33}. 
\item[(M3)] $\sC^\shortline, \sR^\shortmid=\infty$: $X$ consists of two distinct components on $\bG^2_\pm$ respectively and enjoys only self-interacting jumps on $\partial \bG^2_\pm$  mentioned in \textbf{(M2)}.
\end{description}
\end{itemize}
Particularly,  the first case illustrated in Figure~\ref{Figure1} is an analogue with a bounded barrier of \textbf{(N2)},  since in this case $\sC^\shortline_\varepsilon=\varepsilon^2, \sR^\shortmid_\varepsilon=1$ and $\sM_\varepsilon=\varepsilon$, and meanwhile the second case illustrated in Figure~\ref{Figure1} is an analogue with a bounded barrier of \textbf{(T2)}, since currently $\sC^\shortline_\varepsilon=1, \sR^\shortmid_\varepsilon=\varepsilon^2$ and $\sM_\varepsilon=1/\varepsilon$. 

\begin{table}
\centering
\begin{tabular}{cccc}
\toprule  
Phase  &   Boundary condition  & 1st continuity & 2nd continuity \\
\midrule  
\textbf{(N1)},  \textbf{(T1)},  \textbf{(M1)}   & (B.I) & Y & Y	\\
\midrule
\textbf{(N2)} &  (B.II) & N & Y \\
\textbf{(N3)} &  (B.III) & N & Y \\
\midrule
\textbf{(T2)} &  (B.IV) & Y & N\\
\textbf{(T3)} &  (B.V) & Y & N	\\
\midrule
\textbf{(M2)} &  (B.VI) & N & N \\ 
\textbf{(M3)} & (B.VII) & N & N \\
\bottomrule 
\end{tabular}
\caption{1st and 2nd continuities for limiting flux}
\label{table2}
\end{table}

As a byproduct of Theorem~\ref{THMAIN}, we will obtain in \S\ref{SEC42} that the limiting flux $u$ exists as a weak solution to the standard heat equation outside the barrier.  For all the cases above,  $u\in H^1(\bG^2)$ and various boundary conditions at the barrier are satisfied by $u$; see Corollary~\ref{COR44}.  Although these boundary conditions are not stated here,  we will provide alternative evidence for the phase transitions in Theorem~\ref{THMAIN} by means of them.  The limiting flux $u$ is called to satisfy the \emph{first kind of continuity} (at the barrier) if 
\[
	\gamma_+ u_+=\gamma_- u_-,
\]	
and the \emph{second kind of continuity} (at the barrier) if 
\[
	\left.\frac{\partial u_{+}}{\partial x_2}\right|_{x_2=0+}=\left.\frac{\partial u_{-}}{\partial x_2}\right|_{x_2=0-},
\]
where $u_\pm:=u|_{\bG^2_\pm}$,  $\gamma_\pm u_\pm$ is the trace of $u_\pm$ on $\partial \bG^2_\pm$,  and $\left.\frac{\partial u_{\pm}}{\partial x_2}\right|_{x_2=0\pm}$ is the trace of the normal derivative of $u_\pm$ on $\partial \bG^2_\pm$.   In Table~\ref{table2},  we summarize the results concerning these two continuities, where ``Y" means the first or second kind of continuity holds and ``N" means it does not hold.  Obviously,  in spite of the trivial phases, different phase transitions (i.e.  normal/tangent/mixing ones) lead to different continuities.  More significantly,  normal resisting breaks the first kind of continuity, while tangent accelerating breaks the second kind of continuity. 

The approach to prove the main results in this paper is by virtue of the theory of Dirichlet forms.  A Dirichlet form is a symmetric Markovian bilinear form on an $L^2(E,m)$ space, where $E$ is a nice topological space and $m$ is a fully supported positive Radon measure on it.  The theory of Dirichlet forms is closely related to probability theory due to a series of works by Fukushima, Silverstein in 1970's, and Ma and R\"ockner in 1990's etc.  It is now well known that a so-called \emph{regular} or \emph{quasi-regular} Dirichlet form is always associated with a symmetric Markov process. We refer the notions and terminologies in the theory of Dirichlet forms to \cite{CF12, FOT11}.  In particular, it is easy to figure out that the diffusion process $X^\varepsilon$ in \eqref{eq:19} is associated with the following regular Dirichlet form on $L^2(\bR^2)$:
\begin{equation}\label{defofEEVare}
\begin{aligned}
  \FF^\varepsilon &=H^1(\mathbb{R}^2)\\
 \EE^\varepsilon (u,u) &=
 \frac{1}{2}\int_{\mathbb{R}^2\setminus\Omega_\varepsilon} |\nabla u|^2dx+\dfrac{a^{\raisebox{0mm}{-}}_\varepsilon}{2}\int_{\Omega_\varepsilon} |\partial_{x_1} u|^2dx+\dfrac{a^{\shortmid}_\varepsilon}{2}\int_{\Omega_\varepsilon} |\partial_{x_2} u|^2dx,\quad u\in \FF^\varepsilon.
 \end{aligned}
\end{equation}
To show the convergence of $X^\varepsilon$, more exactly, the convergence of finite dimensional distributions of $X^\varepsilon$, we will employ the so-called \emph{Mosco convergence} of Dirichlet forms.  The terminologies and results concerning this concept will be reviewed in Appendix \ref{SECA1} for readers' convenience. 

Although this paper concentrates on the stiff problem illustrated in Figure~\ref{2-dimstiff},  the argument based on the effects of normal resisting and tangent accelerating should be still helpful in studying general stiff problems even in high dimensional space.  For a small domain $\Omega_\varepsilon\subset \bR^d$ collapsing to the barrier $\omega$ of dimension $d-1$,  we may think of the collapsing direction of $\Omega_\varepsilon$ as an analogue of the normal direction and the hyperplane that contains $\omega$ as an analogue of the tangent direction.  By introducing sensible parameters measuring the analogical normal resisting and tangent accelerating,  it is possible to formulate the limiting process as well as the limiting flux for a general stiff problem.  

The rest of this paper is organized as follows.  In \S\ref{SEC2},  we will characterize all possible limiting processes by means of their associated Dirichlet forms.  Following \cite{LS19}, we also say that each limit gives a phase related to the stiff problem.  A brief summarization of all the phases will be presented in Table~\ref{table1}.  The main result, Theorem~\ref{THMAIN}, will be stated and proved in \S\ref{SEC3}.  Furthermore,  in \S\ref{SEC34} we will show that the normal/tangent/mixing phase transition appearing in Theorem~\ref{THMAIN} is continuous in the critical parameter $\sR^\shortmid$ or $\sC^\shortline$.  Finally in \S\ref{SEC4} we reconsider this stiff problem in terms of heat equations. Particularly the existence of the limiting flux $u$ is shown and related boundary conditions at the barrier are derived.  


\subsection*{Notations}

We prepare notations that will be frequently used for handy reference.  Let $\mathbb{G}:=(-\infty,0-]\cup[0+,\infty)$, where $0$ in $\mathbb{R}$ corresponds to either $0+$ or $0-$ viewed as two distinct points. In other words, $\mathbb{G}$ is composed of two connected components, say $\mathbb{G}_+:=(-\infty,0-]$ and $\mathbb{G}_-:=[0+,\infty)$. Write
$\bG^2:=\bR\times \bG$ and $\bG^2_\pm:=\bR\times \bG_\pm$.  
Similarly, denote 
\[\mathbb{R}^2_+:=\{(x_1,x_2): x_1\in\mathbb{R}, x_2\in(0, \infty)\}, \quad \mathbb{R}^2_-:=\{(x_1,x_2): x_1\in\mathbb{R},  x_2\in (-\infty,0)\},\]
and $\Omega_\varepsilon^+:=\mathbb{R}^2_+\cap \Omega_\varepsilon$, $\Omega_\varepsilon^-:=\mathbb{R}^2_-\cap \Omega_\varepsilon$. The closure of $\bR^2_\pm$ is denoted by $\bar{\bR}^2_\pm$. 

Denote $H=L^2(\mathbb{R}^2)=L^2(\mathbb{G}^2)=L^2(\mathbb{R}^2_0)$, where $\mathbb{R}^2_0:=\mathbb{R}^2_+\cup\mathbb{R}^2_-$.  For an open set $D\subset \bR^d$ with $d=1$ or $2$,  $C_c^\infty(D)$ stands for the family of all smooth functions with compact support in $D$ and
\[
	C_c^\infty(\bar{D}):=\{f|_{\bar{D}}: f\in C_c^\infty(\bR^d)\}. 
\] 
Let $H^1(D):=\{u\in L^2(D): \nabla u\in L^2(D)\}$ be the Sobolev space of first order on $D$, where $\nabla u$ is defined in the sense of Schwartz distribution, and the closure of $C_c^\infty(D)$ in $H^1(D)$ is denoted by $H^1_0(D)$. Set
\[
	H^1_\Delta(D):=\{u\in H^1(D): \Delta u\in L^2(D)\},
\]
where $\Delta u$ is also in the sense of Schwartz distribution.  
Set $H^1(\bG^2_\pm):=H^1(\bR^2_\pm)$ and $H^1_\Delta(\bG^2_\pm):=H^1_\Delta(\bR^2_\pm)$.  Let
\[
	H^1(\bG^2):=\{u\in L^2(\bG^2): u|_{\bG^2_\pm}\in H^1(\bG^2_\pm)\},\quad H^1_\Delta(\bG^2):=\{u\in L^2(\bG^2): u|_{\bG^2_\pm}\in H^1_\Delta(\bG^2_\pm)\}.
\]
Then by zero extension,  we have $H^1(\bG^2_\pm)\subset H^1(\bG^2)$ and $H^1_\Delta(\bG^2_\pm)\subset H^1_\Delta(\bG^2)$.  
By the embedding map $\iota$ with $\iota(u)|_{\bG^2_\pm}:=u|_{\bar{\bR}^2_\pm}$ for a function $u$ defined on $\bR^2$,  one may also write	$H^1(\bR^2)\subset H^1(\bG^2)$ and $H^1_\Delta(\bR^2)\subset H^1_\Delta(\bG^2)$.

Generally speaking, the operator $\Delta$ (or $\nabla$) is defined on the space $\mathscr{D}'(U)$ of all Schwartz distributions on an open set $U$.  To be more exact, we should write $\Delta_U$ (or $\nabla_U$) in place of $\Delta$ (or $\nabla$).  For example,  $\Delta_U: \mathscr{D}'(U)\rightarrow \mathscr{D}'(U)$ and for $T\in \mathscr{D}'(U)$,  
\[
	\left\langle \Delta_U T,\varphi\right\rangle:= \left\langle T, \Delta \varphi \right\rangle,\quad \forall \varphi\in C_c^\infty(U),
\]
where $\Delta \varphi$ is the normal Laplacian.  When $U=\mathbb{R}^2_0$,  we write
\begin{equation}\label{eq:rDelta}
	\mathring{\Delta}:=\Delta_{\bR^2_0}
\end{equation}
for convenience.  Otherwise if no confusions cause,  the superscript $U$ will be omitted. 

For a function $u$ defined on $\bG^2$ (resp. $\bR^2$),  $u_\pm:=u|_{\bG^2_\pm}$ (resp.  $u_\pm:=u|_{\bR^2_\pm}$) denote the restrictions of $u$ to $\bG^2_\pm$ (resp. $\bR^2$). 
As noted in Appendix~\ref{AP1},  for a function $u\in H^1(\bG^2_\pm)$, the trace of $u$ on the boundary,  i.e.  $\bR\times \{0\pm\}$, is denoted by $u(\cdot\pm)$ or $\gamma_\pm u$. For $u\in H^1(\bR^2)$, the trace of $u$ on the $x_1$-axis is denoted by $u|_\bR$ or $\gamma u$.  In addition, for $u\in H^1_\Delta(\bG^2_\pm)$,  $\gamma_\pm^{\partial_2} u$ or $\left.\frac{\partial u}{\partial x_2}\right|_{x_2=0\pm}$ stands for the trace of the normal derivative of $u$ on $\partial \bG^2_\pm$.  In abuse of notations,  for $u\in H^1(\bG^2)$ (resp. $u\in H^1_\Delta(\bG^2)$),  we also use $u(\cdot\pm)$ or $\gamma_\pm u$  (resp.  $\gamma_\pm^{\partial_2} u$ or $\left.\frac{\partial u}{\partial x_2}\right|_{x_2=0\pm}$) to stand for $u_\pm(\cdot\pm)$ or $\gamma_\pm u_\pm$  (resp.  $\gamma_\pm^{\partial_2} u_\pm$ or $\left.\frac{\partial u_\pm}{\partial x_2}\right|_{x_2=0\pm}$).

The symbol $\lesssim$ (resp. $\gtrsim$) means that the left (resp. right) term is bounded by the right (resp. left) term multiplying a non-essential constant.

\section{All possible limiting phases}\label{SEC2}

In this section, we will introduce seven different phases related to the two-dimensional stiff problem in Figure~\ref{2-dimstiff}. These phases will be described by means of Dirichlet forms and associated Markov processes. 

Given a Dirichlet form $(\sE,\sF)$ on $H$,  let $\mathcal{L}$ be its generator on $H$ with the domain $\mathcal{D}(\L)$.  Note that $u\in \D(\L)$, $f=\L u$ if and only if $u\in\FF$ and $\EE(u,v)=(-f, v)_H$ for any $v\in\FF$ due to \cite[Corollary 1.3.1]{FOT11}.  This fact will be used to formulate the generators appearing in this section.  Recall that $\rDelta$ is the Laplacian operator defined on all Schwartz distributions on $\bR^2_0$, see \eqref{eq:rDelta}.  We will see that  all these generators are the restrictions of $\rDelta$ to different subspaces of $H^1_\Delta(\bG^2)$. 


\subsection{Phase of type I}

The phase of type I is given by the Dirichlet form on $L^2(\mathbb{R}^2)$:
\begin{equation}\label{eq:type3}
\begin{aligned}
  \FF^\mathrm{I}&=H^1(\mathbb{R}^2),\\
 \EE^{\mathrm{I}} (u,u)&=
 \frac{1}{2}\int_{\mathbb{\mathbb{R}}^2} |\nabla u|^2dx,\quad u\in \FF^\mathrm{I},
 \end{aligned}
\end{equation}
which corresponds to the Brownian motion $X^\mathrm{I}$ on $\mathbb{R}^2$. Clearly, the generator of $X^\mathrm{I}$ on $L^2(\bR^2)$ is $\L^\mathrm{I}=\frac{1}{2}\Delta$ with the domain
{\begin{equation}\label{eq:type3-generator}
\D\left(\L^\mathrm{I}\right)=\left\{u\in\FF^\mathrm{I}: \Delta u\in L^2\left(\mathbb{R}^2\right)\right\}=:H^1_\Delta(\bR^2).
\end{equation}}
The lemma below gives an alternative expression of $\mathcal{L}^\mathrm{I}$. 

\begin{lemma}
The generator $\L^\mathrm{I}$  is equal to $\frac{1}{2}\rDelta$ restricted to 
\[
	\tilde{\D}:=\left\{u\in H^1_\Delta(\bG^2): \gamma_+u_+=\gamma_-u_-,  \left.\frac{\partial u_{+}}{\partial x_2}\right|_{x_2=0+}=\left.\frac{\partial u_{-}}{\partial x_2}\right|_{x_2=0-} \right\}.
\]
In other words,  $H^1_\Delta=\tilde{\D}$ and for any $u\in H^1_\Delta(\bR^2)$, it holds that $\L^\mathrm{I}u=\frac{1}{2}\rDelta u$. 
\end{lemma}
\begin{proof}
Take $u\in H^1_\Delta (\bR^2)$,  and it follows from Lemma~\ref{LMB1} that $u\in H^1(\bG^2)$ and $\gamma_+u_+=\gamma_-u_-$.  In addition,  since $\Delta u\in L^2(\bR^2)$,  it holds that $\Delta u_\pm=(\Delta u)_\pm$ and the Green-Gauss formula \eqref{eq:Green2} implies that for any $\varphi\in C_c^\infty(\bR^2)$, 
\[
\left\langle	\left.\frac{\partial u_{+}}{\partial x_2}\right|_{x_2=0+}-\left.\frac{\partial u_{-}}{\partial x_2}\right|_{x_2=0-}, \varphi(\cdot, 0)\right\rangle =-(\Delta u, \varphi)_H-(\nabla u, \nabla \varphi)_H=0. 
\]
This indicates $u\in \tilde{\D}$.  Therefore $H^1_\Delta(\bR^2)\subset \tilde{D}$ and for any $u\in H^1_\Delta (\bR^2)$,  it holds $\rDelta u=\Delta u$.  

To the contrary, it suffices to show $\tilde{\D}\subset H^1_\Delta(\bR^2)$.   To do this, take $u\in \tilde{\D}$, and Lemma~\ref{LMB1} yields that $u\in H^1(\bR^2)$.  By applying \eqref{eq:Green2} and using the last equality in the definition of $\tilde{\D}$, we have for any $\varphi\in C_c^\infty(\bR^2)$, 
\begin{equation}\label{eq:23}
	\int_{\bR^2}\nabla u \cdot \nabla \varphi dx = -\int_{\bR^2_+}\Delta u_+ \cdot \varphi_+dx -\int_{\bR^2_-}\Delta u_-\cdot  \varphi_- dx.
\end{equation}
Note that the right hand side of \eqref{eq:23} is equal to $-\int_{\bR^2} \rDelta u \cdot  \varphi dx$,  and $\rDelta u\in L^2(\bG^2)=L^2(\bR^2)$.  Therefore we can conclude that $\Delta u=\rDelta u\in L^2(\bR^2)$ by the definition of the weak divergence $\Delta u=\nabla \cdot (\nabla u)$ of $\nabla u\in L^2(\bR^2)$.  That completes the proof.
\end{proof}

\subsection{Phase of type II}

The phase of  type II (related to a constant $\kappa>0$) is given by the quadratic form on $L^2(\bG^2)$: 
\begin{equation}\label{eq:type2}
\begin{aligned}
 \FF^\mathrm{II}&=H^1(\mathbb{G}^2),\\
 \EE^\mathrm{II} (u,u)&=
 \frac{1}{2}\int_{\mathbb{G}^2} |\nabla u|^2dx+ \frac{\kappa}{4}\int_{\mathbb{R}} \left(u(x_1+)-u(x_1-)\right)^2dx_1,\quad u\in \FF^\mathrm{II},
 \end{aligned}
\end{equation}
where $u(\cdot \pm):=\gamma_\pm u_\pm$ are the traces of $u_\pm:=u|_{\bG^2_\pm}\in H^1(\mathbb{G}^2_\pm)$ on the boundary (see Appendix~\ref{AP1}). 
The following lemma obtains the regularity of \eqref{eq:type2} and the expression of its associated generator.  Hereafter, we will denote its associated Markov process by $X^\mathrm{II}$. 


\begin{lemma}\label{LEM21}
\begin{itemize}
\item[(1)] The quadratic form $\left(\mathscr{E}^\mathrm{II}, \sF^\mathrm{II}\right)$ is a regular Dirichlet form  on $L^2(\bG^2)$.  
\item[(2)] The generator $\L^\mathrm{II}$ of $\left(\mathscr{E}^\mathrm{II}, \sF^\mathrm{II}\right)$ on $L^2(\bG^2)$ is $\frac{1}{2}\rDelta u$ restricted to the domain
{\begin{equation}\label{eq:type2-generator}
\D\left(\L^\mathrm{II}\right)=\left\{u\in H^1_\Delta(\bG^2):\left.\frac{\partial u_{\pm}}{\partial x_2}\right|_{x_2=0\pm}=\frac{\kappa}{2}\left(u(x_1+)-u(x_1-)\right)\right\}.
\end{equation}}
\end{itemize}
\end{lemma}
\begin{proof}
\begin{itemize}
\item[(1)] From \eqref{EQ2TRH2} we can conclude that  
\[
\int_{\mathbb{R}}u(x_1\pm)^2dx_1\leq C\|u\|^2_{H^1({\mathbb{G}^2})}, 
\]
so that
\[\EE^\mathrm{III}_1(u,u)\leq \EE^\mathrm{II}_1(u,u)\lesssim\EE^\mathrm{III}_1(u,u), \]
{where $\left(\sE^\mathrm{III},\sF^\mathrm{III}\right)$ is the Dirichlet form of reflecting Brownian motion on $\mathbb{G}^2$, as we shall introduce in Section 2.3.} Thus the norm 
 $\|\cdot \|_{\EE^\mathrm{III}_1}$ is equivalent to $\|\cdot\|_{\EE^\mathrm{II}_1}$,  and the regularity of  $\left(\sE^\mathrm{III},\sF^\mathrm{III}\right)$ implies that $(\EE^\mathrm{II},\FF^{\mathrm{II}})$ is also a regular Dirichlet form on $L^2(\mathbb{G}^2)$. 



\item[(2)] Denote the right hand side of \eqref{eq:type2-generator} by $\mathcal{G}$.   By means of the Green-Gauss formula in Lemma~\ref{LMB3},   it is straightforward to verify that for any $u\in \mathcal{G}$,  
\[
	u\in \sF^\mathrm{II},\quad \sE^\mathrm{II}(u,v)=(-\frac{1}{2}\rDelta u, v)_H,\quad \forall v\in \sF^\mathrm{II}. 
\]
This yields
\[
	\mathcal{G}\subset\D(\L^\mathrm{II}),\quad \L^\mathrm{II}u=\frac{1}{2}\rDelta u,\quad \forall u\in \mathcal{G}.
\]
To show $\mathcal{D}(\L^\mathrm{II})\subset \mathcal{G}$, suppose that $u\in \D(\L^\mathrm{II})\subset H^1(\bG^2)$ with $f:=\L^\mathrm{II} u$. Then for any $v\in C_c^\infty(\bR^2_\pm$),  we have
\[
-(f,v)_{H}=\EE^\mathrm{II} (u,v)=\frac{1}{2}\int_{\mathbb{G}^2}\nabla u\cdot\nabla v dx=-\frac{1}{2}\left(\Delta u_{\pm},v_{\pm}\right)_{L^2(\mathbb{G}^2_{\pm})}.\]
Since $f_{\pm}\in L^2(\mathbb{G}^2_{\pm})$, it follows that 
\[
	\Delta u_{\pm}=2f_\pm=2\left(\L^\mathrm{II}u\right)|_\pm\in L^2(\mathbb{G}^2_{\pm}).
\]
Hence $u\in H^1_\Delta(\bG^2)$.  Now we take $v\in H^1(\mathbb{G}^2)$ instead.  Then it holds
\begin{equation}\label{EQ22L1}
(-f,v)_H=\frac{1}{2}\int_{\mathbb{G}^2}\nabla u\cdot\nabla v dx+\frac{\kappa}{4}\int_{\mathbb{R}} \left(u(x_1+)-u(x_1-)\right)\left(v(x_1+)-v(x_1-)\right)dx_1.
\end{equation}
By applying the Green-Gauss formula \eqref{eq:GreenGaussformula} to \eqref{EQ22L1},  it follows that
\[
\begin{split}
\frac{\kappa}{2}\int_{\mathbb{R}} \left(u(x_1+)-\right.
& \left. u(x_1-)\right)\left(v(x_1+)-v(x_1-)\right)dx_1\\
&=\left\langle\left.\frac{\partial u_{+}}{\partial x_2}\right|_{x_2=0+}, \gamma_+v\right\rangle-
\left\langle\left.\frac{\partial u_{-}}{\partial x_2}\right|_{x_2=0-}, \gamma_-v \right\rangle.
\end{split}
\]
Eventually by letting $v_-\equiv0 $ or $v_+\equiv0 $, we can conclude that 
\[\left.\frac{\partial u_{\pm}}{\partial x_2}\right|_{x_2=0\pm}=\frac{\kappa}{2}\left(u(x_1+)-u(x_1-)\right),\]
which, together with $u\in H^1_\Delta(\bG^2)$, implies $\mathcal{D}(\L^\mathrm{II})\subset \mathcal{G}$. 
\end{itemize}
That completes the proof. 
\end{proof}

Now we have a position to present a probabilistic representation of $X^\mathrm{II}$ by the SNOB $\beta^\s$ on $\bG$ (with a parameter $\kappa$).  Note that $\beta^\s$ is irreducible and associated with the regular Dirichlet form on $L^2(\bG)$ (see \cite{LS19}): 
\begin{equation*}
\begin{aligned}
 \mathcal{F}^\s&=H^1(\bG), \\
 \cE^\s(u,u)&=
 \frac{1}{2}\int_\mathbb{G} \left(u'(x)\right)^2dx+\dfrac{\kappa}{4}(u(0+)-u(0-))^2,\quad u\in \cF^{\s}.
 \end{aligned}
 \end{equation*}
The following lemma states that $X^\mathrm{II}$ is indeed the independent coupling of one-dimensional Brownian motion and $\beta^\s$. 

\begin{lemma}
$X^\mathrm{II}$ is equivalent to the process $(\beta, \beta^\s)$ on $\bG^2$, where $\beta=(\beta_t)_{t\geq 0}$ is a one-dimensional Brownian motion and $\beta^\s$ is the SNOB with the parameter $\kappa$ independent of $\beta$. 
Particularly, $X^\mathrm{II}$ is irreducible. 
\end{lemma}
\begin{proof}
Let ${X}^\s=(\beta, \beta^\s)$, and $\left({\EE}^\s,{\FF}^\s\right)$ be the Dirichlet form associated with ${X}^\s$. We are to derive $\left({\EE}^\s,{\FF}^\s\right)$ using the direct product of Drichlet forms. 
Note that the Dirichlet form of one-dimensional Brownian motion is $(\frac{1}{2}\mathbf{D}, H^1(\mathbb{R}))$, where $\mathbf{D}(u,u)=\int_{\mathbb{R}}\left(u'(x)\right)^2dx$,  and the Dirichlet form of $\beta^\s$ is $\left(\cE^\s,\cF^\s\right)$. 
Since $\beta$ and $\beta^\s$ are independent, $\left({\EE}^\s,{\FF}^\s\right)$ is the direct product of $\left(\frac{1}{2}\mathbf{D}, H^1(\mathbb{R})\right)$ and $\left(\cE^\s, \cF^\s\right)$. It follows from \cite[Proposition 3.2]{LY15} or \cite[Theorem 1.4]{O89} that 
\[
\begin{split}
{\FF}^\s=\{u\in L^2(\mathbb{G}^2): \
& u(\cdot, x_2)\in H^1(\mathbb{R}) \text{ for a.e. } x_2,  x_2\mapsto \|u(\cdot, x_2)\|_{\mathbf{D}_1}\in L^2(\mathbb{G}) \text{ and }\\
& u(x_1,\cdot)\in H^1(\mathbb{G}) \text{ for a.e. } x_1,  x_1\mapsto \|u(x_1,\cdot)\|_{\cE^{\s}_1}\in L^2(\mathbb{R})
\}\end{split}
\]
and
\[
	\sE^\s(u,u)=\frac{1}{2} \int_{\mathbb{G}}\mathbf{D}\left(u(\cdot, x_2), u(\cdot, x_2)\right)dx_2+\int_{\mathbb{R}}\cE^{\s}\left(u(x_1,\cdot), u(x_1, \cdot)\right)dx_1,\quad u\in \sF^\s. 
\]
From \cite[Theorem 1.1 and Theorem 1.2, \S 1.1.3]{M11} we know that
${\FF}^\s=H^1(\mathbb{G}^2)$ and for $u\in \sF^\s$, 
\[
\begin{split}
{\EE}^\s(u,u)
=\frac{1}{2}\int_{\mathbb{G}^2} |\nabla u|^2dx+ \frac{\kappa}{4}\int_{\mathbb{R}} \left(u(x_1+)-u(x_1-)\right)^2dx_1.
\end{split}
\]
These yield $\left({\EE}^\s,{\FF}^\s\right)=\left(\EE^\mathrm{II},\FF^{\mathrm{II}}\right)$.  Particularly since both  $\beta$ and $\beta^{\mathrm{s}}$ are irreducible, \cite[Theorem 2.6]{O89} implies that $X^{\mathrm{II}}$ is also irreducible.  That completes the proof. 
\end{proof}


\subsection{Phase of type III}

The phase of type III is given by the associated Dirichlet form of reflecting Brownian motion $X^\mathrm{III}$ on $\bG^2$, i.e. 
\begin{equation}\label{eq:p1}
\begin{aligned}
  \FF^\mathrm{III}&=H^1(\mathbb{G}^2),\\
 \EE^\mathrm{III} (u,u)&=
\int_{\mathbb{G}^2} |\nabla u|^2dx,\quad u\in \FF^\mathrm{III},
 \end{aligned}
\end{equation}
Note that $X^\mathrm{III}$ is the union of two distinct reflecting Brownian motions on $\mathbb{G}^2_\pm$.  
The following lemma states basic properties of $\left(\sE^\mathrm{III},\sF^\mathrm{III}\right)$.  The proof is trivial and we omit it.  

\begin{lemma}
The Dirichlet form $\left(\EE^\mathrm{III}, \FF^\mathrm{III}\right)$ is regular but not irreducible on $L^2(\mathbb{G}^2)$. Furthermore, the generator $\L^\mathrm{III}$ of $X^\mathrm{III}$ on $L^2(\bG^2)$ is  $\frac{1}{2}\rDelta u$ restricted to the domain
\begin{equation}\label{eq:type1-generator}
\D\left(\L^\mathrm{III}\right)=\left\{u\in H^1_\Delta(\bG^2): \left.\frac{\partial u_{\pm}}{\partial x_2}\right|_{x_2=0\pm}=0\right\}.
\end{equation}
\end{lemma}

\subsection{Phase of type IV}

The phase of type IV is given by the closure of  the quadratic form
\begin{equation}\label{eq:type4}
\begin{aligned}
  \D[\EE^\mathrm{IV}] &=C_c^\infty(\mathbb{R}^2),\\
 \EE^\mathrm{IV} (u,u)&=
 \frac{1}{2}\int_{\mathbb{\mathbb{R}}^2} |\nabla u|^2dx+\lambda\int_\mathbb{R}\left(\frac{\partial u(x_1,0)}{\partial x_1}\right)^2dx_1,\quad u\in C_c^\infty(\mathbb{R}^2),
 \end{aligned}
\end{equation}
where $\lambda>0$ is a constant,  on $L^2(\bR^2)$. The closability of $(\EE^\mathrm{IV},\DD[\EE^\mathrm{IV}])$ is due to \cite[Theorem 3.1.4]{FOT11}, and we denote its closure by $(\EE^\mathrm{IV}, \FF^\mathrm{IV})$.  The following lemma obtains the expression of $(\sE^\mathrm{IV}, \sF^\mathrm{IV})$ and characterizes its associated Markov process $X^\mathrm{IV}$.  Note that $u|_\bR$ stands for the trace of $u$ on the $x_1$-axis; see Appendix~\ref{AP1}.

\begin{lemma}\label{LM24}
The Dirichlet form $(\EE^\mathrm{IV}, \FF^\mathrm{IV})$ on $L^2(\bR^2)$ is expressed as
\[
\begin{aligned}
 \FF^\mathrm{IV}&=\left\{u\in H^1(\mathbb{R}^2): u|_{\mathbb{R}}\in H^1(\mathbb{R})\right\},\\
 \EE^\mathrm{IV}(u,u)&=
 \frac{1}{2}\int_{\mathbb{\mathbb{R}}^2} |\nabla u|^2dx+{\lambda}\int_\mathbb{R}\left(\frac{d \left(u|_\mathbb{R}\right)}{d x_1}\right)^2dx_1,\quad u\in \FF^\mathrm{IV},
 \end{aligned}
\]
which is regular, strongly local and irreducible and recurrent. Furthermore, the associated Markov process $X^\mathrm{IV}$ enjoys the following representation: 
\begin{equation}\label{representationoftype4}
X^\mathrm{IV}_t-X^\mathrm{IV}_0=(\beta^1_{t+2\lambda L_t},  \beta^2_t),\quad t\geq 0,
\end{equation}
where $B_t:=(\beta^1_t,\beta^2_t)$ is a certain standard two-dimensional Brownian motion and $L_t$ is the local time of $X^\mathrm{IV}$ at the $x_1$-axis,  i.e.  the positive continuous additive functional (PCAF in abbreviation) corresponding to the smooth measure $dx_1\delta_0(dx_2)$ with respect to $X^\mathrm{IV}$.    
\end{lemma}
\begin{proof}
We first show $(\sE^\mathrm{IV},\sF^\mathrm{IV})$ is a strongly local,  and irreducible Dirichlet form.  Clearly it is a symmetric bilinear quadratic form. To show its closeness, take an $\sE^\mathrm{IV}_1$-Cauchy sequence $\{u_n:n\geq 1\}$.  Then $u_n\in H^1(\bR^2)=\sF^\mathrm{I}$ and $u_n$ is $\sE^\mathrm{I}_1$-Cauchy. Thus there exists $u\in H^1(\bR^2)$ such that $u_n$ converges to $u$ under the $\sE^\mathrm{I}_1$-norm.  Since \eqref{traceineq} holds for every $u\in H^1(\bR^2)$,  it follows that $u_n|_\bR$ converges to $u|_\bR$ under the $L^2(\bR)$-norm.  On the other hand, note that $u_n$ is $\sE^\mathrm{IV}_1$-Cauchy and 
\[
	\int_\mathbb{R}\left(\frac{d \left(u_n|_\mathbb{R}-u_m|_\bR\right)}{d x_1}\right)^2dx_1\lesssim \sE^\mathrm{IV}(u_n-u_m,u_n-u_m).
\]
Hence we can obtain that $u_n|_\bR\in H^1(\bR)$ is $\mathbf{D}_1$-Cauchy and $u_n|_\bR$ converges to $u|_\bR\in H^1(\bR)$ under the $\mathbf{D}_1$-norm.  Therefore $u\in \sF^\mathrm{IV}$ and $u_n$ converges to $u$ under the $\sE^\mathrm{IV}_1$-norm.  In addition, the strong locality of $(\EE^\mathrm{IV}, \FF^\mathrm{IV})$ is obvious, and its irreducibility is implied by that of $(\sE^\mathrm{I},\sF^\mathrm{I})$ and \cite[Corollary~4.6.4]{FOT11}.  

To prove the recurrence,  it suffices to find out a sequence $\{u_n:n\geq 1\}\subset \sF^\mathrm{IV}$ such that $u_n\rightarrow 1$,  a.e. and $\lim_{n\rightarrow \infty}\sE^\mathrm{IV}(u_n,u_n)=0$.  To do this,  take $\varphi_n\in C_c^1([0,\infty))$ such that 
\[
\begin{aligned}
	&\varphi_n(r)=1\text{ for }0\leq r\leq n;\quad  \varphi_n(r)=0\text{ for } r\geq 2n+1; \\
	&|\varphi'_n(r)|\leq 1/n  \text{ and } 0\leq \varphi(r)\leq 1\text{ for all }r\geq 0;
\end{aligned}
\]
and a function $s\in C^1([0,\infty))$:  $s(r):=r$ for $0\leq r\leq 1$ and $s(r):=\log r+1$ for $r>1$.  For $n\geq 1$, define $u_n(x):=\varphi_n(s(|x|))$ for any $x\in \bR^2$.  We assert that $\{u_n:n\geq 1\}$ is the desirable sequence.  It is easy to verify that $u_n\in C_c^1(\bR^2)\subset \sF^\mathrm{IV}$ and $u_n\rightarrow 1$ pointwisely.  A computation yields 
\[
	\sE^\mathrm{IV}(u_n,u_n)=\pi \int_0^\infty \left|(\varphi_n\circ s)'(r)\right|^2rdr + 2\lambda \int_0^\infty \left|(\varphi_n\circ s)'(r)\right|^2 dr.  
\]
Hence we only need to show $\int_0^1\left|(\varphi_n\circ s)'(r)\right|^2 dr,  \int_1^\infty\left|(\varphi_n\circ s)'(r)\right|^2 rdr\rightarrow 0$,  which imply $\sE^\mathrm{IV}(u_n,u_n)\rightarrow 0$.  To this end,  note that for any $n\geq 1$, 
\[
	\int_0^1\left|(\varphi_n\circ s)'(r)\right|^2 dr=\int_0^1 |\varphi'_n(r)|^2dr =0,
\]
and we use the substitution $y:=\log r+1$ to find that
\[
\int_1^\infty\left|(\varphi_n\circ s)'(r)\right|^2 rdr=\int_1^\infty \left|\varphi'_n(\log r+1)\right|^2\frac{dr}{r}=\int_1^\infty \left|\varphi'_n(y)\right|^2dy\leq \int_n^{2n+1} \left(\frac{1}{n}\right)^2dy\rightarrow 0.  
\]
Eventually the recurrence of $(\sE^\mathrm{IV}, \sF^\mathrm{IV})$ is verified. 

The regularity will be proved by several steps as follows.   Firstly,  we assert that the family  $\sF^\mathrm{IV}_c$ of all functions with compact support in $\sF^\mathrm{IV}$ is $\sE^\mathrm{IV}_1$-dense in $\sF^\mathrm{IV}$.  Indeed,  take a function $\varphi\in C_c^\infty(\bR)$ such that $\varphi(r)\equiv 1$ for $|r|\leq 1$,  $\varphi(r)\equiv 0$ for $|r|>2$ and $0\leq \varphi\leq 1$.  Then for any $u\in \sF^\mathrm{IV}$,  define $u_n(x):=u(x)\varphi(x_1/n)\varphi(x_2/n)$ for $x=(x_1,x_2)\in \bR^2$.  It is easy to verify that $u_n\in \sF^\mathrm{IV}_c$ and $u_n$ converges to $u$ under the $\sE^\mathrm{IV}_1$-norm.  Secondly,  for a fixed $u\in \sF^\mathrm{IV}_c$,  define a function for any $\delta>0$
\begin{equation}\label{eq:211}
u_\delta(x_1,x_2)=\begin{cases}
u(x_1,x_2-\delta), \quad & x_2>\delta,\\
u|_{\bR}(x_1), \quad & |x_2|\leq \delta, \\
u(x_1,x_2+\delta), \quad & x_2<-\delta. 
\end{cases}
\end{equation}
We assert that $u_\delta\in \sF^\mathrm{IV}_c$ and $u_\delta$ converges to $u$ under the $\sE^\mathrm{IV}_1$-norm as $\delta\downarrow 0$.  The fact $u_\delta\in H^1(\bR^2)$ can be deduced by mimicking the proof of Lemma~\ref{LMB1},  and in addition $u_\delta|_{\bR}=u|_\bR\in H^1(\bR)$.  Hence $u_\delta\in \sF^\mathrm{IV}_c$ holds.  To show $\EE^\mathrm{IV}_1 (u_\delta-u,u_\delta-u)\rightarrow 0$,  note that
\begin{equation}\label{EQ3UGCON}
\begin{split}
\| u_\delta- u\|^2_{L^2(\mathbb{R}^2)}
& \lesssim \int_\mathbb{R}\int_{\delta}^\infty\left[ u(x_1,x_2-\delta)- u(x_1,x_2)\right]^2dx_2dx_1\\
&\qquad +\int_\mathbb{R}\int_{-\delta}^{\delta}\left[ u_\delta(x_1,x_2)- u(x_1,x_2)\right]^2dx_2dx_1\\
&\qquad +\int_\mathbb{R}\int^{-\delta}_{-\infty}\left[ u(x_1,x_2+\delta)- u(x_1,x_2)\right]^2dx_2dx_1\\
&=:I_1+I_2+I_3.
\end{split}
\end{equation}
The terms $I_1$ and $I_3$ must tend to $0$ due to the equicontinuity of the integration (see, e.g., \cite[Theorem 2.32]{AF03}).  We also have $I_2\rightarrow 0$ because
\[
\int_{\Omega_\delta}| u_\delta|^2dx=2\delta\int_\mathbb{R}u|_{\bR}(x_1)^2dx_1\rightarrow 0,
\]
Hence $\| u_\delta- u\|^2_{L^2(\mathbb{R}^2)}\rightarrow 0$.  Analogically we can obtain that $\|\nabla u_\delta-\nabla u\|^2_{L^2(\mathbb{R}^2)}\rightarrow 0$.  Since $u_\delta|_\bR=u|_\bR$,  it eventually follows that $\EE^\mathrm{IV}_1 (u_\delta-u,u_\delta-u)\rightarrow 0$.  Thirdly,  we show that $u_\delta$ in \eqref{eq:211} can be approximated by a sequence of functions in $C_c^\infty(\bR^2)$ under the $\sE^\mathrm{IV}_1$-norm.  To accomplish this,  take an even function $J\in C_c^\infty(\bR)$ such that $\text{supp}[J]\subset [-1,1]$  and $\int_\bR J(r)dr=1$,  and for any $\varepsilon>0$ set $J_\varepsilon(\cdot):=\varepsilon^{-1}J(\cdot/\varepsilon)$.  Define
\[
	v_\varepsilon(x_1,x_2):=\int_{\bR^2} J_\varepsilon(x_1-x'_1)J_\varepsilon(x_2-x'_2)u_\delta(x'_1,x'_2)dx'_1dx'_2,\quad x=(x_1,x_2)\in \bR^2.  
\]
Since $u_\delta \in H^1(\bR^2)$ has compact support,  we have $v_\varepsilon\in C_c^\infty(\bR^2)$ and $v_\varepsilon$ converges to $u_\delta$ in $H^1(\bR^2)$ as $\varepsilon\downarrow 0$.  Note that when $\varepsilon<\delta$, 
\[
	h_\varepsilon(x_1):=v_\varepsilon(x_1, 0)=\int_{\bR^2} J_\varepsilon(x_1-x'_1)J_\varepsilon(-x'_2)u|_\bR(x'_1)dx'_1dx'_2=\int_\bR J_\varepsilon(x_1-x'_1)u|_\bR(x'_1)dx'_1.  
\]
This implies $\mathbf{D}(h_\varepsilon-u|_\bR, h_\varepsilon-u|_\bR)\rightarrow 0$  as $\varepsilon\downarrow 0$.  Therefore $v_\varepsilon$ converges to $u_\delta$ under the $\sE^\mathrm{IV}_1$-norm as $\varepsilon\downarrow 0$.  Finally,  we can eventually conclude that $C_c^\infty(\bR^2)$ is $\sE^\mathrm{IV}_1$-dense in $\sF^\mathrm{IV}$ by a standard argument. 

Now we turn to formulate the representation \eqref{representationoftype4}.  The basic tool  is the so-called Fukushima's decomposition;  see,  e.g. \cite[Chapter 5]{FOT11}.  Denote
$\phi(x):=(\phi_1(x), \phi_2(x))=(x_1,x_2)$ for any $x=(x_1, x_2)\in \bR^2$.  
Clearly $\phi_i\in \FF^\mathrm{IV}_{loc}$,  the family of all functions locally in $\sF^\mathrm{IV}$.  Note that the recurrence of $(\sE^\mathrm{IV},\sF^\mathrm{IV})$ implies that $X^\mathrm{IV}$ is conservative.  By means of \cite[Theorem 5.5.1]{FOT11}, we can write the Fukushima decomposition of $\phi_i(X^\mathrm{IV}_t)$ for $i=1,2$  and any $t\geq 0$ as
\[
	\phi_i(X^\mathrm{IV}_t)-\phi_i(X^\mathrm{IV}_0)=M_t^{[\phi_i]}+N_t^{[\phi_i]},
\]
where $M_t^{[\phi_i]}$ is a continuous martingale additive functional locally of finite energy, and $N_t^{[\phi_i]}$ is a continuous additive functional locally of  zero energy.  Firstly we assert that 
\begin{equation}\label{eq:2131}
N^{[\phi_1]}_t=N^{[\phi_2]}_t=0,\quad t\geq 0.
\end{equation}
Indeed,  a straightforward computation yields $\EE^\mathrm{IV}(\phi_i, v)=0=0$ for any $v\in C_c^\infty(\bR^2)$.  Then \eqref{eq:2131} follows from \cite[Corollary~5.5.1]{FOT11}.  Secondly,  we figure out the expression of $M_t^{[\phi_i]}$ by computing the energy measures $\mu_{\langle \phi_i\rangle}$ and $\mu_{\langle \phi_1,\phi_2\rangle}$.  Note that \cite[Theorem~5.5.2]{FOT11} implies
\[\int_{\mathbb{R}^2}fd\mu_{\langle\phi_i\rangle}=2\EE^\mathrm{IV}(\phi_if,\phi_i)-\EE^\mathrm{IV}(\phi_i^2,f), \quad f\in C^\infty_c(\mathbb{R}^2).\]
The right hand side is equal to
\[\sum_{j=1}^2\int_{\mathbb{R}^2}\delta_i^j fdx_1dx_2+2\lambda\int_\mathbb{R} f|_\mathbb{R}\delta_i^1dx_1, \quad i=1,2,\]
where $\delta_i^j$ is the Kronecker delta. It follows that 
\begin{equation}\label{eq:2141}
\mu_{\langle\phi_1\rangle}(dx)=dx_1dx_2+2\lambda dx_1\delta_0(dx_2) \quad \mu_{\langle\phi_2\rangle}(dx)=dx_1dx_2,
\end{equation} 
where $\delta_0$ is the Dirac measure at $0$.  Similarly we can also obtain that
\[
\mu_{\langle \phi_1+\phi_2\rangle}=2dx_1dx_2+2\lambda dx_1\delta_0(dx_2).
\]
Then it follows from the polarization identity that
\begin{equation}\label{eq:2151}
\mu_{\langle\phi_1,\phi_2\rangle}=\frac{1}{2}\left(\mu_{\langle\phi_1+\phi_2\rangle}-\mu_{\langle\phi_1\rangle}-\mu_{\langle\phi_2\rangle}\right)=0.
\end{equation}
Hence \eqref{eq:2141} and \eqref{eq:2151} yield
\[
\langle M^{[\phi_1]}\rangle_t=t+2\lambda L_t, \quad \langle M^{[\phi_2]}\rangle_t=t, \quad \langle M^{[\phi_1]}, M^{[\phi_2]}\rangle_t=0,
\]
where $L_t$ is the local time of $X^\mathrm{IV}$, i.e. the PCAF corresponding to the Revuz measure $dx_1\delta_0(dx_2)$.  Using a martingale representation theorem such as \cite[Chapter V,  Theorem~1.9]{RY99},  we eventually arrive at \eqref{eq:211}.  That completes the proof.
\end{proof}
\begin{remark}
We should point out that $L_0=0$ and $t\mapsto L_t$ is a non-decreasing process.   Particularly,  $L_t$ increases at time $t$ only when $X^\mathrm{IV}_t$ hits the $x_1$-axis.  
\end{remark}

The following lemma gives the generator of $X^\mathrm{IV}$.  Note that $u\in H^1_\Delta(\bG^2)$ and $\gamma_+ u_+=\gamma_-u_-$ imply $u\in H^1(\bR^2)$, and thus $u|_\bR=\gamma_+ u_+$ is well defined due to Lemma~\ref{LMB1}. 

\begin{lemma}
The generator $\L^\mathrm{IV}$ of $X^\mathrm{IV}$ on $L^2(\bR^2)$ is $\frac{1}{2}\rDelta u$ restricted to the domain
\begin{equation}\label{eq:type4-generator}
{\D(\L^\mathrm{IV})=\left\{u\in H^1_\Delta(\bG^2): \gamma_+ u_+=\gamma_-u_-,  \left.\frac{\partial u_+}{\partial x_2}\right|_{x_2=0+}-\left.\frac{\partial u_-}{\partial x_2}\right|_{x_2=0-}=-2\lambda\frac{d^2 (u|_\mathbb{R})}{d x_1^2} \right\}.}
\end{equation}
\end{lemma}
\begin{proof}
Denote the right hand side of \eqref{eq:type4-generator} by $\mathcal{G}$.  We first prove $\mathcal{G}\subset \mathcal{D}(\L^\mathrm{IV})$ and $(\L^\mathrm{IV} u)_\pm=\frac{1}{2}\Delta u_\pm$ for any $u\in \mathcal{G}$.  Fix $u\in \mathcal{G}$ and it suffices to show that 
\begin{equation}\label{eq:213}
	u\in \sF^\mathrm{IV},\quad \sE^\mathrm{IV}(u,v)=-\frac{1}{2}\int_{\bR^2_+} \Delta u_+  \cdot v_+ dx-\frac{1}{2}\int_{\bR^2_-} \Delta u_-  \cdot v_- dx,\quad \forall v\in C_c^\infty(\bR^2),
\end{equation}
since $C_c^\infty(\bR^2)$ is $\sE^\mathrm{IV}_1$-dense in $\sF^\mathrm{IV}$ due to Lemma~\ref{LM24}.  Indeed,  $u\in H^1(\bG^2)$ and $\gamma_+u_+=\gamma_-u_-$ imply that $u\in H^1(\bR^2)$ and $u|_{\bR}= \gamma_+u_+=\gamma_-u_-\in H^{1/2}(\bR)$ by Lemma~\ref{LMB1}.  Set $g:=u|_\bR\in H^{1/2}(\bR)$ for convenience. Since $\left.\frac{\partial u_\pm}{\partial x_2}\right|_{x_2=0\pm}\in H^{-1/2}(\bR)$ due to Lemma~\ref{LEMA2},  the last equality in \eqref{eq:type4-generator} yields $g''\in H^{-1/2}(\bR)$, i.e. 
\begin{equation}\label{eq:218}
\int_\bR \left(1+|\xi|^2\right)^{-1/2} |\widehat{g''}(\xi)|^2d\xi=	\int_\bR \left(1+|\xi|^2\right)^{-1/2}|\xi|^4 |\widehat{g}(\xi)|^2d\xi <\infty, 
\end{equation}
where $\widehat{g}$ is the Fourier transform of $g$.  Note that $\int_\bR |\widehat{g}(\xi)|^2d\xi<\infty$ due to $g\in H^{1/2}(\bR)\subset L^2(\bR)$.  Together with \eqref{eq:218},  we can obtain that
\[
\begin{aligned}
\int_\bR \left(1+|\xi|^2\right)^{3/2} |\widehat{g}(\xi)|^2d\xi&\leq 2\int_\bR \left(1+|\xi|^2\right)^{-1/2}(1+|\xi|^4) |\widehat{g}(\xi)|^2d\xi \\
&\leq 2\int_\bR |\widehat{g}(\xi)|^2d\xi+ 2\int_\bR \left(1+|\xi|^2\right)^{-1/2}|\xi|^4 |\widehat{g}(\xi)|^2d\xi <\infty. 
\end{aligned}\]
In other words,  $u|_\bR=g\in H^{3/2}(\bR)\subset H^1(\bR)$.  Particularly $u\in \sF^\mathrm{IV}$.  Furthermore,  it follows from the Green-Gauss formula \eqref{eq:Green2} that 
\[
-\left(\nabla u_\pm, \nabla v_\pm \right)_{L^2\left(\mathbb{R}^2_\pm\right)}=\left(\Delta u_\pm,  v_\pm\right)_{L^2\left(\mathbb{R}^2_\pm\right)}\pm\left\langle\left.\frac{\partial u_\pm}{\partial x_2}\right|_{x_2=0\pm},  v(\cdot, 0)\right\rangle.
\]
This yields 
\begin{equation}\label{eq:214}
\int_{\bR^2_0}\nabla u \cdot \nabla vdx=-\int_{\bR^2_0} \Delta u\cdot vdx-\left\langle\left.\frac{\partial u_+}{\partial x_2}\right|_{x_2=0+}-\left.\frac{\partial u_-}{\partial x_2}\right|_{x_2=0-},  v(\cdot, 0)\right\rangle.
\end{equation}
Note that 
\begin{equation}\label{eq:215}
	-\left\langle \frac{d^2(u|_\bR)}{dx_1^2}, v(\cdot, 0)\right\rangle=\int_\bR \frac{d u|_\bR}{dx_1}\frac{d v(\cdot, 0)}{dx_1}dx_1,
\end{equation}
since $\frac{d^2(u|_\bR)}{dx_1^2}=\frac{d}{dx_1} \left(\frac{d u|_\bR}{dx_1}\right)$ is the weak derivative of $\frac{d u|_\bR}{dx_1}\in L^2(\bR)$ and $v(\cdot, 0)\in C_c^\infty(\bR)$.  Eventually it follows from \eqref{eq:214}, \eqref{eq:215} and the definition of $\mathcal{G}$ that 
\[
	\sE^\mathrm{IV}(u,v)=-\frac{1}{2}\int_{\bR^2_0} \Delta u\cdot v dx. 
\]
As a result, \eqref{eq:213} is proved. 

To the contrary,  we need to prove $\mathcal{D}(\L^\mathrm{IV})\subset \mathcal{G}$.  To accomplish this, take $u\in \D(\L^\mathrm{IV})$ with $f:=\L^\mathrm{IV} u \in L^2(\bR^2)$.  Since $\mathcal{D}(\L^\mathrm{IV})\subset \sF^\mathrm{IV}$, it follows from Lemma~\ref{LMB1} that $u\in H^1(\bG^2)$ and $\gamma_+u_+=\gamma_-u_-=u|_\bR\in H^1(\bR)$.  Hence it suffices to show that $\rDelta u\in L^2(\bR^2_0)$ and the last equality on the right hand side of \eqref{eq:type4-generator} holds.  
To do this, take $v\in C_c^\infty(\bR^2)$ with $\text{supp}[v]\subset \bR^2_\pm$ and it holds that
\[(-f_\pm,v)_{L^2(\bR^2_\pm)}=\EE^\mathrm{IV}(u,v)=\frac{1}{2}\int_{\mathbb{R}^2_\pm}\nabla u_\pm \cdot\nabla v dx. \]
Using the definition of $\Delta u_\pm$, we have $\Delta u_\pm=2f_\pm \in L^2(\bR^2_\pm)$. Hence $\rDelta u\in L^2(\bR^2_0)$ and $u\in H^1_\Delta (\bG^2)$. 
On the other hand,  taking $v(x):=v_1(x_1)v_2(x_2)$ with $v_1,v_2\in C_c^\infty(\bR)$ and $v_2(0)\neq 0$ instead,  it follows from the Green-Gauss formula \eqref{eq:GreenGaussformula} that 
\[
-\left(\nabla u_\pm, \nabla v_\pm\right)_{L^2\left(\mathbb{R}^2_\pm\right)}=\left(\Delta u_\pm,  v_\pm\right)_{L^2\left(\mathbb{R}^2_\pm\right)}\pm v_2(0)\cdot \left\langle\left.\frac{\partial u_\pm}{\partial x_2}\right|_{x_2=0\pm},  v_1\right\rangle.
\]
Since $(\rDelta u, v)_{L^2(\bR^2_0)}=(2f,v)_{L^2(\mathbb{R}^2)}=2\EE^\mathrm{IV}(u,v)$,  we can obtain that for all $v_1\in C_c^\infty(\bR)$, 
\[\left\langle\left.\frac{\partial u_+}{\partial x_2}\right|_{x_2=0+}-\left.\frac{\partial u_-}{\partial x_2}\right|_{x_2=0-},  v_1\right\rangle=2\lambda \left(\frac{d (u|_\mathbb{R})}{d x_1}, \frac{d v_1}{d x_1}\right)_{L^2(\mathbb{R})}=-2\lambda \left\langle \frac{d^2 (u|_\mathbb{R})}{d x_1^2}, v_1  \right\rangle,\]
where the second equality is due to the definition of Schwartz distribution $\frac{d^2(u|_\bR)}{dx_1^2}=\frac{d}{dx_1} \left(\frac{d u|_\bR}{dx_1}\right)$.  Therefore  
 \[-2\lambda\frac{d^2 (u|_\mathbb{R})}{d x_1^2}=\left.\frac{\partial u_+}{\partial x_2}\right|_{x_2=0+}-\left.\frac{\partial u_-}{\partial x_2}\right|_{x_2=0-}\in H^{-1/2}(\bR). \]
 That completes the proof. 
\end{proof}

\subsection{Phase of type V}

The phase of type V  is given by the Dirichlet form on $L^2\left(\mathbb{R}^2_0\right)$:
\begin{equation}\label{eq:type5}
\begin{aligned}
 & \FF^\mathrm{V}=\left\{u\in L^2\left(\mathbb{R}^2_0\right): u_\pm\in H^1_0\left(\mathbb{R}^2_\pm\right) \right\},\\
 &\EE^\mathrm{V} (u,u)=
 \frac{1}{2}\int_{\mathbb{\mathbb{R}}^2_0} |\nabla u|^2dx,\quad u\in \FF^\mathrm{V}.
 \end{aligned}
\end{equation}
Clearly,  $(\EE^\mathrm{V}, \FF^{\mathrm{V}})$ is a regular but not irreducible Dirichlet form on $L^2\left(\mathbb{R}^2_0\right)$. The associated Markov process $X^\mathrm{V}$ is the absorbing Brownian motion on $\bR^2_0$, namely,  
\[X^\mathrm{V}_t=\begin{cases}
X^\mathrm{I}_t, \quad & t<\zeta,\\
\partial, \quad & t\geq\zeta,
\end{cases}
\]
where $\zeta:=\inf\{t>0: X^\mathrm{I}_t\notin \mathbb{R}^2_0\}$ and $\mathbb{R}^2_0\cup \{\partial\}$ is the one-point compactification of $\mathbb{R}^2_0$.
The generator $\L^\mathrm{V}$ is $\frac{1}{2}\rDelta u$ restricted to the domain
\begin{equation}\label{eq:type5-generator}
\D(\L^\mathrm{V})=\left\{u\in H^1_\Delta(\bG^2): \gamma_+ u_+=\gamma_-u_-=0\right\}.\end{equation}


\subsection{Phase of type VI}
The phase of type VI is given by the Dirichlet form on $L^2(\bG^2)$: 
\begin{equation}\label{eq:type6}
\begin{aligned}
  \FF^\mathrm{VI}&=H^1(\mathbb{G}^2),\\
 \EE^\mathrm{VI} (u,u)&=
 \frac{1}{2}\int_{\mathbb{G}^2} |\nabla u|^2dx+ \frac{\mu}{4\pi}\int_{\mathbb{R}\times\mathbb{R}}\frac{\left(u(x_1+)-u(x_1'-)\right)^2}{\left(\frac{2\ell}{\pi}\right)^2\left(\cosh\left(\frac{\pi}{2\ell}(x_1-x_1')\right)+1\right)}dx_1dx_1'\\
&\qquad \qquad+\frac{\mu}{8\pi}\int_{\mathbb{R}\times\mathbb{R}}\frac{\left(u(x_1+)-u(x_1'+)\right)^2+\left(u(x_1-)-u(x_1'-)\right)^2}{\left(\frac{2\ell}{\pi}\right)^2\left(\cosh\left(\frac{\pi}{2\ell}(x_1-x_1')\right)-1\right)}dx_1dx_1',\quad u\in \sF^\mathrm{VI},\\
 \end{aligned}
\end{equation}
where $\ell,\mu >0$ are given constants.  The following lemma states the basic facts about $(\cE^\mathrm{VI}, \cF^\mathrm{VI})$. 

\begin{lemma}\label{LEM27}
 $(\EE^\mathrm{VI},\FF^\mathrm{VI})$ is a regular and irreducible Dirichlet form on $L^2(\mathbb{G}^2)$. Furthermore, its generator $\L^\mathrm{VI}$ on $L^2(\mathbb{G}^2)$ is $\frac{1}{2}\rDelta$ restricted to the domain
\begin{equation}\label{eq:type6-generator}
\D(\L^\mathrm{VI})=\left\{u\in  H^1_\Delta(\mathbb{G}^2):
 \left.\frac{\partial u_{\pm}}{\partial x_2}\right|_{x_2=0\pm}=
{ \frac{\mu}{2\pi}}\int_\mathbb{R}\frac{\left(u(x_1+)-{u(x'_1-)}\right)}{\left(\frac{2\ell}{\pi}\right)^2\left(\cosh\left(\frac{\pi}{2\ell}(x_1-x_1')\right)+1\right)}dx'_1{\mp\frac{\mu}{4\pi}}\check{\mathcal{L}}_\ell(\gamma_\pm u)\right\},
\end{equation}
where $\check{\mathcal{L}}_\ell$ is a symmetric L\'evy type operator: For any $w\in {H}^{\frac{1}{2}}(\mathbb{R})$, 
\begin{equation}\label{eq:Ll}
\check{\mathcal{L}}_\ell w(x)=\int_{\mathbb{R}}\frac{w(x+y)-2w(x)+w(x-y)}{\left(\frac{2\ell}{\pi}\right)^2\left(\cosh\left(\frac{\pi}{2\ell}y\right)-1\right)}dy.
\end{equation}
\end{lemma} 
\begin{proof}
Note that for any fixed $x$, the function 
$\ell\mapsto\left(\frac{2\ell}{\pi}\right)^2\left(\cosh\left(\frac{\pi}{2\ell}x\right)-1\right)$ is decreasing on $(0,\infty)$, and 
\begin{equation}\label{EQcosh}
\left(\frac{2\ell}{\pi}\right)^2\left(\cosh\left(\frac{\pi}{2\ell}x\right)-1\right)\geq\lim_{\ell\uparrow \infty}\left(\frac{2\ell}{\pi}\right)^2\left(\cosh\left(\frac{\pi}{2\ell}x\right)-1\right)=\frac{x^2}{2}. 
\end{equation}
In addition, for $f\in H^{\frac{1}{2}}(\mathbb{R})$,
\begin{equation}\label{eq:221}
\|f\|^2_{H^{\frac{1}{2}}(\mathbb{R})}=\int_\mathbb{R}f^2dx_1+c\int_{\mathbb{R}\times\mathbb{R}}\frac{\left(f(x_1)-f(x_1')\right)^2}{(x_1-x_1')^2}dx_1dx_1',
\end{equation}
where $c$ is an absolute constant; see, e.g., \cite[Proposition 1.37]{BCD11}.  Hence the third term in $\sE^\mathrm{VI}(u,u)$ is bounded by 
\begin{equation}\label{eq:228}
	\frac{\mu}{4\pi c}\cdot \left(\|\gamma_+ u_+\|^2_{H^{\frac{1}{2}}(\mathbb{R})}+\|\gamma_- u_-\|^2_{H^{\frac{1}{2}}(\mathbb{R})}\right)
\end{equation}
On the other hand,  a straightforward computation yields that the second term in $\sE^\mathrm{VI}(u,u)$  is not greater than
\begin{equation}\label{eq:222}
	  \frac{\mu}{2\pi}\int_{\mathbb{R}\times\mathbb{R}}\frac{u(x_1+)^2+u(x_1'-)^2}{\left(\frac{2\ell}{\pi}\right)^2\left(\cosh\left(\frac{\pi}{2\ell}(x_1-x_1')\right)+1\right)}dx_1dx_1'\lesssim \| \gamma_+u_+\|_{L^2(\bR)}^2 + \| \gamma_-u_-\|_{L^2(\bR)}^2.
\end{equation}
From \eqref{eq:228}, \eqref{eq:222} and \eqref{EQ2TRH2},  we can obtain that for any $u\in H^1(\bG^2)$, 
\[\EE^\mathrm{III}_1(u,u)\leq\EE_1^{\mathrm{VI}}(u,u)\lesssim \|u\|_{H^1(\mathbb{G}^2)}^2+\|\gamma_+u_+\|_{H^{\frac{1}{2}}(\mathbb{R})}^2+\|\gamma_-u_-\|_{H^{\frac{1}{2}}(\mathbb{R})}^2\lesssim \EE^\mathrm{III}_1(u,u).\]
This obviously yields that $(\EE^\mathrm{VI},\FF^\mathrm{VI})$ is a regular Dirichlet form on $L^2(\mathbb{G}^2)$. 

Next we derive the irreducibility of $(\sE^\mathrm{VI}, \sF^\mathrm{VI})$.     Let $m$ be the Lebesgue measure on $\bG^2$ and take an $m$-invariant set $A$ of $X^\mathrm{VI}$.  Note that the part process of $X^\mathrm{VI}$ on $\bR^2_\pm$ is the absorbing Brownian motion denoted by $(B^\pm_t)_{t\geq 0}$.  By means of \cite[Theorem~1.6.1]{FOT11},  it is easy to verify that $A\cap \bR^2_\pm$ is an $m|_{\bR^2_\pm}$-invariant set of $B^\pm$.  Due to the irreducibility of $B^\pm$,  we have $A\cap \bR^2_\pm=\emptyset$ or $\bR^2_\pm$,  $m$-a.e.  It suffices to show $\bR^2_+$ or $\bR^2_-$ is not an $m$-invariant set of $X^\mathrm{IV}$.  Argue by contradiction and suppose that $\bR^2_+$ (resp. $\bR^2_-$) is an $m$-invariant set of $X^\mathrm{IV}$.  Then \cite[Theorem~1.6.1]{FOT11} tells us that for any $u\in H^1(\bG^2)$,  $u\cdot 1_{\bR^2_+}\in H^1(\bG^2)$ (resp.  $u\cdot 1_{\bR^2_-}\in H^1(\bG^2)$) and 
\begin{equation}\label{eq:230}
	\sE^\mathrm{VI}(u,u)=\sE^\mathrm{VI}(u\cdot 1_{\bR^2_+},u\cdot 1_{\bR^2_+}) + \sE^\mathrm{VI}(u\cdot 1_{\bR^2_-},u\cdot 1_{\bR^2_-}).
\end{equation}
However, the right hand side of \eqref{eq:230} is equal to 
\begin{equation}\label{eq:231}
	\sE^\mathrm{VI}(u,u)+\frac{\mu}{2\pi}\int_{\mathbb{R}\times\mathbb{R}}\frac{u(x_1+)u(x_1'-)}{\left(\frac{2\ell}{\pi}\right)^2\left(\cosh\left(\frac{\pi}{2\ell}(x_1-x_1')\right)+1\right)}dx_1dx_1'.
\end{equation}
Hence \eqref{eq:230} implies that the second term of \eqref{eq:231} must be equal to $0$ for all $u\in H^1(\bG^2)$.  This is obviously a contradiction.  Therefore the irreducibility of $(\sE^\mathrm{VI}, \sF^\mathrm{VI})$ can be concluded.  
 
Finally we formulate the generator and its domain.  Denote the family on the right hand side of \eqref{eq:type6-generator} by $\mathcal{G}$.  Similar to the proof of Lemma \ref{LEM21},  it is straightforward to verify, by using the Green-Gauss formula \eqref{eq:GreenGaussformula}, that $\mathcal{G}\subset \mathcal{D}(\L^\mathrm{VI})$ and $\L^\mathrm{VI} u=\frac{1}{2}\rDelta u$ for $u\in \mathcal{G}$.  To show $ \mathcal{D}(\L^\mathrm{VI})\subset \mathcal{G}$, take $u\in \mathcal{D}(\L^\mathrm{VI})$ with $f=\L^\mathrm{VI}u$.  Analogically we can obtain that $\Delta u_\pm=2f_\pm\in L^2(\bR^2_\pm)$ and hence $u\in H^1_\Delta(\bG^2)$.  In addition, for any $v\in H^1(\bG^2)$,  it follows from $(-\frac{1}{2}\rDelta u, v)_H=\frac{1}{2}\int_{\bG^2}\nabla u \cdot \nabla v dx+\frac{\mu}{4\pi}A_1+\frac{\mu}{8\pi}(A_++A_-)$ and the Green-Gauss formula \eqref{eq:GreenGaussformula} that
\begin{equation}\label{eq:224}
\begin{split}
\left\langle\left.\frac{\partial u_{+}}{\partial x_2}\right|_{x_2=0+}, \gamma_+v\right\rangle-\left\langle\left.\frac{\partial u_{-}}{\partial x_2}\right|_{x_2=0+}, \gamma_-v \right\rangle=\frac{\mu}{2\pi}A_1+\frac{\mu}{4\pi}(A_++A_-),
\end{split}
\end{equation}
where
\[
A_1:=\int_{\mathbb{R}\times\mathbb{R}}\frac{\left(u(x_1+)-u(x_1'-)\right)\left(v(x_1+)-v(x_1'-)\right)}{\left(\frac{2\ell}{\pi}\right)^2\left(\cosh\left(\frac{\pi}{2\ell}(x_1-x_1')\right)+1\right)}dx_1dx_1'
\]
and 
\[
A_\pm:=\int_{\mathbb{R}\times\mathbb{R}}\frac{\left(u(x_1\pm)-u(x_1'\pm)\right)\left(v(x_1\pm)-v(x_1'\pm)\right)}{\left(\frac{2\ell}{\pi}\right)^2\left(\cosh\left(\frac{\pi}{2\ell}(x_1-x_1')\right)-1\right)}dx_1dx_1'.
\]
Letting $v(x):=v_1(x_1)v_2(x_2)$ with $v_1\in C_c^\infty(\bR), v_2|_{\bG_+}\in C_c^\infty([0+,\infty))$,  $v_2|_{\bG_-}\equiv 0$ and $v_2(0+)\neq 0$, we have
\[
A_1=v_2(0+)\int_\mathbb{R}\left(\int_\mathbb{R}\frac{u(x_1+)-u(x_1'-)}{\left(\frac{2\ell}{\pi}\right)^2\left(\cosh\left(\frac{\pi}{2\ell}(x_1-x_1')\right)+1\right)}dx_1' \right)v(x_1)dx_1,
\]
and 
\[
A_+=-v_2(0+)\int_\mathbb{R}\check{\mathcal{L}}_\ell(\gamma_+u)\cdot v_1dx,\quad A_-=0.
\]
It is easy to verify that 
\[
	F(x_1):=\frac{\mu}{2\pi}\int_\mathbb{R}\frac{u(x_1+)-u(x_1'-)}{\left(\frac{2\ell}{\pi}\right)^2\left(\cosh\left(\frac{\pi}{2\ell}(x_1-x_1')\right)+1\right)}dx_1'-\frac{\mu}{4\pi}\check{\mathcal{L}}_\ell(\gamma_+u)(x_1)\in L^2(\bR),
\]	
and \eqref{eq:224} indicates 
\[
	\left\langle\left.\frac{\partial u_{+}}{\partial x_2}\right|_{x_2=0+}, v_1\right\rangle =\int_\bR F(x_1)v_1(x_1)dx_1,\quad \forall v_1\in C_c^\infty(\bR). 
\]
Thus we can conclude that 
\[
\left.\frac{\partial u_{+}}{\partial x_2}\right|_{x_2=0+}=\frac{\mu}{2\pi}\int_\mathbb{R}\frac{u(x_1+)-u(x_1'-)}{\left(\frac{2\ell}{\pi}\right)^2\left(\cosh\left(\frac{\pi}{2\ell}(x_1-x_1')\right)+1\right)}dx_1'-\frac{\mu}{4\pi}\check{\mathcal{L}}_\ell(\gamma_+u)\in L^2(\bR)
\] 
and analogically, 
\[
\left.\frac{\partial u_{+}}{\partial x_2}\right|_{x_2=0-}=\frac{\mu}{2\pi}\int_\mathbb{R}\frac{u(x_1+)-u(x_1'-)}{\left(\frac{2\ell}{\pi}\right)^2\left(\cosh\left(\frac{\pi}{2\ell}(x_1-x_1')\right)+1\right)}dx_1'+\frac{\mu}{4\pi}\check{\mathcal{L}}_\ell(\gamma_-u) \in L^2(\bR).
\] 
Note that  by \eqref{EQcosh} we have 
\[
\int_{\mathbb{R}}\frac{1\wedge |y|^2}{\left(\frac{2\ell}{\pi}\right)^2\left(\cosh\left(\frac{\pi}{2\ell}y\right)-1\right)}dy\leq \int_\bR \frac{1\wedge |y|^2}{|y|^2}dy<\infty, 
\]
which means that $\check{\mathcal{L}}_\ell$ is a symmetric L\'evy type operator.
That completes the proof.
\end{proof}

Denote the associated Markov process of $(\sE^\mathrm{VI},\sF^\mathrm{VI})$ by $X^\mathrm{VI}$.
Define a spacial transform
\[\mathbf{T}_\ell:=\mathbb{G}^2\rightarrow \Omega_\ell^c=\{x\in \mathbb{R}^2: |x_2|\geq \ell\}, \quad (x_1,x_2)\rightarrow\left(x_1, x_2+\frac{x_2}{|x_2|}\ell\right).\]
When $\mu=1$, we have the following characterization of $X^\mathrm{VI}$. 

\begin{lemma}\label{LM28}
Assume that $\mu=1$. Then $\mathbf{T}_\ell (X^\mathrm{VI})$ is the trace of a two-dimensional Brownian motion on the region $\Omega_\ell^c=\{x\in \mathbb{R}^2: |x_2|\geq \ell\}$. 
\end{lemma}
\begin{proof}
Fix $\ell>0$.  For simplicity of notation, let $(\sE,\sF):=(\frac{1}{2}\mathbf{D},H^1(\bR^2))$ be the Dirichlet form of the two-dimensional Brownian motion $X^\mathrm{I}$, and we write $F:=\Omega_\ell^c$ and $G:=\Omega_\ell$, and $(\check{\EE},\check{\FF})$ for the trace Dirichlet form of $(\sE,\sF)$ on $F$. 
It is known from \cite[Example 1.6.2]{FOT11} that the extended Dirichlet space of $X^\mathrm{I}$ is
\[\FF_e=\left\{u\in L^2_{\mathrm{loc}} (\mathbb{R}^2): \nabla u\in L^2(\mathbb{R}^2)\right\}.\]
Then $\check{\FF}=\FF_e|_F\cap L^2(F)=:H^1(F)$. Since $(\EE,\FF)$ is recurrent, we have from \cite[Theorem 5.2.5]{CF12} that $(\check{\EE},\check{\FF})$ is recurrent and conservative, and thus $(\check{\EE},\check{\FF})$ has no killing inside. Moreover, from \cite[Corollary 5.6.1]{CF12} we can obtain that for $u\in\check{\FF}_e$,
\begin{equation}\label{TR}
\check{\EE}(u,u)=\frac{1}{2}\mu_{\langle\mathbf{H}_Fu\rangle}(F)+\frac{1}{2}\int_{F\times F\setminus\mathtt{d}}\left(u(x)-u(y)\right)^2U(dx,dy),
\end{equation}
where $\mu_{\langle\mathbf{H}_Fu\rangle}$ is the energy measure of $(\EE,\FF)$ relative to $\mathbf{H}_Ff$, $\mathbf{H}_F$ is the hitting distribution of $X^\mathrm{I}$ on $F$ and $U$ is the so-called Feller measure of $X^\mathrm{I}$ on $F\times F\setminus\mathtt{d}$. Namely, 
\[\mathbf{H}_Fu(x)=\mathbf{E}_x[u(X^\mathrm{I}_{\sigma_F}), \sigma_F<\infty],   
\]
and $\sigma_F$ denotes the hitting time of $F$ with respect to $X^\mathrm{I}$. For $\lambda>0$, we also write
\[
\mathbf{H}^{\lambda}_Fu(x)=\mathbf{E}_x[e^{-\lambda\sigma_F}u(X^\mathrm{I}_{\sigma_F}), \sigma_F<\infty]. 
\]
Since $\mathbf{H}_Fu=u$ on $F$, we have 
\[\mu_{\langle\mathbf{H}_Fu\rangle}(F)=\int_F |\nabla u|^2dx.
\]
On the other hand, for the Feller measure $U$, taking two non-negative function $\phi$ and $\psi$ on $F$, we know from \cite[(5.5.13), (5.5.14)]{CF12} that 
\begin{equation*}
U(\phi\otimes\psi)=\uparrow\lim_{\lambda\uparrow\infty}\lambda(\mathbf{H}^{\lambda}_F\phi, \mathbf{H}_F\psi)_{G}. 
\end{equation*}
Since $X^\mathrm{I}$ has continuous trajectories, it follows that $U$ is supported on $\partial F\times \partial F$, and the distribution of $X^\mathrm{I}$ before $\sigma_F$ is the same as the reflecting Brownian motion on $\bar{\Omega}_\ell$ in Appendix~\ref{APB} before hitting the boundary. Thus the non-local part of \eqref{TR} is the same as $\check{\mathcal{A}^\ell}$ in \eqref{TRACEDF}. That completes the proof. 
\end{proof}

{As an application of Lemma \ref{LM28}, we obtain the recurrence of  $(\EE^\mathrm{VI},\FF^\mathrm{VI})$. 

\begin{corollary}
 $(\EE^\mathrm{VI},\FF^\mathrm{VI})$ is recurrent.
\end{corollary}
\begin{proof}
When $\mu=1$, we have proved in Lemma \ref{LM28} that $\mathbf{T}_\ell (X^\mathrm{VI})$ is the trace of a two-dimensional Brownian motion on $\Omega_\ell^c=\{x\in \mathbb{R}^2: |x_2|\geq \ell\}$.  It follows from \cite[Theorem 5.2.5]{CF12} that $\mathbf{T}_\ell (X^\mathrm{VI})$ is recurrent.  Hence $X^\mathrm{VI}$ is also recurrent. The recurrence for general $\mu\in(0,\infty)$ follows from \cite[Theorem 1.6.6]{FOT11}. 
\end{proof} }

\subsection{Phase of type VII}

The phase of type VII is given by the Dirichlet form on $L^2(\bG^2)$: 
\begin{equation}\label{eq:type7}
\begin{aligned}
\FF^\mathrm{VII}&=H^1(\mathbb{G}^2)\\
 \EE^\mathrm{VII} (u,u)&=
 \frac{1}{2}\int_{\mathbb{G}^2} |\nabla u|^2dx+ \frac{\mu}{8\pi}\bigg(\int_{\mathbb{R}\times\mathbb{R}}\frac{\left((u(x_1+)-u(x_1'+)\right)^2}{(x_1-x_1')^2}dx_1dx_1' \\
 &\qquad \qquad\qquad \quad +\int_{\mathbb{R}\times\mathbb{R}}\frac{\left((u(x_1-)-u(x_1'-)\right)^2}{(x_1-x_1')^2}dx_1dx_1'\bigg),\quad u\in \sF^\mathrm{VII},
 \end{aligned}
\end{equation}
where $\mu>0$ is a given constant.  Note that the non-local part of $(\sE^\mathrm{VI},\sF^\mathrm{VI})$ in \eqref{eq:type6} converges to that of \eqref{eq:type7} as $\ell \uparrow \infty$ as we see in Lemma~\ref{LM52}.  In fact, we may regard the type VII case as the approximating case of type VI as $\ell \uparrow \infty$.  A rigorous statement for this observation will be shown in Theorem~\ref{THM33}. Analogically we have the following.

\begin{lemma}
$(\EE^\mathrm{VII},\FF^\mathrm{VII})$ is a regular but not irreducible Dirichlet form on $L^2(\bG^2)$. Furthermore,  its generator $\L^\mathrm{VII}$ on $L^2(\bG^2)$ is $\frac{1}{2}\rDelta u$ restricted to the domain
\begin{equation}\label{eq:type7-generator}
\D(\L^\mathrm{VII})=\Bigg\{u\in H^1_\Delta(\mathbb{G}^2):
\left.\frac{\partial u_{\pm}}{\partial x_2}\right|_{x_2=0\pm}=\mp\frac{\mu}{4\pi}\check{\mathcal{L}}(\gamma_\pm u)\Bigg\},
\end{equation}
where the symmetric L\'evy type operator
\begin{equation}\label{eq:Ll2}
\check{\mathcal{L}}w(x)=\int_{\mathbb{R}}\frac{w(x+y)-2w(x)+w(x-y)}{y^2}dy
\end{equation}
corresponds to the $1$-stable process (or the Cauchy process) on $\mathbb{R}$.   
\end{lemma}
\begin{proof}
The proof for regularity and the formulation of the generator are analogical to those of Lemma \ref{LEM27}. Here we just prove that $(\EE^\mathrm{VII},\FF^\mathrm{VII})$ is not irreducible. Indeed, from the expression of $\EE^\mathrm{VII}$ and \cite[Theorem 1.6.1]{FOT11} we know that $\bG^2_+$ and $\bG^2_-$ are two invariant sets of $(\EE^\mathrm{VII},\FF^\mathrm{VII})$.  As a result, $(\EE^\mathrm{VII},\FF^\mathrm{VII})$ is not irreducible.
\end{proof}


\subsection{Summarization of all limiting phases}\label{SEC28}

\begin{table}
\centering
\begin{tabular}{cccccc}
\toprule  
Phase & Undetermined & Dirichlet & Associated  & State   & Generator\\
	type		&				coefficients			& form 	& Markov process & space & domain \\
\midrule  
I & -- & \eqref{eq:type3} & BM & $\bR^2$ &\eqref{eq:type3-generator}	\\
II & $\kappa$ & \eqref{eq:type2} & BM with snapping out jumps & $\bG^2$ & \eqref{eq:type2-generator} \\
III & --  & \eqref{eq:p1} & reflecting BM & $\bG^2$ & \eqref{eq:type1-generator} \\
IV & $\lambda$ & \eqref{eq:type4} & BM with drifts on $x_1$-axis  &	 $\bR^2$ & \eqref{eq:type4-generator}\\
V & --  & \eqref{eq:type5} & absorbing BM & $\bR^2_0$ &  \eqref{eq:type5-generator}	\\
VI & $\mu, \ell$ &  \eqref{eq:type6}  &	 BM with interacting jumps & $\bG^2$ & \eqref{eq:type6-generator} \\ 
VII & $\mu$ & \eqref{eq:type7} & BM with self-interacting jumps  & $\bG^2$ & \eqref{eq:type7-generator} \\
\bottomrule 
\end{tabular}
\caption{Summarization of limiting phases}
\label{table1}
\end{table}

For readers' convenience, we summarize all these phases in Table~\ref{table1}. 
Further remarks concerning the Dirichlet forms are as follows:
\begin{itemize}
\item[(1)] All Dirichlet forms are regular on its own $L^2$-space.  The Dirichlet forms of types I, II, IV and VI are irreducible, while the others are not. 
\item[(2)] The $L^2$-generators for all phases are $\frac{1}{2}\rDelta$ restricted to different subspaces of $H^1_\Delta(\bG^2)$ as presented in Table~\ref{table1}. 
\end{itemize}

Regarding their probabilistic counterparts,  all associated Markov processes are equivalent to a Brownian motion before hitting the $x_1$-axis or $\partial \bG^2$.  Hence they differ only near the $x_1$-axis or $\partial \bG^2$.   In the types II, III, VI and VII,  the state space is $\bG^2$,  and there appear three different kinds of jump,  namely \emph{snapping out jump,  interacting jump} and \emph{self-interacting jump}, on $\partial \bG^2$: 
\begin{itemize}
\item[(1)] $X^\mathrm{II}$ enjoys snapping out jumps.  To be precise, when reaches $(x_1,0+)$ (resp. $(x_1, 0-)$), $X^\mathrm{II}$ has a chance to jump to $(x_1,0-)$ (resp. $(x_1,0+)$). 
\item[(2)] $X^\mathrm{III}$ enjoys no jumps.  It reflects back immediately upon hitting the barrier.  
\item[(3)] $X^\mathrm{VI}$ enjoys interacting and self-interacting jumps.    The first kind of jump,  characterized by the second term in the expression of $\sE^\mathrm{VI}$,  takes place between two points on different components of $\partial \bG^2$,  while the second kind of jump,  characterized by the third term in the expression of $\sE^\mathrm{VI}$,  takes place between two points on the same component of $\partial \bG^2$.  
\item[(4)] $X^\mathrm{VII}$ only enjoys self-interacting jumps. 
\end{itemize}
Note that snapping out jumps and interacting jumps link one component of $\bG^2$ with the other,  while self-interacting jumps do not.  As a consequence,  both $X^\mathrm{II}$ and $X^\mathrm{VI}$ are irreducible,  but neither  $X^\mathrm{III}$ nor $X^\mathrm{VII}$ is.  Furthermore, $X^\mathrm{I}$ and $X^\mathrm{V}$ are classical,  and we only explain the diffusion process $X^\mathrm{IV}$ on $\bR^2$ by a few lines.  Comparing to a Brownian motion,  $X^\mathrm{IV}$ enjoys acceleration along the tangent direction upon hitting the $x_1$-axis. More precisely,  as we see in \eqref{representationoftype4},  the normal component of $X^\mathrm{IV}$ is a Brownian path,  while the tangent component is an accelerating Brownian path due to the additional term $2\lambda L_t$ in time,  which increases only when $X^\mathrm{IV}$ hits the $x_1$-axis.


\section{Phase transitions of two-dimensional stiff problem}\label{SEC3}

Recall that $(\sE^\varepsilon,\sF^\varepsilon)$ is the Dirichlet form defines as \eqref{defofEEVare},  which is associated with the diffusion process $X^\varepsilon$ in \eqref{eq:19}. 
For every $\varepsilon>0$,  let 
\begin{equation}
		\mathsf{C}_\varepsilon^{\raisebox{0mm}{-}}:=\varepsilon a^{\raisebox{0mm}{-}}_\varepsilon,\quad \mathsf{R}^\shortmid_\varepsilon:=\frac{\varepsilon}{a^\shortmid_\varepsilon}.
\end{equation}
be the tangent total conductivity and  the normal total resistance of $\Omega_\varepsilon$,  and
\[
	\mathsf{M}_\varepsilon:= \sqrt{\frac{\mathsf{C}^{\raisebox{0mm}{-}}_\varepsilon}{\mathsf{R}^\shortmid_\varepsilon}},
\]
be the mixing scale at $\varepsilon$.  Set further
\[
	\ell_\varepsilon:=\sqrt{\mathsf{C}^{\raisebox{0mm}{-}}_\varepsilon \mathsf{R}^\shortmid_\varepsilon}
\]
called the \emph{splitting length at $\varepsilon$}. 


\subsection{Main theorem}\label{SEC31}

Take a decreasing sequence $\varepsilon_n\downarrow 0$, and write $(\EE^{\varepsilon_n},\FF^{\varepsilon_n})$, $a^{\raisebox{0mm}{-}}_{\varepsilon_n}$, $a^{\shortmid}_{\varepsilon_n}$,  $\mathsf{C}_{\varepsilon_n}^{\raisebox{0mm}{-}}$, $\mathsf{R}^\shortmid_{\varepsilon_n}$, $\mathsf{M}_{\varepsilon_n}$, $\ell_{\varepsilon_n}$ and $\Omega_{\varepsilon_n}$ as $(\EE^n,\FF^n)$, $a^{\raisebox{0mm}{-}}_n$, $a^{\shortmid}_n$, $\mathsf{C}_n^{\raisebox{0mm}{-}}$, $\mathsf{R}^\shortmid_n$, $\mathsf{M}_n$, $\ell_n$ and $\Omega_n$.  Recall the $H=L^2(\bR^2)=L^2(\bG^2)=L^2(\bR^2_0)$ and all appearing Dirichlet forms are defined on $H$. Then the main theorem of this section is as follows.

\begin{theorem}\label{THMAIN}
Assume that the following limits exist in the wide sense:
\begin{equation}\label{eq:limitsexist}
	\mathsf{C}^{\raisebox{0mm}{-}}:=\lim_{n\rightarrow\infty}\mathsf{C}^{\raisebox{0mm}{-}}_n, \quad \mathsf{R}^\shortmid:=\lim_{n\rightarrow\infty}\mathsf{R}^\shortmid_n,\quad \mathsf{M}:=\lim_{n\rightarrow\infty}\mathsf{M}_n.
	\end{equation}
\begin{itemize}
\item[(1)] When $\mathsf{M}=0$, $(\EE^{n},\FF^{n})$ manifests a \emph{normal phase transition} as $n\rightarrow \infty$ in the following sense:
\begin{description}
\item[(N1)] $\mathsf{R}^\shortmid=0$: $(\sE^n,\sF^n)$ converges to $(\sE^\mathrm{I},\sF^\mathrm{I})$ in the sense of Mosco as $n\rightarrow\infty$;
\item[(N2)] $0<\mathsf{R}^\shortmid<\infty$: $(\sE^n,\sF^n)$ converges to $(\sE^\mathrm{II},\sF^\mathrm{II})$ with $\kappa=1/\mathsf{R}^\shortmid$ in the sense of Mosco as $n\rightarrow\infty$.;
\item[(N3)] $\mathsf{R}^\shortmid=\infty$: $(\sE^n,\sF^n)$ converges to $(\sE^\mathrm{III},\sF^\mathrm{III})$ in the sense of Mosco as $n\rightarrow\infty$. 
\end{description}
\item[(2)] When $\mathsf{M}=\infty$,  $(\EE^{n},\FF^{n})$ manifests a \emph{tangent phase transition} as $n\rightarrow \infty$ in the following sense:
\begin{description}
\item[(T1)] $\mathsf{C}^{\raisebox{0mm}{-}}=0$: $(\sE^n,\sF^n)$ converges to $(\sE^\mathrm{I},\sF^\mathrm{I})$ in the sense of Mosco as $n\rightarrow\infty$;
\item[(T2)] $0<\mathsf{C}^{\raisebox{0mm}{-}}<\infty$: $(\sE^n,\sF^n)$ converges to $(\sE^\mathrm{IV},\sF^\mathrm{IV})$ with $\lambda=\mathsf{C}^{\raisebox{0mm}{-}}$ in the sense of Mosco as $n\rightarrow\infty$;
\item[(T3)] $\mathsf{C}^{\raisebox{0mm}{-}}=\infty$: $(\sE^n,\sF^n)$ converges to $(\sE^\mathrm{V},\sF^\mathrm{V})$ in the sense of Mosco as $n\rightarrow\infty$. 
\end{description}
\item[(3)] When $0<\mathsf{M}<\infty$,  $(\EE^{n},\FF^{n})$ manifests a \emph{mixing phase transition} as $n\rightarrow \infty$ in the following sense:
\begin{description}
\item[(M1)] $\mathsf{R}^\shortmid=0$ (equivalently $\mathsf{C}^{\raisebox{0mm}{-}}=0$): $(\sE^n,\sF^n)$ converges to $(\sE^\mathrm{I},\sF^\mathrm{I})$ in the sense of Mosco as $n\rightarrow\infty$;
\item[(M2)] $0<\mathsf{R}^\shortmid<\infty$ (equivalently $0<\mathsf{C}^{\raisebox{0mm}{-}}<\infty$): $(\sE^n,\sF^n)$ converges to $(\sE^\mathrm{VI},\sF^\mathrm{VI})$ with
\[
	\mu=\mathsf{M},\quad \ell=\lim_{n\rightarrow \infty}\ell_n
\]
in the sense of Mosco as $n\rightarrow\infty$;
\item[(M3)] $\mathsf{R}^\shortmid=\infty$ (equivalently $\mathsf{C}^{\raisebox{0mm}{-}}=\infty$): $(\sE^n,\sF^n)$ converges to $(\sE^\mathrm{VII},\sF^\mathrm{VII})$ with $\mu=\sM$ in the sense of Mosco as $n\rightarrow\infty$. 
\end{description}
\end{itemize}
\end{theorem}
\begin{remark}
We should point out that the Mosco convergence implies the convergence of finite dimensional distributions of associated Markov processes; see, e.g., \cite[Corollary~4.1]{LS19}.  This is the sense in which we say $X^n$ converges to the limiting process throughout this paper.  
\end{remark}

\subsection{Further remarks}

 The proof of this main theorem will be completed in the next subsection.  Now we give some remarks on it.  Recall that $\mathsf{C}^{\raisebox{0mm}{-}}$,  $\mathsf{R}^\shortmid$ and $\mathsf{M}$ are called the tangent total conductivity,  the normal total resistance and the mixing scale respectively.  
 
We first describe the related thermal conduction models by some heuristic observations.  In the normal case,  $\sM=0$ roughly indicates that $\sR^\shortmid$ is much greater than $\sC^\shortline$.  Consequently,  normal resisting,  rather than tangent accelerating,  plays a crucial role in this phase transition.  Like the one-dimensional stiff problem studied in \cite{LS19},  $\sR^\shortmid$ determines the pattern of related thermal conduction models:
\begin{itemize}
\item[(1)] $\sR^\shortmid=0$:  It is called in the \emph{totally permeable pattern} ($\mathsf{P}_\mathrm{t}$) in the sense that the barrier makes no sense.  
\item[(2)] $0<\sR^\shortmid<\infty$: It is called in the \emph{normal semi-permeable pattern} ($\mathsf{P}_\mathrm{t}$).  In this case,  the heat flow can penetrate the barrier partially,  and in the probabilistic counterpart,  penetrations are realized by snapping out jumps,  as we explained in \S\ref{SEC28}. 
\item[(3)] $\sR^\shortmid=\infty$: It is called in the \emph{normal impermeable pattern} ($\mathsf{P}_\mathrm{i}^\shortmid$),  because the normal resisting is so strong that the heat flow cannot penetrate the barrier. 
\end{itemize}
In the tangent case $\sM=\infty$,  $\sC^\shortline$  instead of $\sR^\shortmid$ determines the pattern of related thermal conduction models:
\begin{itemize}
\item[(1)] $\sC^{\raisebox{0mm}{-}}=0$: Obviously it is in the {totally permeable pattern} ($\mathsf{P}_\mathrm{t}$).  
\item[(2)] $0<\sC^{\raisebox{0mm}{-}}<\infty$: It is called in the \emph{tangent semi-permeable pattern} ($\mathsf{P}_\mathrm{s}^{\raisebox{0mm}{-}}$).  The effective tangent accelerating leads to the layover of heat flow on the barrier. 
\item[(3)] $\sC^{\raisebox{0mm}{-}}=\infty$: It is called in the \emph{tangent impermeable pattern} ($\mathsf{P}_\mathrm{i}^{\raisebox{0mm}{-}}$),  because the tangent delaying is so long that the heat flow is indeed absorbed by the barrier.  In other words, it cannot penetrate the barrier. 
\end{itemize}
In the mixing case $0<\sM<\infty$,  $\sR^\shortmid$ and $\sC^{\raisebox{0mm}{-}}$ have proportionable effects on the limiting phases.  When $\sR^\shortmid=\sC^{\raisebox{0mm}{-}}=0$,  the conduction model is surely in the totally permeable pattern.   When $0<\sR^\shortmid, \sC^{\raisebox{0mm}{-}}\leq \infty$,  the heat flow may be thought of as the mixture of those in the normal and tangent cases.  We present a heuristic explanation for this mixing method by means of the probabilistic counterparts in Remark~\ref{RM33}.  Particularly,  the conduction model is called in the \emph{mixing semi-permeable  pattern} ($\mathsf{P}_\mathrm{s}^+$)  and \emph{mixing impermeable pattern} ($\mathsf{P}_\mathrm{i}^+$) for $0<\sR^\shortmid, \sC^{\raisebox{0mm}{-}}< \infty$ and $\sR^\shortmid, \sC^{\raisebox{0mm}{-}}= \infty$ respectively.  
All these patterns are illustrated in Figure~\ref{FIG5}.  

 \begin{figure}
 \centering
 \includegraphics[scale=0.55]{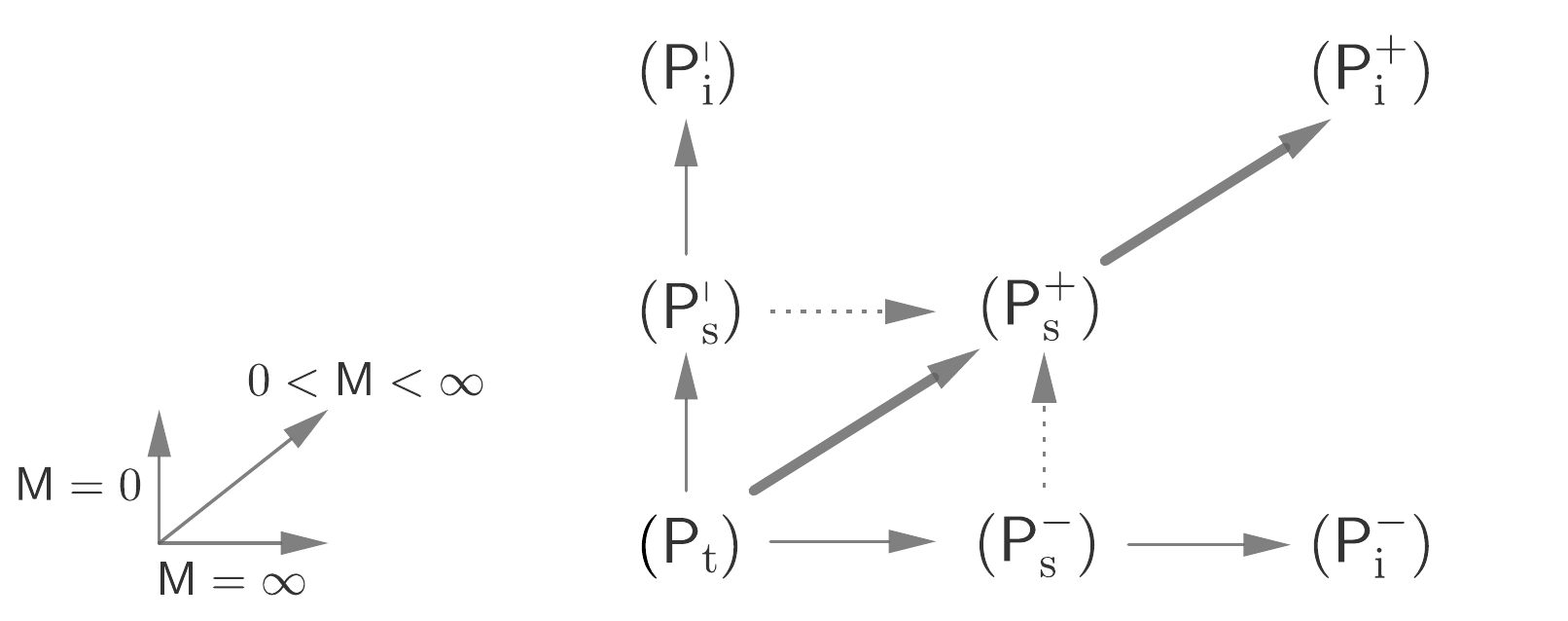}
 \caption{Patterns of thermal conduction models}
 \label{FIG5}
 \end{figure}

\begin{remark}\label{RM33}
We present a heuristic method to obtain a path of $X^\mathrm{VI}$ enjoying interacting and self-interacting jumps on $\partial \bG^2$ by mixing those of $X^\mathrm{II}$ and $X^\mathrm{IV}$ as follows.  
By splitting the barrier into two distinct components due to normal resisting,  a path of $X^\mathrm{IV}$ is cut into two pieces that may enjoy jumps between two points on the same component of the barrier.  These  jumps,  in collaboration with the snapping out jumps enjoyed by $X^\mathrm{II}$ or not,  yield interacting jumps or self-interacting jumps.  Nevertheless,  it is insensible to regard $X^\mathrm{VII}$ as a mixture of a reflecting Brownian motion and an absorbing Brownian motion.  Instead we may think of it as the limit of $X^\mathrm{VI}$ as $\ell$ increases to $\infty$;  see the notes below or Theorem~\ref{THM33}.   
\end{remark}

Next, let us turn to figure out the significance of the parameter $\ell$,  called the \emph{splitting length}.  Thanks to the argument in Lemma~\ref{LM28},  when $\mu=1$,  $X^\mathrm{IV}$ (up to a spatial transform) is the trace of two-dimensional Brownian motion on $\{x\in \bR^2: |x_2|\geq \ell\}$.  Heuristically,  $\ell$ measures the ``real" distance between the two components of $\bG^2$ (with respect to $X^\mathrm{IV}$).  It is worth noting that the normal component $\beta^\s$ of $X^\mathrm{II}$ has a similar representation as shown in \cite[Theorem~3.2]{LS19}, i.e. $\beta^\s$ (up to a spatial transform) is the trace of one-dimensional Brownian motion on $\{x\in \bR: |x|\geq 1/\kappa\}$.  In \textbf{(N2)}, $1/\kappa=\sR^\shortmid$ and it is easy to verify that $\ell$ is equal to the product of $1/\kappa$ and the mixing scale $\sM$, i.e.
\[
	\ell =\sM\cdot \frac{1}{\kappa}.
\]
On the other hand,  $\ell$ also gives insight into understanding the mixing impermeable phase.  In fact,  when $\sR^\shortmid\uparrow \infty$, we have $\ell=\sM\sR^\shortmid \uparrow \infty$.  This leads to the disappearing of interacting jumps,  because the two components of $\bG^2$ are essentially separate. However the self-interacting term in \eqref{eq:type6} has a non-trivial limit as $\ell\uparrow \infty$ and as a result,  $X^\mathrm{VII}$ still enjoys self-interacting jumps. 



Finally we emphasize that all the Dirichlet forms in the totally permeable case (i.e. $(\mathsf{P}_t)$) and semi-permeable cases (i.e. $(\mathsf{P}^\shortmid_\mathrm{s})$,  $(\mathsf{P}^\shortline_\mathrm{s})$ and $(\mathsf{P}^+_\mathrm{s})$) are irreducible, while those in the impermeable cases (i.e. $(\mathsf{P}^\shortmid_\mathrm{i})$,  $(\mathsf{P}^\shortline_\mathrm{i})$ and $(\mathsf{P}^+_\mathrm{i})$) are not.  This classification is in agreement with the behaviour of the heat flow:  It can penetrate a permeable or semi-permeable barrier, but not an impermeable barrier.   

\subsection{Proof of Theorem~\ref{THMAIN}}

The definition of Mosco convergence, consisting of two parts (a) and (b), is reviewed in Definition~\ref{DEF41}.  Since the proof is a little technical and involved,  we split it up in several parts.  
For convenience, we break the case {\bf(T3)} into two cases $\bf {(T3_1)}$ and ${\bf(T3_2)}$, depending on whether $\mathsf{R}^\shortmid>0$ or $\mathsf{R}^\shortmid=0$. 

\subsubsection{Proof of the first part of Mosco convergence}

At first let us give a tactic for the cases in $\mathcal{T}_1:= \{\textbf{(N2)},  \textbf{(N3)},   \textbf{(M2)}, {\bf (M3)},  {\bf (T3_1)} \}$.  To prove the first part (a) of Mosco convergence,  we need to prepare two lemmas concerning a sequence $\{f_n: n\geq 1\}\subset H$ such that $f_n$ converges weakly to $f$ in $H$ and $\sup_{n\geq 1}\sE^n(f_n,f_n)<\infty$.  For each $n$,  set a function $\tilde{f}_n$ on $\bG^2$: 
\begin{equation}\label{EQ32L1}
 \tilde{f}_n(x_1,x_2):=\begin{cases}
f_n(x_1, x_2+\varepsilon_n), \quad & x_2\in [0+, \infty),\\
f_n(x_1, x_2-\varepsilon_n), \quad & x_2\in (-\infty,0-].
\end{cases}
 \end{equation}
 The first lemma shows the convergences of $\tilde{f}_n$ and $\gamma_\pm \tilde{f}_n$.  
 
 \begin{lemma}\label{LM34}
 Let $\{f_n\}\subset H^1(\bR^2)$ be a sequence such that $f_n$ converges weakly to $f$ in $H$ and $$\sup_{n\geq 1}\sE^n(f_n,f_n)<\infty.$$
 Further let $\tilde{f}_n$ be defined as \eqref{EQ32L1}. Then the following hold:
 \begin{itemize}
 \item[(1)] $\tilde{f}_n, f\in H^1(\bG^2)$.
 \item[(2)] $\tilde{f}_n$ converges weakly to $f$ in $H$.
 \item[(3)] A subsequence of $\{\gamma_{\pm}\tilde{f}_n\}$ converges weakly to $\gamma_{\pm}f$ both  in $H^{\frac{1}{2}}(\mathbb{R})$ and $L^2(\bR)$. 
\end{itemize}
 \end{lemma}
 \begin{proof}
At first,  we claim that $\{\tilde{f}_n\}$ converge to $f$ weakly in $H$. Indeed, for any $g\in H$, we have
\[(f_n-\tilde{f}_n,g)_H=(f_n-\tilde{f}_n,g)_{L^2(\mathbb{G}^2_+)}+(f_n-\tilde{f}_n,g)_{L^2(\mathbb{G}^2_-)}, \]
where
\[
\begin{split}
 -(f_n-&\tilde{f}_n,g)_{L^2(\mathbb{G}^2_+)}\\
&= \int_\mathbb{R}\int_0^\infty f_n(x_1, x_2+\varepsilon_n)g(x_1,x_2)dx_1dx_2-\int_\mathbb{R}\int_0^\infty f_n(x_1, x_2)g(x_1,x_2)dx_1dx_2\\
&
= \int_\mathbb{R}\int_{\varepsilon_n}^\infty f_n(x_1,x_2)(g(x_1,x_2-\varepsilon_n)-g(x_1,x_2))dx_1dx_2-\int_\mathbb{R}\int_0^{\varepsilon_n}f_n(x_1, x_2)g(x_1,x_2)dx_1dx_2
\\
&=:I_1+I_2. 
\end{split}
\]
Note that $\sup_{n\geq 1}\|f_n\|_H<\infty$ due to the weak convergence of $f_n$.
Then it follows from $\int_{\Omega_n}g^2dx\rightarrow 0$ as $n\rightarrow\infty$ and the Cauchy-Schwarz inequality that $I_2\rightarrow 0$.   As for the term $I_1$, we have
\begin{equation}\label{EQ3EQUI}
I_1\leq \left(\int_{\mathbb{R}}\int^\infty_{\varepsilon_n}f_n^2dx_1dx_2\right)^{\frac{1}{2}} \left(\int_\mathbb{R}\int_{\varepsilon_n}^\infty |(g(x_1,x_2-\varepsilon_n)-g(x_1,x_2))|^2dx_1dx_2\right)^{\frac{1}{2}}.
\end{equation}
The second term on the right hand side tends to zero due to the equicontinuity of the integral (see, e.g., \cite[Theorem 2.32]{AF03}).
Hence $(f_n-\tilde{f}_n,g)_{L^2(\mathbb{G}^2_+)}\rightarrow 0$ as $n\rightarrow\infty$, and analogically we can also obtain that $(f_n-\tilde{f}_n,g)_{L^2(\mathbb{G}^2_-)}\rightarrow 0$.  Since $f_n$ converges to $f$ weakly in $H$,  we eventually conclude that $\tilde{f}_n$ converges to $f$ weakly in $H$. 

By the definition of $\tilde{f_n}$, 
\[
\int_{\mathbb{G}^2}|\tilde{f}_n|^2dx=\int_{\Omega_n^c}|f_n|^2dx,\quad \int_{\mathbb{G}^2}|\nabla \tilde{f}_n|^2dx=\int_{\Omega_n^c}|\nabla f_n|^2dx. 
\]
which imply $\tilde{f}_n\in H^1(\bG^2)$ and 
\[\sup_n\|\tilde{f_n}\|_{H^1(\mathbb{G}^2)}<\infty. \]
Hence,  taking a subsequence if necessary,  we may assume that $\frac{1}{n}\sum_{k=1}^n \tilde{f}_k$ converges to a function $h\in H^1(\bG^2)$ in $H^1(\bG^2)$.  Since $\tilde{f_n}$ converges to $f$ weakly in $H$,  it is easy to see that $$f=h\in H^1(\bG^2).$$ 
In addition, the trace theorem implies 
\begin{equation}\label{eq:35}
\sup_n\|\gamma_{\pm}\tilde{f_n}\|_{L^{2}(\mathbb{R})}\leq \sup_n\|\gamma_{\pm}\tilde{f_n}\|_{H^{\frac{1}{2}}(\mathbb{R})}\lesssim \sup_n\|\tilde{f_n}\|_{H^1(\mathbb{G}^2)} <\infty,
\end{equation}
and $\frac{1}{n}\sum_{k=1}^n \gamma_\pm \tilde{f}_k$ converges to $\gamma_\pm f$ both in $H^{1/2}(\bR)$ and $L^2(\bR)$.   
By means of Banach-Alaoglu theorem,  we can eventually conclude that a subsequence of $\{\gamma_\pm \tilde{f}_n\}$ converges weakly to $\gamma_\pm f$ both in $H^{1/2}(\mathbb{R})$ and $L^2(\bR)$.  That completes the proof.
 \end{proof}

Recall that $\ell_n=\sqrt{\frac{a^{\raisebox{0mm}{-}}_n}{a^{\shortmid}_n}}\varepsilon_n$.  Let $B^{\ell_n}$ be the reflecting Brownian motion on $\Omega_{\ell_n}$,  whose associated Dirichlet form is 
\[
\begin{aligned}
  \mathcal{G}^{\ell_n}&=H^1(\Omega_{\ell_n}),\\
 \mathcal{A}^{\ell_n} (u,u)&=
 \frac{1}{2}\int_{\Omega_{\ell_n}} |\nabla u|^2dx,\quad u\in \mathcal{G}^{\ell_n}. 
 \end{aligned}
\]
Further let $\tau:=\inf\{t>0: B^{\ell_n}_t\in\partial \Omega_{\ell_n}\}$ be the first hitting time of the boundary $\partial \Omega_{\ell_n}$, and for a suitable function $g$ on $\partial \Omega_{\ell_n}$, define 
\begin{equation}\label{EQ3EXPE}
\mathbf{H}g(x):=\mathbf{E}_x\left[g(B^{\ell_n}_\tau), \tau<\infty\right].
\end{equation} The second lemma obtains a lower bound involving $\tilde{f}_n$ for the energy of $f_n$ restricted to $\Omega_n$.  

\begin{lemma}\label{LM35}
Let $f_n$ and $\tilde{f}_n$ be in Lemma~\ref{LM34}. 
For any $n\geq 1$, it holds that
\begin{equation}\label{EQ3EXPLI}
	\dfrac{a^{\raisebox{0mm}{-}}_n}{2}\int_{\Omega_{n}} |\partial_{x_1} f_n|^2dx+\dfrac{a^{\shortmid}_n}{2}\int_{\Omega_{n}} |\partial_{x_2} f_n|^2dx \geq J_{n1}+J_{n2},
\end{equation}
where 
\[
J_{n1}:=\frac{(a^{\raisebox{0mm}{-}}_na^{\shortmid}_n)^{1/2}}{4 \ell_n}\int_{\mathbb{R}\times\mathbb{R}}\frac{\left(\tilde{f}_n(x_1+)-\tilde{f}_n(x_1'-)\right)^2}{2\left(\frac{2\ell_n}{\pi}\right)\left(\cosh\left(\frac{\pi}{2\ell_n}(x_1-x_1')\right)+1\right)}dx_1dx_1'
\]
and 
\[
J_{n2}:=\frac{(a^{\raisebox{0mm}{-}}_na^{\shortmid}_n)^{1/2}}{8\pi}\int_{\mathbb{R}\times\mathbb{R}}\frac{\left((\tilde{f}_n(x_1+)-\tilde{f}_n(x_1'+)\right)^2+\left((\tilde{f}_n(x_1-)-\tilde{f}_n(x_1'-)\right)^2}{\left(\frac{2\ell_n}{\pi}\right)^2\left(\cosh\left(\frac{\pi}{2\ell_n}(x_1-x_1')\right)-1\right)}dx_1dx_1'. 
\]
\end{lemma}  
\begin{proof}
Set
\[
u_n(x_1,x_2):=(a^{\raisebox{0mm}{-}}_na^{\shortmid}_n)^{1/4}f_n\left(x_1, \sqrt{\frac{a^{\shortmid}_n}{a^{\raisebox{0mm}{-}}_n}} x_2\right), \quad |x_2|<\sqrt{\frac{a^{\raisebox{0mm}{-}}_n}{a^{\shortmid}_n}}\varepsilon_n.
\]
A straightforward computation yields that
\begin{equation}\label{EQ3NORMSC}
 \mathcal{A}^{\ell_n}(u_n,u_n)=\dfrac{a^{\raisebox{0mm}{-}}_n}{2}\int_{\Omega_{n}} |\partial_{x_1} f_n|^2dx+\dfrac{a^{\shortmid}_n}{2}\int_{\Omega_{n}} |\partial_{x_2} f_n|^2dx<\infty.
\end{equation}
Particularly,  $u_n\in \mathcal{G}^{\ell_n}$ due to $\|u_n\|_{L^2(\Omega_{\ell_n})}\lesssim \|f\|_{L^2(\Omega_n)}$.  Then it follows from \cite[Proposition 3.4.1 and Theorem 3.4.8]{CF12} that
\begin{equation*}
\mathcal{A}^{\ell_n} (\mathbf{H}(u_n|_{\partial\Omega_{\ell_n}}),u_n-\mathbf{H}(u_n|_{\partial\Omega_{\ell_n}}))=0,
\end{equation*}
and thus 
\begin{equation}\label{EQ3GQTR}
 \mathcal{A}^{\ell_n}(u_n,u_n)\geq \mathcal{A}^{\ell_n}(\mathbf{H}(u_n|_{\partial\Omega_{\ell_n}}),\mathbf{H}(u_n|_{\partial\Omega_{\ell_n}}))=\check{\mathcal{A}}^{\ell_n}(u_n|_{\partial\Omega_{\ell_n}},u_n|_{\partial\Omega_{\ell_n}}),
\end{equation}
where $\check{\mathcal{A}}^{\ell_n}$ is the trace Dirichlet form of $(\mathcal{A}^{\ell_n}, \mathcal{G}^{\ell_n})$ on $\partial \Omega_{\ell_n}$, see Appendix~\ref{APB}.  Note that 
\[u_n(x_1, \pm \ell_n)=(a^{\raisebox{0mm}{-}}_na^{\shortmid}_n)^{1/4}f_n(x_1, \pm\varepsilon_n)=(a^{\raisebox{0mm}{-}}_na^{\shortmid}_n)^{1/4}\tilde{f}_n(x_1\pm).\]
By means of \eqref{TRACEDF},  we can verify that the last term in \eqref{EQ3GQTR} is equal to $J_{n1}+J_{n2}$.  Therefore \eqref{EQ3GQTR} leads to \eqref{EQ3EXPLI}.  That completes the proof.
\end{proof}

Now we have a position to prove the first part (a) of Mosco convergence for the cases in $\mathcal{T}_1$. 

\begin{proof}[of Mosco (a) for cases in $\mathcal{T}_1$]
Let $\{f_n\}\subset H$ that converges weakly to $f$ in $H$.  Without loss of generality, we may assume that 
\[\lim_{n\rightarrow \infty}\sE^n(f_n,f_n)\leq \sup_{n\geq 1}\EE^n(f_n,f_n)=:M<\infty. \]
Particularly,  $f_n\in H^1(\bR^2)$ and,  by Lemma~\ref{LM34}, $f\in H^1(\bG^2)$. Set a function $\tilde{f}_n$ for any $n$ as in \eqref{EQ32L1}. 
By Lemma~\ref{LM34},  $\tilde{f}_n\in H^1(\bG^2)$ converges to $f$ weakly in $H$,  and we may further assume without loss of generality that $\gamma_\pm \tilde{f}_n$ converges to $\gamma_\pm f$ weakly in $H^{1/2}(\bR)$.  Note that
 \begin{equation}\label{EQ3INF}
\liminf_{n\rightarrow\infty}\int_{\Omega_n^c}|\nabla f_n|^2dx=\liminf_{n\rightarrow\infty}\int_{\mathbb{G}^2}|\nabla \tilde{f}_n|^2dx\geq \int_{\mathbb{G}^2}|\nabla f|^2dx.
\end{equation}
The last inequality is due to the lower semi-continuity of $\mathscr{E}^\mathrm{III}$; see Remark~\ref{RMA3}.  Note further that the coefficients  in Lemma~\ref{LM35} converge to the constants below:
\begin{equation}\label{eq:constants}
\frac{(a^{\raisebox{0mm}{-}}_na^{\shortmid}_n)^{1/2}}{\ell_n}=\frac{a^{\shortmid}_n}{\varepsilon_n}\rightarrow \frac{1}{\mathsf{R}^\shortmid}\in [0,\infty], \quad (a^{\raisebox{0mm}{-}}_na^{\shortmid}_n)^{1/2}\rightarrow \mathsf{M}\in [0,\infty].
\end{equation}
In what follows,  we will prove \eqref{eqMoscoa} regarding $f_n$ and $f$ for the cases in $\mathcal{T}_1$: 
\begin{itemize}
\item[{\bf (N2)}]  In this case $\sM=0, 0<\sR^\shortmid<\infty$,  $\ell_n\rightarrow \ell=\mathsf{M}\mathsf{R}^\shortmid=0$ and $f\in \sF^\mathrm{II}$. Since $\gamma_{\pm}\tilde{f}_n$ converges weakly to $\gamma_{\pm}f$ in $H^{\frac{1}{2}}(\mathbb{R})$,  it follows from Lemma \ref{LM52} and \eqref{eq:constants} that 
\[
\liminf_{n\rightarrow\infty}(J_{n1}+J_{n2})\geq \frac{1}{4\mathsf{R}^\shortmid}\int_\mathbb{R}\left(f(x_1+)-f(x_1-)\right)^2dx_1. 
\]
Combining with \eqref{EQ3EXPLI} and \eqref{EQ3INF}, we can conclude that 
\begin{equation}\label{EQ3INF3}
\liminf_{n\rightarrow\infty}\EE^n (f_n,f_n)
\geq \frac{1}{2} \int_{\mathbb{G}^2}|\nabla f|^2dx+\frac{1}{4\mathsf{R}^\shortmid}\int_\mathbb{R}\left(f(x_1+)-f(x_1-)\right)^2dx_1
= \EE^{\mathrm{II}}(f,f)
\end{equation}
with $\kappa=1/\mathsf{R}^\shortmid$. 
\item[{\bf (N3)}] Note that $f\in \sF^\mathrm{III}$. Then it follows from \eqref{EQ3INF} that 
\[
\liminf_{n\rightarrow\infty}\EE^n (f_n,f_n)\geq \liminf_{n\rightarrow\infty}\frac{1}{2}\int_{\Omega_n^c}|\nabla f_n|^2dx \geq\frac{1}{2} \int_{\mathbb{G}^2}|\nabla f|^2dx=\EE^\mathrm{III}(f,f).
\]

\item[{\bf (M2)}]  In this case $0<\sM, \sR^\shortmid<\infty$,  $\ell_n\rightarrow \ell=\mathsf{M}\mathsf{R}^\shortmid\in (0,\infty)$ and $f\in \sF^\mathrm{VI}$.  Then it follows from \eqref{EQ3EXPLI},  \eqref{eq:constants}, \eqref{EQ3INF} and Lemma \ref{LM52} that 
\[
\liminf_{n\rightarrow\infty}\EE^n (f_n,f_n)\geq \liminf_{n\rightarrow\infty}\frac{1}{2}\int_{\Omega_n^c}|\nabla f_n|^2dx+\liminf_{n\rightarrow\infty} (J_{n1}+J_{n2})\geq \EE^\mathrm{VI}(f,f)
\]
with $\mu=M$,  $\ell=\lim_{n\rightarrow \infty}\ell_n$.  
\item[{\bf (M3)}] In this case $0<\sM<\infty,  \sR^\shortmid=\infty$,  $\ell_n\rightarrow \ell=\mathsf{M}\mathsf{R}^\shortmid=\infty$ and $f\in \sF^\mathrm{VII}$.  Similarly we have $\liminf_{n\rightarrow\infty}\EE^n (f_n,f_n)\geq \EE^{\mathrm{VII}}(f,f)$ with $\mu=M$.
\item[$\bf {(T3_1)}$] In this case $\ell_n\rightarrow \ell =\infty$ due to $\sM=\infty$ and $\sR^\shortmid>0$.  We first show $\gamma_\pm f=0$,  so that Lemma~\ref{LMB1} leads to $f\in\FF^{\mathrm{V}}$.   To do this,  note that $\sE^n(f_n,f_n)\leq M$.  It follows from \eqref{EQ3EXPLI} that  
\[
\int_{\mathbb{R}\times\mathbb{R}}\frac{\left((\tilde{f}_n(x_1+)-\tilde{f}_n(x_1'+)\right)^2+\left((\tilde{f}_n(x_1-)-\tilde{f}_n(x_1'-)\right)^2}{\left(\frac{2\ell_n}{\pi}\right)^2\left(\cosh\left(\frac{\pi}{2\ell_n}(x_1-x_1')\right)-1\right)}dx_1dx_1'\lesssim\frac{M}{(a^{\raisebox{0mm}{-}}_na^{\shortmid}_n)^{1/2}}
\]
Again from $\ell_n\rightarrow\infty$, \eqref{eq:constants} and Lemma \ref{LM52}, we obtain that 
\[
\int_{\mathbb{R}\times\mathbb{R}}\frac{\left(\gamma_\pm f(x_1)-\gamma_\pm f(x'_1)\right)^2}{(x_1-x'_1)^2}  dx_1dx_1' \lesssim \frac{M}{(a^{\raisebox{0mm}{-}}_na^{\shortmid}_n)^{1/2}}\rightarrow 0.
\]
Hence $\gamma_\pm f$ is a constant,  while $\gamma_\pm f \in H^{1/2}(\bR)$.  It must hold $\gamma_\pm f=0$.  Furthermore,  \eqref{EQ3INF} and $f\in \sF^\mathrm{V}$ imply that 
$$\liminf_{n\rightarrow\infty}\EE^n(f_n,f_n)\geq \frac{1}{2}\int_{\bG^2} |\nabla f|^2dx = \EE^\mathrm{V}(f,f).$$ 
\end{itemize}
Eventually the proof of the first part (a) of Mosco convergence is completed for the cases in $\mathcal{T}_1$. 
\end{proof}


Next we provide another tactic for the rest cases in $\mathcal{T}_2:= \{{\bf (N1)}, {\bf (T1)}, {\bf (M1)}, {\bf (T2)}, \bf {(T3_2)}\}$.  For all these cases,  it holds $\mathsf{R}^\shortmid=0$.  We still consider $f$ and $f_n$ in Lemma~\ref{LM34}.  The lemma below shows that $f\in H^1(\bR^2)$ for the cases under consideration.

\begin{lemma}\label{LM38}
Let $f$ and $f_n$ be in Lemma~\ref{LM34}, and assume $\sR^\shortmid=0$.  Then $f\in H^1(\bR^2)$.  
\end{lemma}
\begin{proof}
It suffices to show $\gamma_+f=\gamma_-f$ by Lemma~\ref{LMB1}.  To do this,  note that 
\begin{equation}\label{EQ3G1FU}
\begin{split}
\int_{\Omega_n}\left(\frac{\partial f_n}{\partial x_2}\right)^2dx_2dx_1
& \geq \frac{1}{2\varepsilon_n}\int_\mathbb{R}\left(\int_{-\varepsilon_n}^{\varepsilon_n}\frac{\partial f_n}{\partial x_2}dx_2\right)^2dx_1\\
& =\frac{1}{2\varepsilon_n}\int_\mathbb{R}\left(f_n(x_1,\varepsilon_n)-f_n(x_1, -\varepsilon_n) \right)^2dx_1\\
&= \frac{1}{2\varepsilon_n}\int_\mathbb{R}\left(\tilde{f_n}(x_1+)-\tilde{f_n}(x_1-) \right)^2dx_1.
\end{split}
\end{equation}
Thus
\begin{equation}\label{EQ3G1FU2}
\int_\mathbb{R}\left(\tilde{f_n}(x_1+)-\tilde{f_n}(x_1-) \right)^2dx_1\leq\frac{\varepsilon_n}{a^{\shortmid}_n}\cdot a^{\shortmid}_n\int_{\Omega_n}\left(\frac{\partial f_n}{\partial x_2}\right)^2dx\lesssim \frac{\varepsilon_n}{a^\shortmid_n} \sE^n(f_n,f_n)\leq  \frac{\varepsilon_n}{a^{\shortmid}_n} M\rightarrow 0,
\end{equation}
due to $\varepsilon_n/a^{\shortmid}_n\rightarrow \mathsf{R}^\shortmid=0$. Since a subsequence of $\{\gamma_{\pm}\tilde{f}_n\}$ converges to $\gamma_{\pm}f$ weakly in $L^2(\mathbb{R})$ by \eqref{eq:35},  the Fatou's lemma (see, e.g., \cite[Theorem 7 of \S8.2]{R88}) implies that 
\[
\int_\mathbb{R}\left(f(x_1+)-f(x_1-) \right)^2dx_1\leq \liminf_{n\rightarrow\infty}\int_\mathbb{R}\left(\tilde{f_n}(x_1+)-\tilde{f_n}(x_1-) \right)^2dx_1= 0.
\]
Therefore $\gamma_+ f=\gamma_-f$.  That completes the proof. 
\end{proof}

Another lemma introduces a pair of auxiliary functions that will play a crucial role in this tactic of the proof.

\begin{lemma}\label{LM39}
Let $f$ and $f_n$ be in Lemma~\ref{LM34}.  Set two functions on $\bR$:
\begin{equation}\label{eq:321}
F_n^+(\cdot):=\frac{1}{\varepsilon_n}\int_{0}^{\varepsilon_n}f_n(\cdot,x_2)dx_2,\quad  F^-_n(\cdot ):=\frac{1}{\varepsilon_n}\int^{0}_{-\varepsilon_n}f_n(\cdot,x_2)dx_2.\end{equation}
Then the following hold:
\begin{itemize}
\item[(1)] $F^\pm_n \in H^1(\bR)$ and 
\begin{equation}\label{EQ34L1}
\mathbf{D}(F^\pm_n,F^\pm_n):=\int_\mathbb{R}\left|\frac{dF^\pm_n}{d {x_1}}(x_1)\right|^2dx_1\leq \frac{a^{\raisebox{0mm}{-}}_n}{\sC^\shortline_n}\int_{\Omega^\pm_n}|\partial_{x_1} f_n|^2dx,
\end{equation}
where $\Omega^+_n:=\{(x_1,x_2)\in \Omega_n: x_1>0\}$ and $\Omega^-_n:=\{(x_1,x_2)\in \Omega_n: x_1<0\}$.
\item[(2)] Assume further $\sR^\shortmid=0$.  Then a subsequence of $\{F^\pm_n\}$ converges to $f|_\bR=\gamma_\pm f$ weakly in $L^2(\bR)$.   Particularly if $0<\sC^\shortline<\infty$ (resp. $\sC^\shortline=\infty$),  then $f|_\bR\in H^1(\bR)$ (resp.  $f|_\bR=0$).  
\end{itemize}
\end{lemma}
\begin{proof}
\begin{itemize}
\item[(1)] We only treat $F^+_n$.  Note that $F^+_n\in L^2(\bR)$ since
\begin{equation*}
\int_\mathbb{R}|F^+_n(x_1)|^2dx_1=\frac{1}{\varepsilon_n^2}\int_{\mathbb{R}}\left(\int_{0}^{\varepsilon_n}f_n(x_1,x_2)dx_2\right)^2dx_1\leq \frac{1}{\varepsilon_n}\int_{\Omega^+_n}|f_n|^2dx<\infty. 
\end{equation*}
Denote the weak derivative of $F^+_n$ by $F'^{+}_n$.  We claim that $F'^+_n\in L^2(\bR)$ so that $F^+_n\in H^1(\bR)$.  Indeed,  for any $g, h\in C_c^\infty(\bR)$ with $\text{supp}[h]\subset (0,\varepsilon)$,  
\[
	\int_{\Omega^+_n} \frac{\partial f_n}{\partial x_1}(x_1,x_2)g(x_1)h(x_2)dx_1dx_2=-\int_{\Omega^+_n} f_n(x_1,x_2)g'(x_1)h(x_2)dx_1dx_2.    
\]
Clearly $\frac{\partial f_n}{\partial x_1} g, f_ng'\in L^1(\Omega^+_n)$.   By letting $h\uparrow 1_{(0,\varepsilon)}$,  it follows from the dominated convergence theorem and Fubini theorem that 
\[
	\int_{\Omega^+_n} \frac{\partial f_n}{\partial x_1}(x_1,x_2)g(x_1)dx_1dx_2=-\int_{\Omega^+_n} f_n(x_1,x_2)g'(x_1)dx_1dx_2.
\]
Then the Fubini theorem indicates 
\begin{equation}\label{eq:322}
	\int_{\bR} \left(\frac{1}{\varepsilon_n}\int_0^\varepsilon \frac{\partial f_n}{\partial x_1}(x_1,x_2)dx_2\right)g(x_1)dx_1=-\int_\bR F^+_n(x_1)g'(x_1)dx_1,\quad \forall g\in C_c^\infty(\bR).
\end{equation}
Note that $G(x_1):=\frac{1}{\varepsilon_n}\int_0^\varepsilon \frac{\partial f_n}{\partial x_1}(x_1,x_2)dx_2\in L^2(\bR)$ due to
\begin{equation}\label{eq:323}
\frac{1}{\varepsilon_n^2}\int_{\mathbb{R}}\left(\int_{0}^{\varepsilon_n}\frac{\partial f_n}{\partial x_1} (x_1,x_2)dx_2\right)^2dx_1\leq \frac{1}{\varepsilon_n}\int_{\Omega^+_n}|\partial_{x_1} f_n|^2dx<\infty.
\end{equation}
As a consequence,  \eqref{eq:322} leads to $F'^+_n=G\in L^2(\bR)$.  Furthermore,  \eqref{eq:323} also implies 
\[
	\int_\bR F'^+_n(x_1)^2dx_2=\int_\bR G(x_1)^2dx_1\leq  \frac{a^{\raisebox{0mm}{-}}_n}{\sC^\shortline_n}\int_{\Omega^+_n}|\partial_{x_1} f_n|^2dx. 
\]
In other words,  \eqref{EQ34L1} holds. 
\item[(2)] We first prove that
\begin{equation}\label{eq:324}
	\|F^+_n-\gamma_+\tilde{f}_n\|^2_{L^2(\mathbb{R})}\lesssim \sR^\shortmid_n \cdot  a^{\shortmid}_n\int_{\Omega^+_n}\left(\frac{\partial f_n}{\partial x_2}\right)^2dx.
\end{equation}
When $f_n\in C_c^\infty(\bR^2)$,  we have $\gamma_+\tilde{f}_n=f_n(x_1, \varepsilon_n)$ and the left hand side of \eqref{eq:324} is equal to
\[
\begin{split}
\int_{\mathbb{R}}\left(\frac{1}{\varepsilon_n}\int_{0}^{\varepsilon_n}\left[f_n(x_1,\varepsilon_n)-f_n(x_1,x_2)\right]dx_2\right)^2dx_1 \lesssim \sR^\shortmid_n\cdot a^{\shortmid}_n\int_{\Omega^+_n}\left(\frac{\partial f_n}{\partial x_2}\right)^2dx.
\end{split}
\]
For a general $f_n\in H^1(\bR^2)$,  take a sequence $g^k\in C_c^\infty(\bR^2)$  such that $g^k$ converges to $f_n$ in $H^1(\bR^2)$ as $k\rightarrow \infty$.  Particularly,  $\frac{1}{\varepsilon_n}\int_0^{\varepsilon_n}g^k(\cdot, x_2)dx_2$ converges to $F^+_n$ and $g^k(\cdot, \varepsilon_n)$ converges to $\gamma_+ \tilde{f}_n$   both in $L^2(\bR)$.  Since \eqref{eq:324} holds for $g^k$ in place of $f_n$,  by letting $k\uparrow \infty$, we can obtain \eqref{eq:324} for this $f_n$.  Next, since $\sR^\shortmid_n\rightarrow 0$ and $a^{\shortmid}_n\int_{\Omega^+_n}\left(\frac{\partial f_n}{\partial x_2}\right)^2dx\leq 2\sE^n(f_n,f_n)\leq 2\sup_{n}\sE^n(f_n,f_n)<\infty$,  it follows from \eqref{eq:324} that $\|F^+_n-\gamma_+\tilde{f}_n\|^2_{L^2(\mathbb{R})}\rightarrow 0$.  Note that a subsequence of $\gamma_+\tilde{f}_n$ converges to $\gamma_+ f$ ($=f|_\bR$ due to Lemma~\ref{LM38} and Lemma~\ref{LMB1}) weakly in $L^2(\bR)$ by Lemma~\ref{LM34}.  Therefore a subsequence of $F^+_n$ converges to $\gamma_+f$ weakly in $L^2(\bR)$.  

When $\sC^\shortline>0$,  assume without loss of generality that $F^+_n$ converges to $\gamma_+f$ weakly in $L^2(\bR)$.  It follows from \eqref{eq:322} and the lower semi-continuity of $\mathbf{D}$ that 
\[
	\mathbf{D}(f|_\bR,f|_\bR)\leq \liminf_{n\rightarrow \infty} \mathbf{D}(F^+_n,F^+_n)\leq \liminf_{n\rightarrow \infty} \frac{1}{\sC^\shortline_n}\sE^n(f_n,f_n)\leq \frac{1}{\sC^\shortline}\sup_{n} \sE^n(f_n,f_n)<\infty. 
\]
Particularly $f|_\bR\in H^1(\bR)$.   When $\sC^\shortline=\infty$,  the above argument leads to $f|_\bR\in H^1(\bR)$ and $\mathbf{D}(f|_\bR,f|_\bR)=0$.  Hence $f|_\bR$ has to be the constant $0$.   
\end{itemize}
Eventually we complete the proof.
\end{proof}

Now we prove the first part (a) of Mosco convergence for the cases in $\mathcal{T}_2$.

\begin{proof}[of Mosco (a) for cases in $\mathcal{T}_2$]
Let $\{f_n\}\subset H$ that converges weakly to $f$ in $H$.  Without loss of generality, we may assume that 
\[\lim_{n\rightarrow \infty}\sE^n(f_n,f_n)\leq \sup_{n\geq 1}\EE^n(f_n,f_n)=:M<\infty. \]
We need to verify \eqref{eqMoscoa} for the cases in $\mathcal{T}_2$. 

In the cases {\bf (N1)}, {\bf (T1)} and {\bf (M1)},  it follows from $\mathsf{R}^\shortmid=0$, Lemma~\ref{LM38} and \eqref{EQ3INF} that $f\in H^1(\mathbb{R}^2)=\FF^\mathrm{I}$ and
\[
\liminf_{n\rightarrow\infty}\EE^n (f_n,f_n) \geq\frac{1}{2} \int_{\mathbb{G}^2}|\nabla f|^2dx=\frac{1}{2} \int_{\mathbb{R}^2}|\nabla f|^2dx=\EE^{\mathrm{I}}(f,f).
\]
To treat the cases {\bf (T2)} and $\bf {(T3_1)}$,  let $F^\pm_n$ be given by \eqref{eq:321}.  In what follows we prove them respectively.
\begin{itemize}
\item[{\bf (T2)}] In this case $\mathsf{R}^\shortmid=0$ and $\mathsf{C}^{\raisebox{0mm}{-}}\in (0,\infty)$.  Then Lemma~\ref{LM38} implies that $f\in H^1(\bR^2)$ and $f|_\bR\in H^1(\bR)$,  i.e. $f\in \sF^\mathrm{IV}$.  It follows from \eqref{eq:322} that 
\[
\begin{split}
\EE^n(f_n,f_n)
&\geq\frac{1}{2}\int_{\mathbb{G}^2} |\nabla \tilde{f}_n|^2dx+\dfrac{a^{\raisebox{0mm}{-}}_n}{2}\int_{\Omega_{n}} |\partial_{x_1} f_n|^2dx\\
&\geq \frac{1}{2}\int_{\mathbb{G}^2} |\nabla \tilde{f}_n|^2dx+\frac{\sC^\shortline_n}{2}\left(\mathbf{D}(F^+_n,F^+_n)+\mathbf{D}(F^-_n+F^-_n)\right)
\end{split}
\]
By means of \eqref{EQ3INF},  the lower semi-continuity of $\mathbf{D}$ and Lemma~\ref{LM39}~(2),  we obtain 
\[\liminf_{n\rightarrow\infty}\EE^n(f_n,f_n)\geq\frac{1}{2}\int_{\mathbb{R}^2}|\nabla f|^2dx+\mathsf{C}^{\raisebox{0mm}{-}}\mathbf{D}(f|_\bR,f|_\bR)=\EE^\mathrm{IV}(f,f)\]
with $\lambda=\mathsf{C}^{\raisebox{0mm}{-}}$. 
\item[$\bf {(T3_2)}$]  In this case $\mathsf{R}^\shortmid=0$ and $\mathsf{C}^{\raisebox{0mm}{-}}=\infty$.  Then Lemma~\ref{LM39}~(2) yields $f\in H^1(\bR^2)$ and $f|_\bR=0$,  i.e. $f\in \sF^\mathrm{V}$.  As a result, it follows from \eqref{EQ3INF} that 
\[
	\liminf_{n\rightarrow \infty}\sE^n(f_n,f_n)\geq \frac{1}{2}\int_{\bG^2}|\nabla f|^2dx=\frac{1}{2}\int_{\bR^2}|\nabla f|^2dx=\sE^\mathrm{V}(f,f).
\]
\end{itemize}
Therefore we complete the proof of Mosco (a) for the cases in $\mathcal{T}_2$.  
\end{proof}



%

\subsubsection{Proof of the second part of Mosco convergence}

Regarding the second part (b) of Mosco convergence,  we first treat the cases in $$\mathcal{T}_3:=\{\bf {(N2)}, \bf {(N3)},  \bf {(T2)}, \bf {(T3)},  \bf {(M2)},  \bf {(M3)}\}.$$ Consider $g\in H^1(\mathbb{G}^2)$ and define another function $g_n$ for any $n\geq 1$: 
\begin{equation}\label{EQ3UPG}
g_n(x_1,x_2):=\left\lbrace 
\begin{aligned}
&g(x_1,x_2-\varepsilon_n),\quad x_2>\varepsilon_n, \\
& \mathbf{H}(\gamma_{\pm}g)\left(x_1,\sqrt{\frac{a^{\raisebox{0mm}{-}}_n}{a^{\shortmid}_n}}x_2\right),\quad |x_2|\leq \varepsilon_n, \\
&g(x_1,x_2+\varepsilon_n),\quad x_2<-\varepsilon_n,
 \end{aligned}
 \right.
\end{equation}
where $\mathbf{H}$ is defined as \eqref{EQ3EXPE} with the boundary values $\gamma_\pm g$ on $\partial \Omega_{\ell_n}$. 

\begin{lemma}\label{LM37}
Let $g\in H^1(\bG^2)$ and $g_n$ be given by \eqref{EQ3UPG}.  Then $g_n\in H^1(\mathbb{R}^2)$ and $\|g_n-g\|_H\rightarrow 0$ as $n\rightarrow \infty$.  
\end{lemma}
\begin{proof}
Note that $g|_{\Omega^{c\pm}_n}\in H^1(\Omega^{c\pm}_n)$ and $g|_{\Omega_n}\in H^1(\Omega_n)$,  where $\Omega^{c\pm}_n:=\{x=(x_1,x_2): \pm x_2>\varepsilon_n\}$,  and the traces of $g|_{\Omega^{c\pm}_n}$ on the boundary coincide with the traces of $g|_{\Omega_n}$ on the upper and lower boundaries respectively.  
Mimicking the proof of Lemma~\ref{LMB1}~(2),  we can obtain that $g_n\in H^1(\bR^2)$.   

To prove  $\|g_n-g\|_H\rightarrow 0$, we need only to show $\int_{\Omega_n}|g_n|^2dx\rightarrow 0$,  analogical to the proof of \eqref{EQ3UGCON}. 
For convenience, set $h_n(x_1,x_2):=\mathbf{H}(\gamma_{\pm}g)\left(x_1,x_2\right)$.  Then
\begin{equation}\label{eq:313}
\int_{\Omega_n}|g_n|^2dx=\sqrt{\frac{a^{\shortmid}_n}{a^{\raisebox{0mm}{-}}_n}}\int_{\Omega_{\ell_n}}|h_n|^2dx.
\end{equation}
In addition, define
\[\tilde{h}_n(x_1,x_2):=h_n\left(\frac{2\ell_n}{\pi}x_1, \frac{2\ell_n}{\pi}x_2-\ell_n\right),\]
which is harmonic in the interior of $T:=\mathbb{R}\times[0,\pi]$ with the boundary values $\gamma_{\pm}^T\tilde{h}_n(x_1):=\gamma_{\pm}g\left(\frac{2\ell_n}{\pi}x_1\right)$.  Here $\gamma^T_\pm \tilde{h}_n$ stands for the traces of $\tilde{h}_n$ on the upper/lower boundary of $T$. A straightforward computation yields
\begin{equation}\label{eq:314}
\int_{\Omega_{\ell_n}}|h_n|^2dx=\left(\frac{2\ell_n}{\pi}\right)^2\int_{T}\tilde{h}_n^2dx.
\end{equation}
Using the Possion kernel of $T$ (see Appendix \ref{APB}), we can write $\tilde{h}_n$ explicitly as
\[
\tilde{h}_n(x_1, x_2)=\left(P_{x_2}\ast\gamma^T_-\tilde{h}_n\right)(x_1)+\left(P_{\pi-x_2}\ast\gamma^T_+\tilde{h}_n\right)(x_1),
\]
where
\[P_{x_2}(x_1)=\frac{1}{2\pi}\frac{\sin x_2}{\cosh x_1-\cos x_2}.\]
We have 
\[\|P_{x_2}\|_{L^1(\bR)}=1-\frac{x_2}{\pi}.\]
It follows from the Young's inequality for convolutions that for a fixed $x_2$, 
\[
\|\tilde{h}_n(\cdot, x_2)\|^2_{L^2(\bR)}\lesssim \left(1-\frac{x_2}{\pi}\right)^2\|\gamma^T_-\tilde{h}_n\|^2_{L^2(\mathbb{R})}+\left(\frac{x_2}{\pi}\right)^2\|\gamma^T_+\tilde{h}_n\|^2_{L^2(\mathbb{R})},
\]
and hence
\[
\int_{T}\tilde{h}_n(x_1,x_2)^2dx_1dx_2\lesssim \|\gamma^T_-\tilde{h}_n\|^2_{L^2(\mathbb{R})}+\|\gamma^T_+\tilde{h}_n\|^2_{L^2(\mathbb{R})}=\frac{\pi}{2\ell_n}\left(\|\gamma_-g\|^2_{L^2(\mathbb{R})}+\|\gamma_+g\|^2_{L^2(\mathbb{R})}\right). 
\]
Together with \eqref{eq:313} and \eqref{eq:314}, we eventually conclude that 
\[
\int_{\Omega_n}|g_n|^2dx\lesssim \varepsilon_n\left(\|\gamma_-g\|^2_{L^2(\mathbb{R})}+\|\gamma_+g\|^2_{L^2(\mathbb{R})}\right)\rightarrow 0.
\]
Therefore $\|g_n-g\|_H\rightarrow 0$ as $n\rightarrow \infty$.  
That completes the proof. 
\end{proof}

Now we proceed to prove the second part (b) of Mosco convergence for the cases in $\mathcal{T}_3$.   

\begin{proof}[of Mosco (b) for cases in $\mathcal{T}_3$]
Let $g\in H^1(\bG^2)$ and $g_n$ be given by \eqref{EQ3UPG}.  We need to verify \eqref{eqA2}.  To accomplish this,  note that
\begin{equation}\label{eq:315}
	\EE^n(g_n,g_n)
=\frac{1}{2}\int_{\mathbb{G}^2}|\nabla g|^2dx+\dfrac{a^{\raisebox{0mm}{-}}_n}{2}\int_{\Omega_n} |\partial_{x_1} g_n|^2dx+\dfrac{a^{\shortmid}_n}{2}\int_{\Omega_n} |\partial_{x_2} g_n|^2dx. 
\end{equation}
An analogical formulation to Lemma~\ref{LM35} yields that 
\begin{equation}\label{eq:316}
	\dfrac{a^{\raisebox{0mm}{-}}_n}{2}\int_{\Omega_n} |\partial_{x_1} g_n|^2dx+\dfrac{a^{\shortmid}_n}{2}\int_{\Omega_n} |\partial_{x_2} g_n|^2dx=J_{n3}+J_{n4},
\end{equation}
where 
\[
J_{n3}:= \frac{(a^{\raisebox{0mm}{-}}_na^{\shortmid}_n)^{1/2}}{4 \ell_n}\int_{\mathbb{R}\times\mathbb{R}}\frac{\left(g(x_1+)-g(x_1'-)\right)^2}{2\left(\frac{2\ell_n}{\pi}\right)\left(\cosh\left(\frac{\pi}{2\ell_n}(x_1-x_1')\right)+1\right)}dx_1dx_1'
\]
and 
\[
J_{n4}:=\frac{(a^{\raisebox{0mm}{-}}_na^{\shortmid}_n)^{1/2}}{8\pi}\int_{\mathbb{R}\times\mathbb{R}}\frac{\left((g(x_1+)-g(x_1'+)\right)^2+\left((g(x_1-)-g(x_1'-)\right)^2}{\left(\frac{2\ell_n}{\pi}\right)^2\left(\cosh\left(\frac{\pi}{2\ell_n}(x_1-x_1')\right)-1\right)}dx_1dx_1'. 
\]
Further note that 
\begin{equation}\label{eq:317}
\frac{(a^{\raisebox{0mm}{-}}_na^{\shortmid}_n)^{1/2}}{ \ell_n}=\frac{a^{\shortmid}_n}{\varepsilon_n}\rightarrow \frac{1}{\mathsf{R}^\shortmid}, \quad (a^{\raisebox{0mm}{-}}_na^{\shortmid}_n)^{1/2}\rightarrow \mathsf{M}.
\end{equation}
In what follows, we prove Mosco (b) for the cases in $\mathcal{T}_3$:
\begin{itemize}
\item[{\bf (N2)}] In this case,  $\mathsf{M}=0$, $0<\mathsf{R}^\shortmid<\infty, \ell_n\rightarrow 0$ and $g\in \sF^\mathrm{II}$.  It follows from Lemma \ref{LM52} that $J_{n4}\rightarrow 0$, and 
\[J_{n3}\rightarrow\frac{1}{4\mathsf{R}^\shortmid}\int_\mathbb{R}\left(g(x_1+)-g(x_1-)\right)^2dx_1.\]
Hence
\[
\lim_{n\rightarrow\infty}\EE^n(g_n,g_n)=\frac{1}{2}\int_{\mathbb{G}^2}|\nabla g|^2dx+\lim_{n\rightarrow\infty}J_{n3}=\EE^{\mathrm{II}}(g,g)
\]
with $\kappa=1/\mathsf{R}^\shortmid$.

\item[{\bf (N3)}] In this case  $\mathsf{M}=0$,  $\mathsf{R}^\shortmid=\infty$ and $g\in \sF^\mathrm{III}$.  Clearly we have 
$J_{n3},  J_{n4}\rightarrow 0$ no matter $\lim_{n\rightarrow \infty}\ell_n$ is finite or not,  and thus 
\[\lim_{n\rightarrow\infty}\EE^n(g_n,g_n)=\frac{1}{2}\int_{\mathbb{G}^2}|\nabla g|^2dx=\EE^\mathrm{III}(g,g).\]
\item[$\bf {(T2)}$] In this case,  we have $0<\sC^\shortline<\infty,  \sR^\shortmid=0$ and $\ell_n\rightarrow 0$.  When $g\notin \sF^\mathrm{IV}$,  \eqref{eqA2} trivially holds.  When $g\in \sF^\mathrm{IV}\subset H^1(\bR^2)$,  since $g|_{\bR}\in H^1(\bR)$,  it follows from Corollary~\ref{LMD3} and Remark~\ref{RMD4} that   
\[
	J_{n3}= \frac{\sC_n^\shortline}{\pi^2}  \int_\bR |\widehat{g|_\bR}(\xi)|^2|\xi|^2d\xi \int_\bR \frac{y^2dy}{\cosh y +1}\cdot \frac{1-\cos (2\ell_n \xi y/\pi)}{(2\ell_n \xi y/\pi)^2}\rightarrow \frac{\sC^\shortline}{3}\cdot \mathbf{D}(g|_\bR, g|_\bR) 
\]
and analogically $J_{n4}\rightarrow \frac{2\sC^\shortline}{3}\cdot \mathbf{D}(g|_\bR, g|_\bR)$.  By means of \eqref{eq:315} and \eqref{eq:316},  we eventually conclude that
\[
	\lim_{n\rightarrow \infty} \sE(g_n,g_n)=\int_{\bR^2}|\nabla g|^2dx+\sC^\shortline\cdot \mathbf{D}(g|_\bR,g|_\bR)=\sE^\mathrm{IV}(g,g). 
\]
\item[$\bf {(T3)}$] We only consider $g\in\FF^\mathrm{V}$.  Particularly, $\gamma_\pm g=0$, and thus $g_n\equiv0$ for $|x_2|\leq\varepsilon_n$.  As a result,
\[\lim_{n\rightarrow\infty}\EE^n(g_n,g_n)=\frac{1}{2}\int_{\mathbb{G}^2}|\nabla g|^2dx=\EE^\mathrm{V}(g,g). \]
\item[{\bf (M2)}]  In this case $0<\mathsf{M}<\infty$, $0<\mathsf{R}^\shortmid<\infty$, $\ell_n\rightarrow \ell= \mathsf{M}\mathsf{R}^\shortmid\in (0,\infty)$ and $g\in \sF^\mathrm{VI}$.  Then it follows from \eqref{eq:315}, \eqref{eq:316}, \eqref{eq:317} and Lemma~\ref{LM52} that 
\[
\lim_{n\rightarrow\infty}\EE^n(g_n,g_n)=\EE^{\mathrm{VI}}(g,g)
\]
with $\mu=\mathsf{M}$ and $\ell =\mathsf{M}\mathsf{R}^\shortmid$.
\item[{\bf (M3)}]  In this case $0<\mathsf{M}<\infty$, $\mathsf{R}^\shortmid=\infty$, $\ell_n\rightarrow\infty$ and $g\in \sF^\mathrm{VII}$.  Analogical to the case \textbf{(M2)},  we can obtain that
\[
\lim_{n\rightarrow\infty}\EE^n(g_n,g_n)=\EE^{\mathrm{VII}}(g,g),
\]
with $\mu=\mathsf{M}$.
\end{itemize}
Eventually we complete the proof.
\end{proof}

Let us turn to deal with the case {\bf (N1)} (also {\bf (T1)}, {\bf (M1)}).
For every function $g\in C_c^\infty(\bR^2)$,  define $g_n$ as follows: 
\begin{equation}\label{EQ33MU}
g_n(x_1,x_2)=\begin{cases}
g(x_1,x_2-\varepsilon_n), \quad & x_2>\varepsilon_n,\\
g(x_1, 0), \quad & |x_2|\leq \varepsilon_n, \\
g(x_1,x_2+\varepsilon_n), \quad & x_2<-\varepsilon_n.
\end{cases}
\end{equation}
The following lemma shows that $g_n$ converges to $g$ in $H^1(\bR^2)$. 

\begin{lemma}\label{LM310} 
Let $g\in C_c^\infty(\bR^2)$ and $g_n$ be defined as \eqref{EQ33MU}.  Then $g_n\in H^1(\bR^2)$ and $\|g_n-g\|_{H^1(\bR^2)}\rightarrow 0$ as $n\rightarrow \infty$.  
\end{lemma}
\begin{proof}
The fact $g_n\in H^1(\bR^2)$ can be verified by an analogical argument to Lemma~\ref{LMB1}.  To prove $\|g_n-g\|_{H^1(\bR^2)}\rightarrow 0$,
by mimicking the proof of \eqref{EQ3UGCON},  it suffices to show
\[
	\int_{\Omega_n} \left(|g_n|^2 + |\nabla g_n|^2\right)dx\rightarrow 0.  
\]
Indeed,  set $h(\cdot):=g(\cdot, 0)\in C_c^\infty(\bR)$ and we have
\[
	\int_{\Omega_n} \left(|g_n|^2 + |\nabla g_n|^2\right)dx=2\varepsilon_n \int_\bR \left(h(x_1)^2+h'(x_1)^2\right)dx_1\rightarrow 0.
\]
That completes the proof.  
\end{proof}

Finally we complete this proof as below.  

\begin{proof}[of Mosco (b) for \textbf{(N1)},  \textbf{(T1)} and \textbf{(M1)}]
In these cases,  $\sC^\shortline_n\rightarrow \mathsf{C}^{\raisebox{0mm}{-}}=0$.  Set
\[
\mathscr{C}:=\{g_n:  g_n\text{ is defined as \eqref{EQ33MU} for some }g\in C_c^\infty(\bR^2)\text{ and }n\geq 1\}.  
\]	
Then Lemma~\ref{LM310} indicates that $\mathscr{C}$ is dense in $H^1(\bR^2)$.  We further assert that for any $g\in \mathscr{C}$,  
\[
	\lim_{n\rightarrow \infty} \sE^n(g,g)=\sE^\mathrm{I}(g,g).  
\]
To do this, it suffices to verify that 
\begin{equation}\label{eq:326}
	a^\shortline_n \int_{\Omega_n} |\partial_{x_1} g|^2dx+a^\shortmid_n\int_{\Omega_n} |\partial_{x_2} g|^2dx\rightarrow 0.
\end{equation}
Note that $g(x_1, x_2)=g(x_1,0)$ for $|x_2|<\delta$ and some $\delta>0$,  and $h(\cdot):=g(\cdot, 0)\in C_c^\infty(\bR)$.  Then when $n$ is large enough,  the left hand side of \eqref{eq:326} is equal to
\[
	2a^\shortline_n \varepsilon_n \int_\bR h'(x_1)^2dx_1=2\sC^\shortline_n  \int_\bR h'(x_1)^2dx_1 \rightarrow 0.  
\]

With these two facts at hand,  we proceed to prove the second part (b) of Mosco convergence.  When $g\notin H^1(\bR^2)$,   \eqref{eqA2} trivially holds for any sequence converging to $g$ strongly in $H$.  Now we consider $g\in H^1(\bR^2)$.  Since $\mathscr{C}$ is dense in $H^1(\bR^2)$,  we can take a sequence $\{v_k\}\subset \mathscr{C}$ such that 
\[
	\lim_{k\rightarrow \infty} \sE^\mathrm{I}(v_k,v_k)=\sE^\mathrm{I}(g,g). 
\]
Since $\lim_{n\rightarrow \infty}\sE^n(v_k,v_k)=\sE^\mathrm{I}(v_k,v_k)$,  there exists $n_k$ such that for any $n\geq n_k$, 
\[
	\left|\sE^n(v_k,v_k)-\sE^\mathrm{I}(v_k,v_k)\right|<2^{-k}. 
\]
Put $g_1=\cdots =g_{n_1-1}=0$ and for $k\geq 1$,  $g_{n_k}=\cdots =g_{n_{k+1}-1}=v_k$.  Then clearly $\{g_n\}$ forms a sequence in $\mathscr{C}$ converging to $g$ strongly in $H$,  and 
\begin{equation}\label{eq328}
	\lim_{n\rightarrow \infty}\sE^n(g_n,g_n)=\sE^\mathrm{I}(g,g).
\end{equation}
That completes the proof.
\end{proof}

\subsection{Some consequences}

In this subsection, we present some consequences of Theorem~\ref{THMAIN}.  Since the sequence $\varepsilon_n\downarrow 0$ can be chosen arbitrarily, we say $(\sE^\varepsilon,\sF^\varepsilon)$ (as $\varepsilon\downarrow 0$), instead of $(\sE^{\varepsilon_n},\sF^{\varepsilon_n})$ (as $n\uparrow \infty$), manifests a phase transition or converges to a certain Dirichlet form in a little abuse of terminology.  

At first,  we emphasize that the limiting phase is the trivial one if $\sR^\shortmid=\sC^\shortline=0$, no matter the limit $\sM=\lim_{\varepsilon\downarrow 0} \sM_\varepsilon$ exists or not.  This fact can be easily figured out in the proof of Theorem~\ref{THMAIN}.  

\begin{corollary}
Assume that $\lim_{\varepsilon\downarrow 0}\sR^\shortmid_\varepsilon=\lim_{\varepsilon\downarrow 0} \sC^\shortline_\varepsilon=0$.  Then $(\sE^\varepsilon, \sF^\varepsilon)$ converges to $(\sE^\mathrm{I}, \sF^\mathrm{I})$ in the sense of Mosco. 
\end{corollary}
\begin{remark}
More precisely,  $\sR^\shortmid=0$ implies the first part (a) of Mosco convergence and $\sC^\shortline=0$ implies the second part (b) of Mosco convergence in this special case. 
\end{remark}

In the reminder of this subsection, we take $a^{\raisebox{0mm}{-}}_\varepsilon$ and $a^\shortmid_\varepsilon$ to be the monomials of $\varepsilon$:
\begin{equation}\label{eq:definitionofa}
	a^{\raisebox{0mm}{-}}_\varepsilon:= c^{\raisebox{0mm}{-}} \varepsilon^{\alpha},\quad a^\shortmid_\varepsilon:=c^\shortmid \varepsilon^{\beta},
\end{equation}
where $\alpha,\beta\in \bR$ and $c^{\raisebox{0mm}{-}}, c^\shortmid>0$.   The following corollary is a straightforward consequence of Theorem~\ref{THMAIN}. 

\begin{corollary}
Let $a^{\raisebox{0mm}{-}}_\varepsilon$ and $a^\shortmid_\varepsilon$ be given by \eqref{eq:definitionofa}. Then the following hold:
\begin{itemize}
\item[(1)] $(\sE^\varepsilon, \sF^\varepsilon)$ manifests a normal phase transition if and only if $\alpha+\beta>0$.  In this case,  the following convergences as $\varepsilon\downarrow 0$ hold in the sense of Mosco:
\begin{itemize}
\itemsep=0pt \parskip=0pt
\item $\beta<1$: $(\EE^\varepsilon,\FF^\varepsilon)$ converges to $(\EE^\mathrm{I}, \FF^\mathrm{I})$;
\item $\beta=1$: $(\EE^\varepsilon,\FF^\varepsilon)$ converges to $(\EE^\mathrm{II}, \FF^\mathrm{II})$ with $\kappa=c^\shortmid$;
\item $\beta>1$: $(\EE^\varepsilon,\FF^\varepsilon)$ converges to $(\EE^\mathrm{III}, \FF^{\mathrm{III}})$.
\end{itemize}
\item[(2)] $(\sE^\varepsilon, \sF^\varepsilon)$ manifests a tangent phase transition if and only if $\alpha+\beta<0$.  In this case,  the following convergences as $\varepsilon\downarrow 0$ hold in the sense of Mosco:
\begin{itemize}
\itemsep=0pt \parskip=0pt
\item $\alpha>-1$: $(\EE^\varepsilon,\FF^\varepsilon)$ converges to $(\EE^\mathrm{I}, \FF^\mathrm{I})$;
\item $\alpha=-1$: $(\EE^\varepsilon,\FF^\varepsilon)$ converges to $(\EE^\mathrm{IV}, \FF^\mathrm{IV})$ with $\lambda=c^{\raisebox{0mm}{-}}$;
\item $\alpha<-1$: $(\EE^\varepsilon,\FF^\varepsilon)$ converges to $(\EE^\mathrm{V}, \FF^{\mathrm{V}})$.
\end{itemize}
\item[(3)] $(\sE^\varepsilon, \sF^\varepsilon)$ manifests a mixing phase transition if and only if $\alpha+\beta=0$.  In this case,  the following convergences as $\varepsilon\downarrow 0$  hold in the sense of Mosco:
\begin{itemize}
\itemsep=0pt \parskip=0pt
\item $\alpha>-1$: $(\EE^\varepsilon,\FF^\varepsilon)$ converges to $(\EE^\mathrm{I}, \FF^\mathrm{I})$;
\item $\alpha=-1$: $(\EE^\varepsilon,\FF^\varepsilon)$ converges to $(\EE^\mathrm{VI}, \FF^\mathrm{VI})$ with $\mu=\sqrt{c^{\raisebox{0mm}{-}} c^\shortmid}$ and $\ell=\sqrt{\frac{c^{\raisebox{0mm}{-}}}{c^\shortmid}}$;
\item $\alpha<-1$: $(\EE^\varepsilon,\FF^\varepsilon)$ converges to $(\EE^\mathrm{VII}, \FF^{\mathrm{VII}})$ with $\mu=\sqrt{c^{\raisebox{0mm}{-}} c^\shortmid}$.
\end{itemize}
\end{itemize} 
\end{corollary}

Finally, we consider the special case $\alpha=\beta$, i.e. $a^{\raisebox{0mm}{-}}_\varepsilon:= c^{\raisebox{0mm}{-}} \varepsilon^{\alpha}, a^\shortmid_\varepsilon:=c^\shortmid \varepsilon^{\alpha}$, where the phases of type VI and  type VII do not appear. In other words, it links the normal phase transition with the tangent one. The proof of the following corollary is also straightforward and we omit it.

\begin{corollary}\label{COR32}
Let $a^{\raisebox{0mm}{-}}_\varepsilon$ and $a^\shortmid_\varepsilon$ be given by \eqref{eq:definitionofa} with $\alpha=\beta$.  Then the following convergences as $\varepsilon\downarrow 0$ hold in the sense of Mosco:
\begin{itemize}
\itemsep=0pt \parskip=0pt
\item[(1)] $\alpha>1$: $(\EE^\varepsilon,\FF^\varepsilon)$ converges to $(\EE^\mathrm{III}, \FF^\mathrm{III})$;
\item[(2)] $\alpha=1$: $(\EE^\varepsilon,\FF^\varepsilon)$ converges to $(\EE^\mathrm{II}, \FF^\mathrm{II})$ with $\kappa=c^\shortmid$;
\item[(3)] $-1<\alpha<1$: $(\EE^\varepsilon,\FF^\varepsilon)$ converges to $(\EE^\mathrm{I}, \FF^\mathrm{I})$;
\item[(4)] $\alpha=-1$: $(\EE^\varepsilon,\FF^\varepsilon)$ converges to $(\EE^\mathrm{IV}, \FF^\mathrm{IV})$ with $\lambda=c^{\raisebox{0mm}{-}}$;
\item[(5)]  $\alpha<-1$: $(\EE^\varepsilon,\FF^\varepsilon)$ converges to $(\EE^\mathrm{V}, \FF^{\mathrm{V}})$.
\end{itemize}
\end{corollary}

\subsection{Continuity of the phase transitions}\label{SEC34}

This subsection is devoted to showing the continuity of the phase transitions appearing in Theorem~\ref{THMAIN}.  Analogically to \cite[\S4.5]{LS19}, we will prove the continuity of related Dirichlet forms in the parameter $\sR^\shortmid$ or $\sC^{\raisebox{0mm}{-}}$ to accomplish it.  The normal case is under consideration at first.  

\begin{theorem}\label{THMCT1}
Consider the normal case $\sM=0$ in Theorem~\ref{THMAIN}. Set $(\cE^0,\cF^0):=(\sE^\mathrm{I}, \sF^\mathrm{I})$ and $(\cE^\infty, \cF^\infty):=(\sE^\mathrm{III}, \sF^\mathrm{III})$.  For $a\in (0,\infty)$, set $(\cE^a, \cF^a):=(\sE^\mathrm{II},\sF^\mathrm{II})$ with $\kappa:=1/a$.  Take a sequence $\{a^l:l\geq 1\}$ in $(0,\infty)$ such that $\lim_{l\rightarrow\infty}a^l=a\in [0, \infty]$. Then $(\cE^{a^l},\cF^{a^l})$ converges to $(\cE^a,\cF^a)$ in the sense of Mosco as $l\rightarrow \infty$. 
\end{theorem}
\begin{proof}
For the sake of brevity, write $(\cE^l, \cF^l)$ for $(\cE^{a^l},\cF^{a^l})$ in abuse of notations.  To prove the first part of Mosco convergence, 
suppose $\{u_l: l\geq 1\}$ converges weakly to $u$ in $H$, and thus there exists $K>0$ such that $\sup_l\|u_l\|^2_{H}\leq K$. Without loss of generality we can further assume that 
\begin{equation}\label{EQ4LEL}
	\lim_{l\rightarrow \infty}\mathcal{E}^l(u_l,u_l)\leq \sup_{l}\cE^l(u_l,u_l)=: M<\infty.
\end{equation}
These, together with the trace theorem, yield
 $$\sup_l \| \gamma_\pm u_l\|_{L^2(\bR)}^2 \lesssim  \sup_l\|u_l\|^2_{H^1(\mathbb{G}^2)}\leq \sup_l \|u_l\|_{H}^2+\sup_l \sE^l(u_l,u_l) \leq M+K.$$
 Consequently, mimicking the proof of Lemma~\ref{LM34}, $u\in H^1(\bG^2)$ and there exists a subsequence of $\{u_l: l\geq 1\}$, which is still denoted by $\{u_l: l\geq 1\}$, such that $\gamma_\pm u_l$ converges weakly to $\gamma_\pm u$ in $L^2(\mathbb{R})$.  Now we prove \eqref{eqMoscoa} for the following cases respectively:
\begin{itemize}
\item[(1)] $a=\infty$. Since $u_l\in \cF^l=\cF^{\infty}=H^1(\mathbb{G}^2)$, we have
\begin{equation}\label{EQ4EUU}
	\cE^{\infty}(u,u)\leq \liminf_{l\rightarrow \infty} \cE^{\infty}(u_l,u_l)\leq \liminf_{l\rightarrow\infty} \cE^l(u_l,u_l). 
\end{equation}

\item[(2)] $a\in (0,\infty)$.  We do not lose a great deal by assuming $K>a$.  Then there exists some $N$ such that $a^l<K$ for all $l>N$. It follows that 
\[
\begin{split}
\frac{1}{4K}\sup_{l>N}\int_{\mathbb{R}}&\left(u_l(x_1+)-u_l(x_1-)\right)^2dx_1 \\
& \leq \sup_{l>N} \frac{1}{4\beta^l}\int_{\mathbb{R}}\left(u_l(x_1+)-u_l(x_1-)\right)^2dx_1 \leq \sup_{l>N}\cE^l(u_l,u_l)\leq M,
\end{split}
\]
which yields that $\sup_{l>N}\int_{\mathbb{R}}\left(u_l(x_1+)-u_l(x_1-)\right)^2dx_1\leq 4KM$. As a consequence,
\[
\begin{aligned}
\cE^{a}(u,u)&\leq \liminf_{l\rightarrow \infty}\cE^{a}(u_l,u_l) \\
&=\liminf_{l>N,l\rightarrow\infty} \left(\cE^l(u_l,u_l)+\left(\frac{1}{4a}-\frac{1}{4a^l}\right)\cdot \int_{\mathbb{R}}\left(u_l(x_1+)-u_l(x_1-)\right)^2dx_1\right)\\
&=\liminf_{l>N,l\rightarrow\infty} \cE^l(u_l,u_l) \\
&=\liminf_{l\rightarrow\infty} \cE^l(u_l,u_l). 
\end{aligned}\]

\item[(3)] $a=0$. We know by \eqref{EQ4LEL} that 
\[
\sup_l\frac{1}{a^l}\int_\mathbb{R}\left(u_l(x_1+)-u_l(x_1-)\right)^2dx\leq 4M.
\]
Since $a^l\rightarrow 0$ and $\gamma_\pm u_l$ converge weakly to $\gamma_\pm u$ in $L^2(\mathbb{R})$, it follows from the Fatou's lemma (see, e.g., \cite[Theorem 7 of \S8.2]{R88})  that 
\[\int_\mathbb{R}\left(u(x_1+)-u(x_1-)\right)^2dx=0.\]
Thus $u\in\cF^0$, and 
\[
\cE^0(u,u)=\int_{\mathbb{R}^2}|\nabla u|^2dx\leq \liminf_{l\rightarrow\infty}\int_{\mathbb{G}^2}|\nabla u_l|^2dx\leq \liminf_{l\rightarrow\infty}\cE^l(u_l, u_l). 
\]

\end{itemize}

For the second part of Mosco convergence,  we only consider $u\in\cF^a$ for all $a\in[0,\infty]$.  Put $u_l:=u$.  It is easy to check that
\[
\lim_{l\rightarrow \infty}\cE^l(u_l,u_l)=\lim_{l\rightarrow \infty}\cE^l(u,u)=\cE^{a}(u,u). 
\] 
That completes the proof. 
\end{proof}

The following result is concerned with the tangent case.

\begin{theorem}\label{THM314}
Consider the tangent case $\sM=\infty$ in Theorem~\ref{THMAIN}. Set $(\cE^0,\cF^0):=(\sE^\mathrm{I}, \sF^\mathrm{I})$ and $(\cE^\infty, \cF^\infty):=(\sE^\mathrm{V}, \sF^\mathrm{V})$.  For $a\in (0,\infty)$, set $(\cE^a, \cF^a):=(\sE^\mathrm{IV},\sF^\mathrm{IV})$ with $\lambda:=a$.  Take a sequence $\{a^l:l\geq 1\}$ in $(0,\infty)$ such that $\lim_{l\rightarrow\infty}a^l=a\in [0, \infty]$. Then $(\cE^{a^l},\cF^{a^l})$ converges to $(\cE^a,\cF^a)$ in the sense of Mosco as $l\rightarrow \infty$. 
\end{theorem}
\begin{proof}
For the sake of brevity, write $(\cE^l, \cF^l)$ for $(\cE^{a^l},\cF^{a^l})$ in abuse of notations.  
To prove the first part of Mosco convergence, 
suppose $\{u_l: l\geq 1\}$ converges weakly to $u$ in $H$, and thus there exists $K>0$ such that $\sup_l\|u_l\|^2_{H}\leq K$. Without loss of generality we can further assume that 
\begin{equation}\label{EQ4LEL2}
	\lim_{l\rightarrow \infty}\cE^l(u_l,u_l)\leq \sup_{l}\cE^l(u_l,u_l)=: M<\infty.
\end{equation}
Mimicking the proof of Theorem~\ref{THMCT1}, we can obtain that $u\in H^1(\bR^2)$ and a subsequence of $\{u_l: l\geq 1\}$, which is still denoted by $\{u_l: l\geq 1\}$, satisfies that $ u_l|_\mathbb{R}$ converges weakly to $ u|_\mathbb{R}$ in $L^2(\mathbb{R})$.  In what follows we prove \eqref{eqMoscoa} for the three cases respectively:

\begin{itemize}
\item[(1)] $a=0$. Since $u_l\in \cF^l\subset H^1(\mathbb{R}^2)$, we have
\begin{equation}\label{EQ4EUU1}
	\cE^{0}(u,u)\leq \liminf_{l\rightarrow \infty} \cE^{0}(u_l,u_l)\leq \liminf_{l\rightarrow\infty} \cE^l(u_l,u_l). 
\end{equation}

\item[(2)] $a\in (0,\infty)$. It is similar to the case (2) in the proof of Theorem \ref{THMCT1}, and we omit it. 

\item[(3)] $a=\infty$. It suffices to show that $u\in\cF^{\infty}=\FF^{\mathrm{V}}$. 
Note that $a^l\rightarrow\infty$ and by \eqref{EQ4LEL2}, 
\begin{equation}\label{EQ42V}
\sup_l a^l\mathbf{D}(u_l|_\bR, u_l|_\bR) \leq M.
\end{equation}
Since $u_l|_\mathbb{R}$ converges weakly to $u|_\mathbb{R}$ in $L^2(\bR)$,  it follows from the lower semi-continuity of $\mathbf{D}$ that
\[
\mathbf{D}(u|_\bR, u|_\bR)\leq \liminf_{l\rightarrow\infty}\mathbf{D}(u_l|_\bR, u_l|_\bR)\leq \liminf_{l\rightarrow\infty}\frac{M}{a^l}=0.
\]
Hence $u|_\mathbb{R}\in H^1(\bR)$ is constant.  Therefore it must hold $u|_\bR=0$, which leads to $u\in\cF^{\infty}$. 
\end{itemize}

For the second part of Mosco convergence,  we only consider $u\in\cF^a$ for all $a\in[0,\infty]$.  When $a\in (0,\infty]$, put $u_l:=u$,  and it is easy to verify that
\begin{equation}\label{EQ42UP}
\lim_{l\rightarrow \infty}\cE^l(u_l,u_l)=\cE^{a}(u,u). 
\end{equation}
Regarding the case $a=0$,  note that for any $u\in C_c^\infty(\bR^2)$,  
\[
\cE^l(u_l,u_l)=\cE^0(u,u)+{a^l}\mathbf{D}(u(\cdot, 0), u(\cdot,0))\rightarrow \cE^0(u,u).
\]   
Then mimicking the argument deriving \eqref{eq328},  we can conclude that there exists a sequence $\{u_l\}$ converging to $u$ strongly in $H$ such that \eqref{EQ42UP} holds.  
That completes the proof. 
\end{proof}

Finally we prove the mixing case.

\begin{theorem}\label{THM33}
Consider the mixing case $0<\sM<\infty$ in Theorem~\ref{THMAIN}. Set $(\cE^0,\cF^0):=(\sE^\mathrm{I}, \sF^\mathrm{I})$ and $(\cE^\infty, \cF^\infty):=(\sE^\mathrm{VII}, \sF^\mathrm{VII})$ with $\mu=\sM$.  For $a\in (0,\infty)$, set $(\cE^a, \cF^a):=(\sE^\mathrm{IV},\sF^\mathrm{IV})$ with $\mu:=\sM$ and $\ell:=\sM a$.  Take a sequence $\{a^l:l\geq 1\}$ in $(0,\infty)$ such that $\lim_{l\rightarrow\infty}a^l=a\in [0, \infty]$. Then $(\cE^{a^l},\cF^{a^l})$ converges to $(\cE^a,\cF^a)$ in the sense of Mosco as $l\rightarrow \infty$. 
\end{theorem}
\begin{proof}
In abuse of notations, write $(\cE^l, \cF^l)$ for $(\cE^{a^l},\cF^{a^l})$.  
To prove the first part of Mosco convergence, 
suppose $\{u_l: l\geq 1\}$ converges weakly to $u$ in $H$, and thus there exists $K>0$ such that $\sup_l\|u_l\|^2_{H}\leq K$. Without loss of generality we can further assume that 
\begin{equation}\label{EQ4LEL1}
	\lim_{l\rightarrow \infty}\cE^l(u_l,u_l)\leq \sup_{l}\cE^l(u_l,u_l)=: M<\infty.
\end{equation}
By mimicking the proof of Theorem~\ref{THMCT1},  we can obtain that $u\in H^1(\bG^2)$ and there exists a subsequence of $\{u_l: l\geq 1\}$, which is still denoted by $\{u_l: l\geq 1\}$, such that $\gamma_\pm u_l$ converges weakly to $\gamma_\pm u$ in $H^{\frac{1}{2}}(\mathbb{R})$.
\begin{itemize}
\item[(1)] $a\in (0,\infty]$. From Lemma \ref{LM52} we directly have 
\[
\liminf_{l\rightarrow\infty}\cE^l(u_l,u_l)\geq \cE^a(u,u).
\]
\item[(2)] $a=0$. In this case we only need to prove that  $u\in\cF^0=H^1(\bR^2)$,  so that
\[
	\cE^0(u,u)=\frac{1}{2}\int_{\bG^2}|\nabla u|^2dx \leq \liminf_{l\rightarrow \infty} \frac{1}{2}\int_{\bG^2}|\nabla u_l|^2dx \leq \liminf_{l\rightarrow \infty} \cE^l(u_l,u_l).  
\]
In fact  the definition of $\cE^l$ yields
\[
\int_{\mathbb{R}\times\mathbb{R}}\frac{\left(\gamma_+u_l(x_1)-\gamma_-u_l(x_1')\right)^2}{2\left(\frac{2\sM a^l}{\pi}\right)\left(\cosh\left(\frac{\pi}{2\sM a^l}(x_1-x_1')\right)+1\right)}dx_1dx_1'\lesssim  a^l \cE^l(u_l,u_l)\leq a_l M
\]
Since $a^l\rightarrow 0$,  it follows from Lemma \ref{LM52} that 
\[\int_\mathbb{R}\left(\gamma_+u(x_1)-\gamma_-u(x_1)\right)^2dx_1\leq 0.\]
Therefore $\gamma_+u=\gamma_-u$ and $u\in\cF^0=H^1(\mathbb{R}^2)$.

\end{itemize}

For the second part of Mosco convergence,  suppose $v\in H^1(\mathbb{G}^2)$.  When $a\in (0,\infty]$,  put $v_l=v$.  Still from Lemma \ref{LM52} we  have 
\[
\lim_{l\rightarrow\infty}\cE^l(v_l,v_l)=\lim_{l\rightarrow\infty}\cE^l(v,v)= \cE^a(v,v).
\]
When $a=0$,  mimicking the proof of the same case in Theorem~\ref{THM314},  it suffices to show that for any $v\in C^\infty_c(\mathbb{R}^2)$,  $\lim_{l\rightarrow \infty}\cE^l(v,v)=\cE^0(v,v)$.  
This also amounts to that for $h(\cdot):=v(\cdot, 0)\in C_c^\infty(\bR)$,  
\[
	\lim_{l\rightarrow \infty} \frac{1}{\sM^2a_l^2}\int_{\mathbb{R}\times\mathbb{R}}\frac{\left(h(x_1)-h(x_1')\right)^2}{\cosh\left(\frac{\pi}{2\sM a_l}(x_1-x_1')\right)\pm 1}dx_1dx_1'=0.
\]
In practise,  it follows from Corollary~\ref{LMD3} and Remark~\ref{RMD4} that the left hand side is not greater than 
\[
	\lim_{l\rightarrow \infty} \frac{64 \sM a_l}{3\pi}\mathbf{D}(h,h)=0.  
\]
That completes the proof.
\end{proof}



\section{Boundary conditions of limiting flux at the barrier}\label{SEC4}

In this section, we turn to study the two-dimensional stiff problem in the context of heat equations.  Particularly, the boundary conditions satisfied by the limiting flux at the barrier will be derived. 

\subsection{Heat equation with the barrier $\Omega_\varepsilon$}\label{SEC41}

Recall that $\mathbf{A}_\varepsilon(x)$ is defined as $\text{Id}_2$ in $\Omega_\varepsilon$ and \eqref{eq:Avarepsilon} outside $\Omega_\varepsilon$.  What we are concerned with is the heat equation
\begin{equation}\label{eq:heatequationwithvarepsilon}
\begin{aligned}
	&\frac{\partial u^\varepsilon}{\partial t}(t,x)=\frac{1}{2}\nabla \cdot \left(\mathbf{A}_\varepsilon(x)\nabla u^\varepsilon(t,x) \right),\quad t\geq 0,x\in \bR^2, \\
	&u^\varepsilon(0,\cdot)=u_0. 
\end{aligned}
\end{equation}
The solution to \eqref{eq:heatequationwithvarepsilon}, also called the flux of related thermal conduction model, is considered to be a weak form as follows.

\begin{definition}\label{defofsolutionvare}
A function $u^\varepsilon\in \mathscr{H}(\bR^2):= C_b\left([0,\infty), L^2(\bR^2) \right) \cap L^\infty\left([0,\infty),H^1(\bR^2) \right)$ is called a weak solution to \eqref{eq:heatequationwithvarepsilon} if $u^\varepsilon(0,\cdot)=u_0$ and
\[
	\int_{\bR^2}\left(u_0(x)-u^\varepsilon(t,x) \right)g(x)dx=\frac{1}{2}\int_0^t\int_{\bR^2}\mathbf{A}_\varepsilon(x)\nabla u^\varepsilon(s,x)\cdot \nabla g(x)dxds
\]
for any $t>0$, $g\in C_c^\infty(\bR^2)$.
\end{definition}

It is well-known (see, e.g., \cite[Lemma~5.1]{LS19}) that for every $u_0\in H^1(\bR^2)$, \eqref{eq:heatequationwithvarepsilon} is well posed, i.e.  the weak solutions exist and are unique. Furthermore, the unique weak solution is 
\begin{equation}\label{eq:uvare}
	u^\varepsilon(t,x):=P^\varepsilon_tu_0(x)=\mathbf{E}^x u_0(X^\varepsilon_t),
\end{equation}
where $X^\varepsilon_t$ is the Markov process associated with $(\sE^\varepsilon, \sF^\varepsilon)$ and $(P^\varepsilon_t)$ is its semigroup. 

\subsection{Limiting flux}\label{SEC42}

Since $\mathbf{A}_\varepsilon(x)$ converges to $\mathrm{Id}_2$ a.e. as $\varepsilon\downarrow 0$,  the limiting flux $u$ of $u^\varepsilon$ is expected to be a certain solution to the heat equation 
\begin{equation}\label{eq:heatequation}
\begin{aligned}
	&\frac{\partial u}{\partial t}(t,x)=\frac{1}{2}\Delta u(t,x), \\
	&u(0,\cdot)=u_0.
\end{aligned}
\end{equation}
In an analogical sense of Definition~\ref{defofsolutionvare},  for every $u_0\in H^1(\bR^2)$, the weak solutions to \eqref{eq:heatequation} in $\mathscr{H}(\bR^2)$ exist and are obviously unique.  Indeed, the unique solution is nothing but 
\begin{equation}\label{eq:utBrownianmotion}
	u(t,x)=\mathbf{E}^x u_0(B_t),
\end{equation}
where $B_t$ is a two-dimensional Brownian motion.  However, we have obtained another six limits for $X^\varepsilon$ or $(\sE^\varepsilon, \sF^\varepsilon)$. To figure out the counterparts of these limits in terms of the heat equation \eqref{eq:heatequation}, we will restrict \eqref{eq:heatequation} to the space-time space removing the barrier, i.e. $(t,x)\in [0,\infty)\times \bR^2_0$, where $\bR^2_0=\{(x_1,x_2)\in \bR^2: x_2\neq 0\}$.  Recall that $H=L^2(\bR^2)=L^2(\bG^2)=L^2(\bR^2_0)$. 

\begin{definition}
$u\in C_b\left([0,\infty), H\right)$ is called a weak solution to \eqref{eq:heatequation} outside the barrier if $u(0,\cdot)=u_0$ and 
\begin{equation}\label{eq:solutionoutsidebarrier}
	\int_{\bR^2}\left(u(t,x)-u_0(x)\right)g(x)dx=\frac{1}{2}\int_0^t\int_{\bR^2_0}u(s,x)\Delta g(x)dxds
\end{equation}
for any $t>0$, $g\in C_c^\infty(\bR^2_0)$.
\end{definition}

The following theorem shows the existence of the limit $u$ of $u^\varepsilon$ under the same assumptions as Theorem~\ref{THMAIN} and obtains that $u$ is a weak solution to \eqref{eq:heatequation} outside the barrier. 

\begin{theorem}\label{THM43}
Take a decreasing sequence $\varepsilon_n\downarrow 0$. Assume that \eqref{eq:limitsexist} holds and $u_0\in H^1(\bR^2)$.  Let $u^{\varepsilon_n}$ be in \eqref{eq:uvare}, i.e. the unique weak solution to \eqref{eq:heatequationwithvarepsilon}, for $\varepsilon=\varepsilon_n$.  Then for every $t\geq 0$, the limit $u_t$ of $u^{\varepsilon_n}(t,\cdot)$ exists in $H$ as $n\rightarrow \infty$.  Furthermore, $u(t,x):=u_t(x)$ is a weak solution to \eqref{eq:heatequation} outside the barrier.
\end{theorem}
\begin{proof}
Note that Theorem~\ref{THMAIN} yields that $(\sE^n,\sF^n)$ converges to a Dirichlet form $(\sE,\sF)$, one of the Dirichlet forms $(\sE^\mathrm{I}, \sF^\mathrm{I}), \cdots, (\sE^\mathrm{VII}, \sF^\mathrm{VII})$, in the sense of Mosco.  Let $X$ be the associated Markov process of $(\sE,\sF)$ and $(P_t)$ be its semigroup.  Since $u_0\in H$, this Mosco convergence implies that for every $t\geq 0$,  $u^{\varepsilon_n}(t,\cdot)=P^{\varepsilon_n}_t u_0$ converges to $P_tu_0$ in $H$.  In other words, the limit $u_t$ exists and 
\begin{equation}\label{eq:ut}
	u_t=P_tu_0. 
\end{equation}

Now we show $u(t,x)$ is a weak solution to \eqref{eq:heatequation} outside the barrier.  Clearly, $t\mapsto \|P_tu_0\|_{H}$ is continuous and  $\|P_tu_0\|_{H}\leq \| u_0\|_{H}$. This indicates $u\in C_b\left([0,\infty),H \right)$.  Then it suffices to show \eqref{eq:solutionoutsidebarrier}. To accomplish it, fix $t>0$ and $g\in C_c^\infty(\bR^2_0)$.  There exists an integer $N$ large enough such that for any $n\geq N$,  $\text{supp}[g]\subset \Omega^c_{\varepsilon_n}$.  As a consequence, 
\[
	\int_{\bR^2}\left(u_0(x)-u^{\varepsilon_n}(t,x) \right)g(x)dx=\frac{1}{2}\int_0^t\int_{\bR^2}\nabla u^{\varepsilon_n}(s,x)\cdot \nabla g(x)dxds=-\frac{1}{2}\int_0^t\int_{\bR^2} u^{\varepsilon_n}(s,x)\Delta g(x)dxds. 
\]
Since $u^{\varepsilon_n}(t,\cdot)$ converges to $u(t,\cdot)$ in $H$, it follows that
\begin{equation}\label{eq:45}
	\lim_{n\rightarrow \infty}\int_{\bR^2}\left(u_0(x)-u^{\varepsilon_n}(t,x) \right)g(x)dx= \int_{\bR^2}\left(u_0(x)-u(t,x) \right)g(x)dx.
\end{equation}
Similarly we have $\lim_{n\rightarrow \infty}\int_{\bR^2} u^{\varepsilon_n}(s,x)\Delta g(x)dx=\int_{\bR^2} u(s,x)\Delta g(x)dx$ and note that
\[
\begin{aligned}
	\left| \int_{\bR^2} u^{\varepsilon_n}(s,x)\Delta g(x)dx\right|&\leq \|u^{\varepsilon_n}(s,\cdot)\|_{H} \|\Delta g\|_{H} \\
	&=\|P^{\varepsilon_n}_s u_0\|_{H} \|\Delta g\|_{H} \\
	&\leq \| u_0\|_{H} \|\Delta g\|_{H}.
\end{aligned}\]
Hence the bounded convergence theorem leads to
\begin{equation}\label{eq:46}
\lim_{n\rightarrow \infty}	-\frac{1}{2}\int_0^t\int_{\bR^2} u^{\varepsilon_n}(s,x)\Delta g(x)dxds= -\frac{1}{2}\int_0^t\int_{\bR^2} u(s,x)\Delta g(x)dxds.
\end{equation}
Eventually we can obtain \eqref{eq:solutionoutsidebarrier} by means of \eqref{eq:45} and \eqref{eq:46}. That completes the proof.
\end{proof}

Note that the limit $u_t$ in \eqref{eq:ut} varies for different cases appearing in Theorem~\ref{THMAIN}.  
In what follows, we will derive the  boundary conditions satisfied by the limiting flux at the barrier.
To do this, let us prepare some notations. For any $\alpha>0$, set 
\begin{equation}\label{eq:Ualpha}
	U_\alpha(\cdot):=\int_0^\infty \mathrm{e}^{-\alpha t} u_t(\cdot)dt.
\end{equation}
Note that $U_\alpha=R_\alpha u_0$ where $R_\alpha$ is the resolvent of the limiting Dirichlet form $(\sE,\sF)$ appearing in the proof of Theorem~\ref{THM43}. 
A function $u\in H^1_\Delta(\bG^2)$ is called to satisfy
\begin{itemize}
\item[(B.I)] The boundary condition of type I at the barrier, if
\[
	\gamma_+u_+=\gamma_-u_-,\quad  \left.\frac{\partial u_+}{\partial x_2}\right|_{x_2=0+}=\left.\frac{\partial u_-}{\partial x_2}\right|_{x_2=0-};
\]
\item[(B.II)] The boundary condition of type II (with the parameter $\kappa$) at the barrier, if
\[
\left.\frac{\partial u_{\pm}}{\partial x_2}\right|_{x_2=0\pm}=\frac{\kappa}{2}\left(u(x_1+)-u(x_1-)\right);
\]
\item[(B.III)] The boundary condition of type III at the barrier, if
\[
\left.\frac{\partial u_{\pm}}{\partial x_2}\right|_{x_2=0\pm}=0;
\]
\item[(B.IV)] The boundary condition of type IV at the barrier (with the parameter $\lambda$), if
\[
	\gamma_+u_+=\gamma_-u_-,\quad  \left.\frac{\partial u_+}{\partial x_2}\right|_{x_2=0+}-\left.\frac{\partial u_-}{\partial x_2}\right|_{x_2=0-}=-2\lambda\frac{d^2 (u|_\mathbb{R})}{d x_1^2};
\]
\item[(B.V)] The boundary condition of type V at the barrier, if
\[
	\gamma_+u_+=\gamma_-u_-=0;
\]
\item[(B.VI)] The boundary condition of type VI at the barrier (with the parameters $\mu, \ell$), if
\[
 \left.\frac{\partial u_{\pm}}{\partial x_2}\right|_{x_2=0\pm}=
 \frac{\mu}{2\pi}\int_\mathbb{R}\frac{\left(u(x_1+)-u(x_1-)\right)}{\left(\frac{2\ell}{\pi}\right)^2\left(\cosh\left(\frac{\pi}{2\ell}(x_1-x_1')\right)+1\right)}dx'_1\mp\frac{\mu}{4\pi}\check{\mathcal{L}}_\ell(\gamma_\pm u),
\]
where $\check{\mathcal{L}}_\ell$ is defined as \eqref{eq:Ll};
\item[(B.VII)]  The boundary condition of type VII at the barrier (with parameter $\mu$), if
\[
 \left.\frac{\partial u_{\pm}}{\partial x_2}\right|_{x_2=0\pm}=\mp\frac{\mu}{4\pi}\check{\mathcal{L}}(\gamma_\pm u),
\]
where $\check{\mathcal{L}}$ is defined as \eqref{eq:Ll2}.
\end{itemize}
Now we have a position to state the result concerning the boundary conditions. 

\begin{corollary}\label{COR44}
Under the same assumptions as Theorem~\ref{THM43}, let $U_\alpha$ be given by \eqref{eq:Ualpha} for any $\alpha>0$.  Then  $U_\alpha\in H^1(\bG^2)$ satisfies  (B.I) (resp.  (B.II), (B.III),  (B.IV), (B.V), (B.VI) or (B.VII)) when the trivial case (resp. \textbf{(N2)}, \textbf{(N3)}, \textbf{(T2)}, \textbf{(T3)}, \textbf{(M2)} or \textbf{(M3)}) in Theorem~\ref{THMAIN} appears.  Meanwhile if  if $u_0$ satisfies (B.I) (resp.  (B.II), (B.III),  (B.IV), (B.V), (B.VI) or (B.VII)), then so does $u(t,\cdot)$ for every $t>0$. 
\end{corollary}
\begin{proof}
Let $(\sE,\sF)$ and $(P_t)$ be in the proof of Theorem~\ref{THM43}. Further let $(\mathcal{L}, \mathcal{D}(\mathcal{L}))$ be the generator of $(\sE,\sF)$ on $H$.  Recall that the concrete expressions of $\mathcal{L}$ for all cases are presented in \S\ref{SEC2}; see Table~\ref{table1} for a summarization.  Then it suffices to note that $U_\alpha=R_\alpha u_0\in \mathcal{D}(\mathcal{L})$, since $u_0\in H$, and $u_0\in \mathcal{D}(\mathcal{L})$ leads to $u(t,\cdot)=P_tu_0\in \mathcal{D}(\mathcal{L})$ by virtue of the Hille-Yosida theorem.  That completes the proof.
\end{proof}

\begin{appendices}

\section{Mosco convergence of Dirichlet forms}\label{SECA1}

We shall use the Mosco convergence to describe the phase transition, and we write down its  definition for readers' convenience. Mosco convergence, first raised in \cite{U94},  is a kind of convergence for closed forms. Specifically,  let $(\EE^n,\FF^n)$ be a sequence of closed forms on a same Hilbert space $L^2(E,m)$, and $(\EE,\FF)$ be another closed form on $L^2(E,m)$. We always extend the domains of $\EE$ and $\EE_n$ to $L^2(E,m)$ by letting
\[
\begin{aligned}
	\EE(u,u)&:=\infty, \quad u\in L^2(E,m)\setminus \FF, \\ 
	\EE^n(u,u)&:=\infty,\quad u\in L^2(E,m)\setminus \FF^n.
\end{aligned}
\]
In other words, $u\in \FF$ (resp. $u\in \FF^n$) if and only if $\EE(u,u)<\infty$ (resp. $\EE^n(u,u)<\infty$). 
Furthermore, we say $u_n$ converges to $u$ weakly in $L^2(E,m)$, if for any $v\in L^2(E,m)$, $(u_n,v)_m\rightarrow (u,v)_m$ as $n\rightarrow \infty$, and (strongly) in $L^2(E,m)$, if $\|u_n-u\|_{L^2(E,m)}\rightarrow 0$. 

\begin{definition}\label{DEF41}
Let $(\EE^n,\FF^n)$ and $(\EE,\FF)$ be given above. Then $(\EE^n,\FF^n)$ is said to be convergent to $(\EE,\FF)$ in the sense of Mosco, if
\begin{itemize}
\item[(a)] For any sequence $\{u_n:n\geq 1\}\subset L^2(E,m)$ that converges weakly to $u$ in $L^2(E,m)$, it holds that
\begin{equation}\label{eqMoscoa}
	\EE(u,u)\leq \liminf_{n\rightarrow \infty}\EE^n(u_n,u_n). 
\end{equation}
\item[(b)] For any $u\in L^2(E,m)$, there exists a sequence $\{u_n:n\geq 1\}\subset L^2(E,m)$ that converges strongly to $u$ in $L^2(E,m)$ such that
\begin{equation}\label{eqA2}
	\EE(u,u)\geq \limsup_{n\rightarrow \infty}\EE^n(u_n,u_n). 
\end{equation}
\end{itemize}
\end{definition} 

Let $(T^n_t)_{t\geq 0}$ and $(T_t)_{t\geq 0}$ be the semigroups of $(\EE^n, \FF^n)$ and $(\EE,\FF)$ respectively, and $(G^n_\alpha)_{\alpha>0}, (G_\alpha)_{\alpha>0}$ be their corresponding resolvents.  The following result is well-known (see \cite{U94}).

\begin{proposition}\label{PRO42}
Let $(\EE^n, \FF^n), (\EE, \FF)$ be above. Then the following are equivalent:
\begin{itemize}
\item[(1)] $(\EE^n,\FF^n)$ converges to $(\EE,\FF)$ in the sense of Mosco;
\item[(2)] For every $t>0$ and $f\in L^2(E,m)$,  $T^n_tf$ converges to $T_tf$ strongly in $L^2(E,m)$; 
\item[(3)] For every $\alpha>0$ and $f\in L^2(E,m)$,  $G^n_\alpha f$ converges to $G_\alpha f$ strongly in $L^2(E,m)$. 
\end{itemize}
\end{proposition}
\begin{remark}\label{RMA3}
Particularly,  take $(\sE^n,\sF^n)=(\sE,\sF)$.  Then the Mosco convergence trivially holds, and \eqref{eqMoscoa} coincides with the lower semi-continuity of $\sE$: For any sequence $u_n$ converging weakly to $u$ in $L^2(E,m)$,  it holds that $\sE(u,u)\leq \liminf_{n\rightarrow \infty} \sE(u_n,u_n)$.  
\end{remark}

\section{Traces of functions in certain Sobolev spaces}\label{AP1}

Firstly, the trace of a function $u\in H^1(\bG^2_\pm)$ on the boundary is denoted by $u(\cdot \pm)$ or $\gamma_\pm u$.  Then the trace theorem (see, e.g.,  \cite[Theorem 1.63]{BCD11}) tells us that
\begin{equation}\label{EQ2TRH2}
\|\gamma_\pm u\|_{H^{\frac{1}{2}}(\mathbb{R})}\lesssim\|u\|_{H^1({\mathbb{G}^2_\pm})},
\end{equation}
where
\[
H^{\frac{1}{2}}(\mathbb{R}):=\left\{f\in L^2(\mathbb{R}):\|f\|^2_{H^{\frac{1}{2}}(\mathbb{R})}:=\int_\mathbb{R}(1+|\xi|)|\hat{f}(\xi)|^2d\xi<\infty\right\}
\]
and $\hat{f}$ is the Fourier transform of $f$.  In addition, the following integration by parts formula holds: For $u\in H^1(\bG^2_\pm)$ and $g\in C_c^\infty(\bG^2_\pm)$, 
\begin{equation}\label{eq:partsformula}
\begin{aligned}
	&\int_{\bG^2_\pm} \partial_{x_1} u \cdot g dx=-\int_{\bG^2_\pm} u\cdot \partial_{x_1} g  dx, \\
	&	\int_{\bG^2_\pm} \partial_{x_2} u \cdot g dx=-\int_{\bG^2_\pm} u\cdot \partial_{x_1} g  dx \mp \int_{\bR} \gamma_{\pm} u(x_1) g(x_1,0)dx_1. 
\end{aligned}\end{equation}

Next, consider the Dirichlet form $(\frac{1}{2}\mathbf{D}, H^1(\bR^2))$ of a Brownian motion on $\mathbb{R}^2$.  According to the inequality (see, e.g.,  \cite[Lemma 2.1.1]{F80})
\begin{equation}\label{traceineq}
\int_\mathbb{R}u(x_1,0)^2dx_1\leq C\mathbf{D}_1(u,u), \quad u\in C_c^\infty(\mathbb{R}^2),
\end{equation}
where $C$ is a positive constant independent of $u$,  we know that there exists a bounded linear operator 
\begin{equation}\label{eq:gamma}
\gamma: H^1(\mathbb{R}^2)\rightarrow L^2(\mathbb{R}),
\end{equation}
such that 
\[\gamma u(\cdot)=u|_\mathbb{R}(\cdot):=u(\cdot,0), \quad \text{if } u\in C_c^\infty(\mathbb{R}^2). \]
Throughout this paper,  $\gamma u$ is called the trace of $u$ on the $x_1$-axis and we also write $u|_\bR$ for $\gamma u$ if no confusions cause.  Note that for $u\in H^1(\bR^2)$, $u_\pm:=u|_{\bR^2_\pm}$ gives a function in $H^1(\bR^2_\pm)=H^1(\bG^2_\pm)$.  In the lemma below we conclude that $\gamma_\pm u_\pm$ coincides with $\gamma u$. 

\begin{lemma}\label{LMB1}
\begin{itemize}
\item[(1)] For any $u\in H^1(\bR^2)$, it holds that $\gamma u=\gamma_+ u_+=\gamma_- u_-$.  Particularly, $\gamma$ is also a bounded linear operator from $H^1(\bR^2)$ to $H^{1/2}(\bR)$. 
\item[(2)] If  $u\in H^1(\bG^2)$ and $\gamma_+ u_+=\gamma_- u_-$,  then $u\in H^1(\bR^2)$ and particularly $\gamma u=\gamma_+u_+=\gamma_- u_-$. 
\end{itemize}
\end{lemma}
\begin{proof}
\begin{itemize}
\item[(1)] We first note that for $f\in C_c^\infty(\bR^2)$,  $\gamma f=\gamma_+ f_+=\gamma_-f_-=f(\cdot, 0)$.  For $u\in H^1(\bR^2)$, take a sequence $u_n\in C_c^\infty(\bR^2)$ such that $u_n\rightarrow u$ in $H^1(\bR^2)$.  Consequently, $u_{n\pm}$ converges to $u_\pm$ in $H^1(\bG^2_\pm)$.  It follows from \eqref{EQ2TRH2} that $\gamma_\pm u_{n\pm}$ converges to $\gamma_\pm u_\pm$ in $H^{1/2}(\bR)$, and \eqref{eq:gamma} yields that $\gamma u_n$ converges to $\gamma u$ in $L^2(\bR)$.  Since $\gamma_\pm u_{n\pm}=\gamma u_n$, we can conclude that $\gamma u=\gamma_\pm u_\pm \in H^{1/2}(\bR)$.  Particularly,  $\gamma: H^1(\bR^2)\rightarrow H^{1/2}(\bR)$ is bounded because 
\[
	\|\gamma u\|_{H^{1/2}(\bR)}=\|\gamma_\pm u_\pm\|_{H^{1/2}(\bR)}\lesssim \|u_\pm\|_{H^1(\bG^2_\pm)}\leq \|u\|_{H^1(\bR^2)}. 
\]
\item[(2)] For any $\varphi\in C_c^\infty(\bR^2)$, the integration by parts formula \eqref{eq:partsformula}, together with $\gamma_+u_+=\gamma_-u_-$,  implies for $i=1,2$,
\[
	\int_{\bR^2_+} \partial_{x_i} u_+ \cdot \varphi dx+\int_{\bR^2_-} \partial_{x_i} u_-\cdot \varphi dx=-\int_{\bR^2} u \cdot \partial_{x_i} \varphi dx. 
\]
This leads to $\partial_{x_i} u=\partial_{x_i} u_++\partial_{x_i} u_-\in  L^2(\bR^2)$.  Hence $u\in H^1(\bR^2)$. 
\end{itemize}
That completes the proof.
\end{proof}



Finally we present the trace of the normal derivative of a function $u\in H^1_\Delta(\mathbb{G}^2_\pm)$ on the boundary.  Define a linear functional on $H^{1/2}(\bR)$:
\begin{equation}\label{eq:FUP}
F_u(\phi):=(\Delta u, \Phi)_{L^2(\mathbb{G}_\pm^2)}+(\nabla u,\nabla\Phi)_{L^2(\mathbb{G}_\pm^2)},\quad \phi\in H^{1/2}(\bR),
\end{equation}
where $\Phi\in H^1(\mathbb{G}^2_\pm)$ is an extension of $\phi$, i.e. $\gamma_\pm\Phi=\phi$.  

\begin{lemma}\label{LEMA2}
The functional $F_u$ is a bounded linear functional on $H^{1/2}(\mathbb{R})$. 
\end{lemma}
\begin{proof}
We first show $F_u$ is well-defined, in other words, $F_u(\phi)$ does not depend on the choice of $\Phi$. 
Indeed, suppose that  $\Phi_1,\Phi_2 \in H^1(\mathbb{G}^2_\pm)$ and $\gamma_\pm\Phi_1=\gamma_\pm\Phi_1=\phi$. Denote $\omega:=\Phi_1-\Phi_2$, and we have $\omega\in H_0^1(\mathbb{G}^2_\pm)$. From the classical Green-Gauss formula for $C_c^\infty$ functions and the fact that $C_c^\infty(\mathbb{G}^2_\pm)$ is dense in $H^1_\Delta(\mathbb{G}^2_\pm)$, we can conclude that $F_u(\omega)=0$ and thus $F_u(\phi)$ is independent of the choice of the extension of $\phi$.  Furthermore,  note that for every $\phi\in H^{1/2}(\mathbb{R})$, there exists an extension $\Phi$ such that (see, e.g., \cite[Lemma 7.41]{AF03})
\[
\|\Phi\|_{H^1(\mathbb{G}^2_\pm)}\lesssim \|\phi\|_{H^{\frac{1}{2}}(\mathbb{R})},
\]
 and 
\[
|F_u(\phi)|\lesssim \|u\|_{H^1_\Delta(\mathbb{G}^2_\pm)}\|\Phi\|_{H^1(\mathbb{G}^2_\pm)}.
\]
Eventually we can conclude that $F_u$ is a bounded linear functional on $H^{1/2}(\bR)$. 
\end{proof}

Applying the Riesz representation theorem to $F_u$, there exists $\gamma_\pm^{\partial_2} u\in H^{-1/2}(\mathbb{R})$ such that 
\[
F_u(\phi)=\mp \langle \gamma_\pm^{\partial_2} u, \phi\rangle.
\]
When $u\in C_c^\infty(\mathbb{G}^2_\pm)$, the classical Green-Gauss formula implies that 
\[\gamma_\pm^{\partial_2} u=\left.\frac{\partial u}{\partial x_2}\right|_{x_2=0\pm}.\]
Throughout this paper, we also use the notation $\left.\frac{\partial u}{\partial x_2}\right|_{x_2=0\pm}:=\gamma_\pm^{\partial_2} u$ to denote the trace of the normal derivative of $u\in H^1_\Delta(\mathbb{G}^2_\pm)$. Particularly,  \eqref{eq:FUP}, together with Lemma~\ref{LMB1}, leads to the following Green-Gauss formulae.

\begin{lemma}\label{LMB3}
For $u\in H^1_\Delta(\bG^2_\pm)$ and $\Phi\in H^1(\bG^2_\pm)$, it holds 
\begin{equation}\label{eq:GreenGaussformula}
(\nabla u, \nabla \Phi)_{L^2(\bG^2_\pm)}=-(\Delta u, \Phi)_{L^2(\bG^2_\pm)}\mp \left \langle \left.\frac{\partial u}{\partial x_2}\right|_{x_2=0\pm}, \gamma_\pm\Phi\right\rangle.
\end{equation}
Particularly, for $u\in H^1_\Delta(\bR^2)$ and $\Phi\in H^1(\bR^2)$, it holds 
\begin{equation}\label{eq:Green2}
	(\nabla u_\pm, \nabla \Phi_\pm)_{L^2(\bR^2_\pm)}=-(\Delta u_\pm, \Phi_\pm)_{L^2(\bR^2_\pm)}\mp \left \langle \left.\frac{\partial (u_\pm)}{\partial x_2}\right|_{x_2=0\pm}, \gamma\Phi\right\rangle.
\end{equation}
\end{lemma}

\section{Trace Dirichlet form on the boundary of a strip}\label{APB}

In this appendix we are to derive the traces of reflecting Brownian motions on certain strips on the boundary. Firstly, let us consider the closed strip $T:=\mathbb{R}\times[0,\pi]$ in $\mathbb{R}^2$, and define 
\[
\begin{aligned}
  \mathcal{G}&=H^1(T),\\
 \mathcal{A} (u,u) &=
 \frac{1}{2}\int_{T} |\nabla u|^2dx,\quad u\in \mathcal{G}. 
 \end{aligned}
\]
Then $(\mathcal{A}, \mathcal{G})$ is a regular Dirichlet form on $L^2(T)$ associated with the reflecting Brownian motion on $T$. 
From \cite{B70} and \cite{W61} we know that the Poisson kernel for $T$ is 
\[P(x,x')=\frac{1}{2\pi}\frac{\sin x_2}{\cosh(x_1-x_1')-\cos x_2}1_{\{x_2'=0\}}+\frac{1}{2\pi}\frac{\sin x_2}{\cosh(x_1-x_1')+\cos x_2}1_{\{x_2'=\pi\}}, \]
where $x=(x_1,x_2)\in \mathring{T}:=\bR\times (0,\pi)$, $x'=(x'_1,x'_2)\in \partial T:=\bR\times \{0,\pi\}$.  
In addition, \cite[Example 5.8.1]{CF12} indicates that the Feller kernel for $T$ is
\[U(x,x')=\frac{1}{2}\frac{\partial P(x,x')}{\partial \mathbf{n}_x}, \quad x,x'\in \partial T,\]
where $\mathbf{n}_x$ denotes that inward normal vector at $x$.  A straightforward computation yields 
\[
U(x,x')=\frac{1}{4\pi(\cosh(x_1-x_1')-1)}1_{\{x_2=x_2'\}}+\frac{1}{4\pi(\cosh(x_1-x_1')+1)}1_{\{x_2\neq x_2'\}}.
\]
Set $\check{\mathcal{G}}:=\{u|_{\partial T}: u\in \mathcal{G}\}=\{f\in L^2(\partial T): f_0,f_\pi \in H^{1/2}(\bR) \}$, where $u|_{\partial T}$ stands for the trace of $u\in \mathcal{G}$ on $\partial T$, $f_0(\cdot):=f(\cdot,0)$ and $f_\pi(\cdot):=f(\cdot, \pi)$.  Then the trace Dirichlet form (on $L^2(\partial T)$) of $(\mathcal{G},\mathcal{A})$ on $\partial T$  is $(\check{\mathcal{A}}, \check{\mathcal{G}})$: For every $f\in \check{\mathcal{G}}$, 
\[
\begin{split}
\check{\mathcal{A}}(f,f)
& =\frac{1}{4\pi}\int_{\mathbb{R}\times\mathbb{R}}\frac{\left(f_0(x_1)-f_\pi(x_1')\right)^2}{\cosh(x_1-x_1')+1}dx_1dx_1'\\
&\qquad \qquad+\frac{1}{8\pi}\int_{\mathbb{R}\times\mathbb{R}}\frac{\left(f_0(x_1)-f_0(x_1')\right)^2+\left(f_\pi(x_1)-f_\pi(x_1')\right)^2}{\cosh(x_1-x_1')-1}dx_1dx_1'.
\end{split}
\] 

Next we consider the Dirichlet form on $L^2(\bar{\Omega}_\ell)$ with a constant $\ell>0$:  
\[
\begin{aligned}
  \mathcal{G}^\ell &=H^1(\bar\Omega_\ell),\\
 \mathcal{A}^\ell (u,u)&=
 \frac{1}{2}\int_{\bar\Omega_\ell} |\nabla u|^2dx,\quad u\in \mathcal{G}^\ell,
 \end{aligned}
\]
where $\bar{\Omega}_\ell:=\bR\times [-\ell, \ell]$. 
After a spatial translation and scaling on $(\mathcal{A}, \mathcal{G})$, it is easy to obtain that the trace Dirichlet form (on $L^2(\partial \bar{\Omega}_\ell)$) of $(\mathcal{A}^\ell, \mathcal{G}^\ell)$ on the boundary $\partial \bar\Omega_\ell$ is
\begin{equation}\label{TRACEDF}
\begin{split}
\check{\mathcal{G}}^\ell &=\{u|_{\partial \bar{\Omega}_\ell}: u\in \mathcal{G}^\ell\}=\{f\in L^2(\partial \bar{\Omega}_\ell): f_\ell, f_{-\ell}\in H^{1/2}(\bR)\},\\
\check{\mathcal{A}}^\ell(f,f)
& =\frac{1}{4\pi}\int_{\mathbb{R}\times\mathbb{R}}\frac{\left(f_\ell(x_1)-f_{-\ell}(x_1')\right)^2}{\left(\frac{2\ell}{\pi}\right)^2\left(\cosh\left(\frac{\pi}{2\ell}(x_1-x_1')\right)+1\right)}dx_1dx_1'\\
&\qquad \qquad+\frac{1}{8\pi}\int_{\mathbb{R}\times\mathbb{R}}\frac{\left(f_\ell(x_1)-f_\ell(x_1')\right)^2+\left(f_{-\ell}(x_1)-f_{-\ell}(x_1')\right)^2}{\left(\frac{2\ell}{\pi}\right)^2\left(\cosh\left(\frac{\pi}{2\ell}(x_1-x_1')\right)-1\right)}dx_1dx_1',\quad f\in \check{\mathcal{G}}^\ell, 
\end{split}
\end{equation}
where $f_{\ell}(\cdot):=f(\cdot,\ell)$ and $f_{-\ell}(\cdot ):=f(\cdot,-\ell)$. 

\section{A useful lemma}

Consider the inner products $\check{\mathscr{A}}^{\ell, 1}$ and $\check{\mathscr{A}}^{\ell, 2}$  on $H^{\frac{1}{2}}(\mathbb{R})\times H^{\frac{1}{2}}(\mathbb{R})$ for $\ell\in (0,\infty)$ as follows: For $f=(f^+,f^-)\in H^{\frac{1}{2}}(\mathbb{R})\times H^{\frac{1}{2}}(\mathbb{R})$,
\[
\begin{split}
&\check{\mathscr{A}}^{\ell,1}(f,f):=\int_{\mathbb{R}\times\mathbb{R}}\frac{\left(f^+(x_1)-f^-(x_1')\right)^2}{2\ell\left(\cosh\left(\ell^{-1}(x_1-x_1')\right)+1\right)}dx_1dx_1', \\
&\check{\sA}^{\ell,2}(f,f):=\int_{\mathbb{R}\times\mathbb{R}}\frac{\left(f^+(x_1)-f^+(x_1')\right)^2+\left(f^-(x_1)-f^-(x_1')\right)^2}{\ell^2\left(\cosh\left(\ell^{-1}(x_1-x_1')\right)-1\right)}dx_1dx_1'.
\end{split}
\]
When $\ell=0$, set
\[
	\check{\sA}^{0,1}(f,f):=\int_\mathbb{R}\left(f^+(x_1)-f^-(x_1)\right)^2dx_1, \quad \check{\sA}^{0,2}(f,f)=0
\]
and when $\ell=\infty$, set
\[
\check{\sA}^{\infty,1}(f,f):=0,\quad \check{\sA}^{\infty,2}(f,f):=\int_{\mathbb{R}\times\mathbb{R}}\frac{\left(f^+(x_1)-f^+(x_1')\right)^2+\left(f^-(x_1)-f^-(x_1')\right)^2}{(x_1-x_1')^2}dx_1dx_1'. 
\]	
The following lemma is very useful in proving of the main theorems of this paper.


\begin{lemma}\label{LM52}
Let $f, f_n\in H^{\frac{1}{2}}(\mathbb{R})\times H^{\frac{1}{2}}(\mathbb{R})$ satisfy  that $f^\pm_n\rightarrow f^\pm$ weakly in $ H^{\frac{1}{2}}(\mathbb{R})$, and $\ell_n$ be a sequence in $(0,\infty)$ such that $\lim_{n\rightarrow\infty}\ell_n=\ell_0\in[0,\infty]$. Then the following hold for $i=1,2$:
\begin{itemize}
\item[(1)]  $\lim_{n\rightarrow \infty} \check{\sA}^{\ell_n, i}(f,f)=\check{\sA}^{\ell_0, i}(f,f)$;
\item[(2)] $\liminf_{n\rightarrow\infty}\check{\sA}^{\ell_n,i}(f_n,f_n)\geq \check{\sA}^{\ell_0,i}(f,f)$.
\end{itemize}
\end{lemma}
\begin{proof}
For convenience's sake, set for $f\in H^{\frac{1}{2}}(\mathbb{R})\times H^{\frac{1}{2}}(\mathbb{R})$,
 \[
 \begin{aligned}
 	&(f,f)_{m_{n1}}:=\check{\sA}^{\ell_n,1}(f,f),\quad (f,f)_{m_{n2}}:=\check{\sA}^{\ell_n,2}(f,f), \\
 	&(f,f)_{m_{01}}:=\check{\sA}^{\ell_0,1}(f,f),\quad (f,f)_{m_{02}}:=\check{\sA}^{\ell_0,2}(f,f).
 \end{aligned}\]
 \begin{itemize}
 \item[(1)] We first show that for $f\in H^{\frac{1}{2}}(\mathbb{R})\times H^{\frac{1}{2}}(\mathbb{R})$, 
\begin{equation}\label{EQAPM1}
(f,f)_{m_{01}}=\lim_{n\rightarrow\infty}(f,f)_{m_{n1}}.
\end{equation}
Indeed, it holds
\begin{equation}\label{EQAPM11}
\begin{split}
(f,f)_{m_{n1}}
& \lesssim \int_{\mathbb{R}\times\mathbb{R}}\frac{\left(f^+(x_1)\right)^2+\left(f^-(x_1')\right)^2}{2\ell_n\left(\cosh\left(\ell_n^{-1}(x_1-x_1')\right)+1\right)}dx_1dx_1'\\
& =\|f^+\|_{L^2(\mathbb{R})}+\|f^-\|_{L^2(\mathbb{R})},
\end{split}
\end{equation}
since we have $\int_{\mathbb{R}}\phi(x)dx=1$ for $\varphi(x):=\frac{1}{2(\cosh x+1)}$. Thus for $\ell_0\in(0,\infty]$, \eqref{EQAPM1} holds by the dominated convergence theorem.  For $\ell_0=0$, note that $\phi_n(x):=\ell_n^{-1}\phi(x/\ell_n)$ forms a class of approximations to identity.  Thus 
\[
((f^-)*\phi_n)(x_1)=\int_\mathbb{R}\frac{f^-(x_1')}{2\ell_n\left(\cosh\left(\ell_n^{-1}(x_1-x_1')\right)+1\right)}dx_1'\rightarrow f^-(x_1)
\]
in $L^2(\mathbb{R})$. It follows that 
\[\int_{\mathbb{R}\times\mathbb{R}}\frac{f^+(x_1)f^-(x_1')}{2\ell_n\left(\cosh\left(\ell_n^{-1}(x_1-x_1')\right)+1\right)}dx_1dx_1'\rightarrow\int_\mathbb{R}f^+(x_1)f^-(x_1)dx_1\]
Combining with the second line of \eqref{EQAPM11}, \eqref{EQAPM1} also holds for $\ell_0=0$.  On the other hand,
\[(f,f)_{m_{02}}=\lim_{n\rightarrow\infty}(f,f)_{m_{n2}}\]
holds for $\ell_n\rightarrow \ell_0\in [0,\infty]$ by means of \eqref{EQcosh} and the dominated convergence theorem.  
 \item[(2)] To prove the second assertion, it suffices to show for $i=1,2$, 
$(f_n-f,f)_{m_{ni}}\rightarrow 0$ as $n\rightarrow\infty$,  which lead to 
\[
(f,f)_{m_{0i}}=\lim_{n\rightarrow\infty}(f,f)_{m_{ni}}=\lim_{n\rightarrow\infty}(f_n,f)_{m_{ni}}\leq \lim_{n\rightarrow\infty} (f_n,f_n)_{m_{ni}}^{\frac{1}{2}}(f,f)_{m_{ni}}^{\frac{1}{2}}.
\]
Indeed, by \eqref{EQcosh}, we have
\[
|(f_n-f,f)_{m_{n2}}|\lesssim \left|(f^+_n-f^+,f^+)_{H^{\frac{1}{2}}(\mathbb{R})}\right|+ \left|(f^-_n-f^-,f^-)_{H^{\frac{1}{2}}(\mathbb{R})}\right|\rightarrow 0
\]
since $\gamma_\pm f_n\rightarrow \gamma_\pm f$ weakly in $ H^{\frac{1}{2}}(\mathbb{R})$.  Thus for $\ell_0\in [0,\infty]$, 
\begin{equation}\label{EQ5INF2}
\liminf_{n\rightarrow\infty}(f_n,f_n)_{m_{n2}}\geq (f,f)_{m_{02}}.
\end{equation}
On the other hand, it is obvious by definition that for $\ell_0=\infty$,
\begin{equation}\label{EQ5INF1}
\liminf_{n\rightarrow\infty}(f_n,f_n)_{m_{n1}}\geq (f,f)_{m_{01}}.
\end{equation}
So we only need to prove that $(f_n-f,f)_{m_{n1}}\rightarrow 0$ for $\ell_0\in [0,\infty)$. 
Note that
\[
\begin{split}
 \qquad (f_n-f,f)_{m_{n1}}&=\int_{\mathbb{R}\times\mathbb{R}}\frac{\left((f^+_n-f^+)(x_1)-(f^-_n-f^-)(x_1')\right)(f^+(x_1)-f^-(x_1'))}{2\ell_n\left(\cosh\left(\ell_n^{-1}(x_1-x_1')\right)+1\right)}dx_1dx_1'\\
&=: I_1^++I_1^-+I_2^++I_2^-,
\end{split}
\] 
where 
\[
I_1^\pm=\int_{\mathbb{R}\times\mathbb{R}}\frac{\left((\gamma_\pm f_n-\gamma_\pm f)(x_1)\right)\gamma_\pm f(x_1)}{2\ell_n\left(\cosh\left(\ell_n^{-1}(x_1-x_1')\right)+1\right)}dx_1dx_1'
\]
and 
\[
I_2^\pm=-\int_{\mathbb{R}\times\mathbb{R}}\frac{\left((\gamma_\pm f_n-\gamma_\pm f)(x_1)\right)\gamma_\mp f(x_1')}{2\ell_n\left(\cosh\left(\ell_n^{-1}(x_1-x_1')\right)+1\right)}dx_1dx_1'.
\]
For the terms $I_1^\pm$,  it follows from $\int_{\mathbb{R}}\phi_n(x)dx=1$ and
$\gamma_\pm f_n\rightarrow \gamma_\pm f$ weakly in $L^2(\mathbb{R})$ that
\[
I_1^\pm\lesssim \int_{\mathbb{R}}\left((\gamma_\pm f_n-\gamma_\pm f)(x_1)\right)\gamma_\pm f(x_1)dx_1\rightarrow 0.
\]
To treat the terms $I_2^\pm$, define 
\[
g_{n\mp}(x_1)=\int_\mathbb{R}\frac{\gamma_\mp f(x_1')}{2\ell_n\cosh\left(\ell_n^{-1}(x_1-x_1')\right)+1}dx_1'.
\]
For $\ell_0>0$, it follows from the Young inequality for convolutions that 
\[
\|g_{n\mp}-g_{0\mp}\|_{L^2(\mathbb{R})}\lesssim \|\gamma_\mp f\|_{L^2(\mathbb{R})}\int_\mathbb{R}\left|\frac{1}{\cosh \ell_n^{-1}x_1'+1}-\frac{1}{\cosh \ell_0^{-1}x_1'+1}\right|dx_1'.
\]
Since $\ell_n\rightarrow \ell_0$ and $\ell_0$ is finite,  the dominated convergence theorem yields
\[
\int_\mathbb{R}\left|\frac{1}{\cosh \ell_n^{-1}x_1'+1}-\frac{1}{\cosh \ell_0^{-1}x_1'+1}\right|dx_1'\rightarrow 0.
\]
Hence $g_{n\mp}\rightarrow g_{0\mp}$ in $L^2(\mathbb{R})$. 
For $\ell_0=0$, we have already shown that $g_{n\mp}=(\gamma_\mp f)*\phi_n\rightarrow \gamma_\mp f$ in $L^2(\mathbb{R})$. 
Thus $I_2^\pm\rightarrow 0$ since $\gamma_\pm f_n\rightarrow \gamma_\pm f$ weakly in $L^2(\mathbb{R})$.  Eventually \eqref{EQ5INF1} holds for $\ell_0\in [0,\infty)$.  
 \end{itemize}
 That completes the proof.
\end{proof}

Finally we consider the particular case $f^+=f^-$.  The following lemma formulates the concrete expression of $\check{\sA}^{\ell, 1}(f,f)$ and $\check{\sA}^{\ell, 2}(f,f)$.

\begin{corollary}\label{LMD3}
Let $f\in H^{\frac{1}{2}}(\mathbb{R})\times H^{\frac{1}{2}}(\mathbb{R})$ with $f^+=f^-=:\mathsf{f}$.  Then
\begin{equation}\label{eq:D31}
	\check{\sA}^{\ell, 1}(f,f)=\ell^2 \int_\bR |\hat{\mathsf{f}}(\xi)|^2|\xi|^2d\xi \int_\bR \frac{y^2dy}{\cosh y +1}\cdot \frac{1-\cos (\ell \xi y)}{(\ell \xi y)^2}
\end{equation}
and 
\begin{equation}\label{eq:D32}
\check{\sA}^{\ell, 2}(f,f)=4\ell \int_\bR |\hat{\mathsf{f}}(\xi)|^2|\xi|^2d\xi \int_\bR \frac{y^2dy}{\cosh y -1}\cdot \frac{1-\cos (\ell \xi y)}{(\ell \xi y)^2},
\end{equation}
where $\hat{\mathsf{f}}$ is the Fourier transform of $\mathsf{f}$.  
\end{corollary}
\begin{proof}
We only treat $\check{\sA}^{\ell, 1}(f,f)$ and the formulation of $\check{\sA}^{\ell, 2}(f,f)$ is analogical.  Indeed,
\[
	2\ell \check{\sA}^{\ell, 1}(f,f)=\int_{\bR\times \bR}\frac{\left( \sf(x+y)-\sf(x)\right)^2}{\cosh (\ell^{-1}y) +1} dxdy.
\]
Note that the Fourier transform of $v_y(\cdot):=\sf(\cdot +y)-\sf(\cdot)$ is $\hat{v}_y(\xi)=\hat{\sf}(\xi)(\mathrm{e}^{-i\xi y}-1)$.  It follows that
\[
\begin{aligned}
	2\ell \check{\sA}^{\ell, 1}(f,f)&=\int_{\bR}\frac{dy}{\cosh (\ell^{-1}y) +1} \int_\bR |\hat{v}_y(\xi)|^2d\xi \\
	&=\int_{\bR}\frac{dy}{\cosh (\ell^{-1}y) +1} \int_\bR |\hat{\sf}(\xi)|^2\cdot 2(1-\cos (\xi y))d\xi.
\end{aligned}
\]
By the Fubini theorem and the substitution $y':=\ell^{-1}y$,  we can eventually obtain \eqref{eq:D31}.  That completes the proof.
\end{proof}
\begin{remark}\label{RMD4}
We should point out that $\int_\bR |\hat{\mathsf{f}}(\xi)|^2|\xi|^2d\xi<\infty$ if and only if $\sf\in H^1(\bR)$.  In addition, 
\[
	\int_\bR \frac{y^2dy}{\cosh y +1}=\frac{2}{3}\pi^2,\quad \int_\bR \frac{y^2dy}{\cosh y -1}=\frac{4}{3}\pi^2.  
\]
and for $\ell \xi y\neq 0$,  
\[
	\left|\frac{1-\cos (\ell \xi y)}{(\ell \xi y)^2}\right|\leq 1,\quad \lim_{\ell \downarrow 0} \frac{1-\cos (\ell \xi y)}{(\ell \xi y)^2}=\frac{1}{2}.
\]
\end{remark}

\end{appendices}

\bibliographystyle{abbrv}
\bibliography{stiff3}

\end{document}